\documentclass[a4paper, reqno]{amsart}

\usepackage[T1]{fontenc}
\usepackage[utf8]{inputenc}
\usepackage{hyperref}
\hypersetup{%
	colorlinks=true,
	urlcolor=blue,
	citecolor=purple,
	linkcolor=purple,
}

\usepackage{booktabs}

\usepackage{amssymb,amsthm,tikz-cd,stmaryrd, mathtools}
\mathtoolsset{showonlyrefs}

\tikzcdset{arrow style=tikz, diagrams={>=to}}
\usepackage{stackengine}

\usepackage{newpxtext}
\usepackage{newpxmath}

\usepackage[english]{babel}
\usepackage{csquotes}
\usepackage{graphicx}
\usepackage{overpic}

\usepackage{enumitem}
\setlist[itemize]{label=$\diamond$}

\newcommand{\Sp}{\operatorname{Sp}}
\newcommand{\Un}{\operatorname{U}}

\newcommand{\Ham}{\operatorname{Ham}}
\newcommand{\Sym}{\operatorname{Sym}}

\newcommand{\Fix}{\operatorname{Fix}}
\newcommand{\HF}{\operatorname{HF}}
\newcommand{\HFloc}{\operatorname{HF}^{\mathrm{loc}}}
\newcommand{\CF}{\operatorname{CF}}

\newcommand{\op}{\mathrm{op}}
\newcommand{\delbar}{\overline{\partial}}

\newcommand{\loc}{\mathrm{loc}}

\newcommand{\id}{\textup{id}}

\usepackage{dsfont}
\newcommand{\bid}{\mathds{1}}

\newcommand{\ind}{\operatorname{ind}}

\newcommand{\End}{\operatorname{End}}

\newcommand{\CZ}{\operatorname{CZ}}
\newcommand{\meanCZ}{\overline{\operatorname{CZ}}}
\newcommand{\Mas}{\operatorname{Mas}}

\newcommand{\supp}{\operatorname{supp}}

\newcommand{\floor}[1]{\left\lfloor{#1}\right\rfloor}

\renewcommand{\phi}{\varphi}
\renewcommand{\epsilon}{\varepsilon}

\newcommand{\cA}{\mathcal{A}}
\newcommand{\cC}{\mathcal{C}}

\newcommand{\cG}{\mathcal{G}}
\newcommand{\cH}{\mathcal{H}}
\newcommand{\cI}{\mathcal{I}}
\newcommand{\cJ}{\mathcal{J}}

\newcommand{\cS}{\mathcal{S}}
\newcommand{\cU}{\mathcal{U}}

\newcommand{\bA}{\mathbb{A}}
\newcommand{\bC}{\mathbb{C}}
\newcommand{\bI}{\mathbb{I}}
\newcommand{\bM}{\mathbb{M}}
\newcommand{\bN}{\mathbb{N}}
\newcommand{\bO}{\mathbb{O}}
\newcommand{\bQ}{\mathbb{Q}}
\newcommand{\bR}{\mathbb{R}}
\newcommand{\bS}{\mathbb{S}}
\newcommand{\bZ}{\mathbb{Z}}

\newtheoremstyle{myprop}%
	{1em plus .25em minus .3em}
	{1em plus .25em minus .3em}
	{\itshape}
	{0em}
	{\bfseries}
	{ }
	{5pt plus 1pt minus 1pt}
	{}

\newtheoremstyle{mytheorem}%
	{1em plus .25em minus .3em}
	{1em plus .25em minus .3em}
	{\itshape}
	{0em}
	{\bfseries}
	{ }
	{5pt plus 1pt minus 1pt}
	{}

\newtheoremstyle{mydefinition}%
	{1em plus .25em minus .3em}
	{1em plus .25em minus .3em}
	{\normalfont}
	{0em}
	{\bfseries}
	{ }
	{5pt plus 1pt minus 1pt}
	{}

\newtheoremstyle{myremark}%
	{1em plus .25em minus .3em}
	{1em plus .25em minus .3em}
	{\normalfont}
	{0em}
	{\bfseries\itshape}
	{ }
	{5pt plus 1pt minus 1pt}
	{}

\theoremstyle{mydefinition}
\newtheorem{defn}{Definition}
\newtheorem*{defn*}{Definition}

\theoremstyle{myremark}
\newtheorem*{rmk}{Remark}

\theoremstyle{myprop}

\newtheorem{prop}{Proposition}[section]
\newtheorem{lemma}{Lemma}[section]
\newtheorem*{lemma*}{Lemma}

\theoremstyle{mytheorem}
\newtheorem{thm}{Theorem}
\newtheorem*{thm*}{Theorem}

\usepackage{xcolor}

\usepackage[backend=biber,%
	style=numeric,%
	sorting=nyt,%
	uniquename=init,%
	giveninits=true,%
	url=false,%
	isbn=false,%
]{biblatex}

\DeclareFieldFormat[article]{title}{\emph{#1}}
\DeclareFieldFormat[incollection]{title}{\emph{#1}}
\DeclareFieldFormat[book]{title}{\textsl{#1}}
\renewbibmacro{in:}{}
\DeclareFieldFormat{postnote}{\mknormrange{#1}}
\DeclareFieldFormat[article, inbook, incollection, inproceedings, misc, thesis, unpublished]{number}{\mkbibparens{\textbf{#1}}}
\begin{document}

\title[A Poincaré-Birkhoff theorem for AUHDs]{A Poincaré-Birkhoff Theorem for Asymptotically Unitary Hamiltonian Diffeomorphisms}
\author{Leonardo Masci}
\address{International Center for Mathematics, Southern University of Science and Technology}
\curraddr{Shenzhen, Guangdong, China}
\email{\href{mailto:masci@sustech.edu.cn}{masci@sustech.edu.cn}}
\date{February 2024}

\begin{abstract}
	\noindent We prove that an asymptotically linear Hamiltonian diffeomorphism of the standard symplectic vector space, which is non-degenerate and unitary at infinity and approaches its linear map at infinity quickly enough, has infinitely many periodic points, provided that it satisfies a natural twist condition inspired by the classical Poincaré-Birkhoff theorem.
\end{abstract}

\maketitle

\section*{Introduction}
\label{sec:introduction}

\renewcommand{\theequation}{\alph{equation}}

The Poincaré-Birkhoff theorem is a fundamental result in the field of dynamical systems, concerning the fixed points of area-preserving isotopies of the annulus. The theorem states that an area-preserving isotopy of an annulus which twists the boundaries in opposite directions must admit at least two fixed points. Its origins lie in H. Poincaré's work on celestial mechanics, appearing as a conjecture in Poincaré's last paper \cite{Poincare1912_Geometric}. Its proof was first given by G. D. Birkhoff \cite{Birkhoff1913_Proof} in 1913. This theorem has had a profound impact on the field of dynamical systems. Placed in its natural symplectic geometrical environment, it brings to light the occurrence of forced oscillations in Hamiltonian systems. This phenomenon was an important ingredient in the discovery of the field of symplectic topology, and is still at the center of many contemporary questions in Hamiltonian dynamics.

In the late sixties, V. I. Arnol'd interpreted the Poincaré-Birkhoff theorem in terms of Hamiltonian diffeomorphisms of the $2$-torus, and was led to his famous conjectures concerning fixed points of Hamiltonian diffeomorphisms on compact symplectic manifolds \cite[Appendix 9]{Arnold1989MathMethods}. About a decade later, P. Rabinowitz was the first to show the existence of periodic solutions in a certain class of Hamiltonian systems using variational calculus \cite{Rabinowitz1978_PerOrbHamSys}, a feat which at the time was not considered possible. Following a suggestion of J. Moser, in this work Rabinowitz gives recovers the Poincaré-Birkhoff theorem as a theorem on Hamiltonian systems with one degree of freedom. In this sense, Rabinowitz's results ground the Poincaré-Birkhoff theorem on the variational structure underlying Hamiltonian dynamics.

Around the same time, H. Amann and E. Zehnder were also searching for forced oscillations in Hamiltonian systems. They focused on the class of asymptotically linear Hamiltonian systems, as it was not covered by the work of Rabinowitz. They combined Amann's variational techniques with C. Conley's index theory to obtain existence and multiplicity of periodic orbits in asymptotically linear Hamiltonian systems \cite{AmannZehnder1980-FinDimRed,AmannZehnder1980-PeriodicSolutionsAsyLin}. These results were also interpreted as generalisations of the Poincaré-Birkhoff theorem by their authors. Together with Conley, it was later understood that these multiplicity statements arose from Morse-type relations between periodic orbits \cite{ConleyZehnder1984_CZ}. This led Conley and Zehnder to use a similar strategy to successfully solve the Arnol'd conjecture on symplectic tori \cite{ConleyZehnder1983_BirkhoffLewis}, and foreshadowed the existence of a kind of Morse theory for periodic orbits in Hamiltonian systems. This Morse-like theory was ultimately realized by A. Floer, and it is now known as Floer homology. See \cite{Zehnder2019-BeginSympTop} for an account of this story by Zehnder. From the late eighties to the early two-thousands, many results concerning the existence of periodic orbits in asymptotically linear Hamiltonian systems were obtained using classical variational techniques. We point to \cite[Chapter 5]{Abbondandolo2001_MorseHamilt} and the references therein for a detailed discussion.

In this paper we wish to reconnect the question of existence of periodic orbits in asymptotically linear Hamiltonian systems with a twist condition, in the spirit of the original Poincaré-Birkhoff theorem. The goal is to formulate a suitable twist condition which implies the existence of periodic orbits of arbitrarily high period.

\subsection*{Poincaré-Birkhoff in the plane}
We first consider a toy model, concerning the periodic points of certain area-preserving maps of the plane. Let $\upphi$ be a Hamiltonian diffeomorphism of $\bR^2$ with its standard symplectic form $\omega_0$ (see section \ref{sec:preliminaries}). Assume that there exists a compact set $K\subset\bR^2$ such that $\upphi|_{\bR^2\setminus K}$ is a rotation of signed angle $\theta_\infty\in\bR\setminus 2\pi\bZ$. We exclude the trivial case of $\theta_\infty\in2\pi\bZ$ since for this choice, the diffeomorphism clearly has infinitely many fixed points. We can state a kind of Poincaré-Birkhoff theorem for such a map as follows: if $\upphi$ admits a fixed point $z_0\in\Fix\upphi$ with rotation number $\theta_0\neq\theta_\infty$, then $\upphi$ admits infinitely many periodic points. To see this, fix a large invariant circle which spans a disc $D$ containing $K$, and blow up the fixed point to an invariant circle by introducing local polar coordinates centered at $z_0$. The map $\upphi|_D$ extends to an area and orientation preserving homeomorphism $f$ of the resulting annulus, which has different rotation numbers at the two boundary components. One may now apply the generalized version of the Poincaré-Birkhoff theorem by Franks \cite[Corollary 2.4]{Franks1988_RecurrFixPts} to obtain that $f$, and as a consequence $\upphi$, has infinitely many periodic points.

\subsection*{Asymptotically linear Hamiltonian diffeomorphisms}
The theorem proven in this paper is a generalization of this two-dimensional phenomenon to higher dimensions and less restrictive behaviour at infinity. Let us briefly present the class of Hamiltonian systems we intend to study.

We start by considering Hamiltonian diffeomorphisms $\upphi=\phi^1_H$ on the standard linear symplectic space $\left(\bR^{2n},\omega_0\right)$ which are generated by smooth Hamiltonians $H\in C^{\infty}(S^1\times\bR^{2n})$ that are asymptotically quadratic (definition \ref{def:asymptotically_quadratic_hamiltonians}). This means that we can express
\begin{equation}
	H(t,z)=\frac12\left\langle A(t)z,z\right\rangle+h(t,z)
\end{equation}
where $h\in C^\infty(S^1\times\bR^{2n})$ is a smooth function with sub-linear gradient, i.e. $\nabla h(t,z)=o(|z|)$ as $|z|\to\infty$, and $A\colon S^1\to\Sym(2n)$ is a 1-periodic path of symmetric matrices. The path $A$ defines a time-dependent quadratic form $Q(t,z)=\frac12\left\langle A(t)z,z\right\rangle$ which we call the \emph{quadratic Hamiltonian at infinity}. We refer to remainder $h=H-Q$ as the \emph{sub-quadratic part} of the Hamiltonian $H$.

The time-1 map $\upphi=\phi^1_H$ of the flow of an asymptotically quadratic Hamiltonian $H=Q+h$ will be called an \emph{asymptotically linear Hamiltonian diffeomorphism (ALHD)}. For example, if $\upphi$ is the map of the plane described in the previous section, then we can take $H=\frac{\theta_\infty}{2}|z|^2+h$ with $h$ compactly supported on $K$.

The time-1 map $\upphi_\infty=\phi^1_Q$ of the flow of the quadratic Hamiltonian at infinity $Q$ will be called the \emph{linear map at infinity}. It is easy to see that $\upphi_\infty$ is well defined and independent of the chosen generating Hamiltonian $H$ (see proposition \ref{prop:lin_map_infty_welldef}).

We say that an ALHD $\upphi$ is \emph{non-degenerate at infinity} if the corresponding linear map at infinity $\upphi_\infty$ does not have $1\in\bC$ as an eigenvalue. In the two-dimensional toy model, this is equivalent to assuming that the rotation angle $\theta_\infty$ is not an integer multiple of $2\pi$.

With the standing hypotheses on our set of generating Hamiltonians, ALHDs can be quite intractable, especially if we are interested in periodic points of arbitrarily high period. In order to gain control on the flow for adequately long times, we introduce \emph{rapidly asymptotically quadratic Hamiltonians}, whose sub-quadratic part decays sufficiently quickly at infinity (see definition \ref{def:rapidly_asy_quad} for the precise condition needed). If an ALHD can be generated by a rapidly asymptotically quadratic Hamiltonian, we will call it a \emph{rapidly asymptotically linear Hamiltonian diffeomorphism}. The rapidity condition implies that the sub-quadratic part is a bounded function, but not necessarily of compact support. In particular, we can include a sizable class of analytic Hamiltonians in our treatment.

Finally, we introduce a property generalizing the condition of being a rotation outside a compact set, which we imposed in the toy model. An asymptotically linear Hamiltonian diffeomorphism is said to be \emph{unitary at infinity} if its linear map at infinity is represented by a unitary matrix. For example the time-1 flow of an asymptotically quadratic Hamiltonian $H=Q+h$ with $Q$ time-independent and positive- or negative-definite is unitary at infinity.

\subsection*{A twist condition}
Given an ALHD $\upphi$, with a fixed generating Hamiltonian $H=Q+h$, the flow of the quadratic Hamiltonian at infinity $Q$ defines a path of symplectic matrices. We define the \emph{mean index at infinity} of $H$ to be the mean Conley-Zehnder index of this path (see section \ref{sub:the_conley_zehnder_index}). The mean Conley-Zehnder index is a generalization of the two-dimensional rotation number to paths of symplectic matrices in arbitrary dimension. In the two-dimensional toy model, if we use a generating Hamiltonian of the form $H=\frac{\theta_\infty}{2}|z|^2+h$ as in the previous example, then its mean index at infinity is $2\theta_\infty$. With the twist condition in the Poincaré-Birkhoff theorem in mind, we say that a fixed point of $\upphi$ is a \emph{twist fixed point} if its mean Conley-Zehnder index calculated with respect to $H$ is \emph{different} than the mean index at infinity of $H$. Even though the twist condition is formulated in terms of a generating Hamiltonian, it does not depend on the specific generating Hamiltonian chosen: it is a property of the fixed point of the ALHD (see lemma \ref{lem:twisted_fixed_pt_well_def}).

\begin{thm}\label{thm:main_intro}
	Let $\upphi$ be a rapidly asymptotically linear Hamiltonian diffeomorphism, unitary and non-degenerate at infinity. If $\phi$ has an isolated twist fixed point which is homologically visible, then $\upphi$ has infinitely many periodic points.
\end{thm}
\begin{rmk}
	More can be said on the period of the periodic points obtained in the theorem. Let $\upphi$ be a diffeomorphism and $z$ a periodic point of $\upphi$. The \emph{primitive period} of $z$ is the smallest integer $p>0$ such that $\upphi^p(z)=z$. In the proof, it is shown that an ALHD as in Theorem \ref{thm:main_intro} has infinitely many fixed points or infinitely many periodic points with increasing primitive period.
\end{rmk}
Homological visibility is a technical condition, which can be considered the natural Floer-theoretical generalization of non-degeneracy.
It is formulated in terms of the \emph{local Floer homology} groups of the fixed point (see section \ref{sub:local_floer_homology}). These are interesting invariants, which measure the ``homological weight'' of the fixed point by packaging its non-degenerate bifurcations into a sequence of Floer-type homology groups. Homological visibility can be restated equivalently using generating functions \cite[Chapter 9,\S 48]{Arnold1989MathMethods}: given a generating function for the diffeomorphism in a neighborhood of the isolated fixed point, we assume that the local Morse homology \cite{HeinHryniewiczMacarini2019_HMloc} of its critical point corresponding to the fixed point is non-trivial. For example, if the Lefschetz index of the isolated fixed point is non-zero, then it is homologically visible (but not vice versa).

\subsection*{On a conjecture by Abbondandolo} Asymptotically linear Hamiltonian diffeomorphisms which are non-degenerate at infinity always have at least one fixed point, as was shown by Amann, Conley and Zehnder \cite{AmannZehnder1980-FinDimRed,AmannZehnder1980-PeriodicSolutionsAsyLin,ConleyZehnder1984_CZ}. In \cite[pg. 130]{Abbondandolo2001_MorseHamilt} A. Abbondandolo describes a richer picture of the dynamics of asymptotically linear Hamiltonian systems: he conjectures that an asymptotically linear Hamiltonian system should have one or infinitely many periodic orbits. In other words, if the Hamiltonian system has an additional, ``unnecessary'' periodic orbit other than the one found by Amann, Conley and Zehnder, then it should have infinitely many periodic orbits. This conjecture is the analogue of the Hofer-Zehnder conjecture on Hamiltonian diffeomorphisms of compact symplectic manifolds \cite[p. 263]{HoferZehnder2011_SymplecticInvariants} in the context of asymptotically linear Hamiltonian systems.

Abbondandolo's conjecture in this generality is wide open and, to the best of the author's knowledge, beyond the reach of the current mathematical technology. We consider instead the following weaker, \emph{homological} version, following E. Shelukhin's interpretation of the Hofer-Zehnder conjecture (see \cite{Shelukhin2022_OnHZ} and the references therein).

As mentioned previously, an isolated fixed point $z_0$ of a Hamiltonian diffeomorphism $\upphi=\phi^1_H$ carries a local Floer homology, denoted here by $\HFloc_*(H,z_0)$. For an ALHD $\upphi$ define
\begin{equation}
	\operatorname{N}\left(\upphi\right)=\sum_{z\in\Fix\upphi}\dim\HFloc\left(H,z\right)\in\bN\cup\{\infty\}
\end{equation}
where $H$ is some generating Hamiltonian and the local Floer homology is seen as an ungraded $\bZ/2$-vector space. It is easy to see that the number $\operatorname N(\upphi)$ does not depend on the generating Hamiltonian chosen. For example, if the fixed points of $\upphi$ are all non-degenerate, $\operatorname N(\upphi)$ is just the number of fixed points. In particular, for a non-degenerate linear symplectic map one has $\operatorname N(\upphi)=1$. The homological version of Abbondandolo's conjecture is thus that if $\upphi$ is an ALHD with $\operatorname N(\upphi)>1$ then $\upphi$ has infinitely many periodic points.

The result in this paper can be interpreted as a first partial step towards a solution of the homological version of this conjecture, under strong restrictions on the type of system at infinity and additional hypotheses on the index of the unnecessary orbit. In order to attack the full homological version of the conjecture, perhaps techniques akin to the ones of Shelukhin \cite{Shelukhin2022_OnHZ} must be adapted to asymptotically linear Hamiltonian systems.

\subsection*{Comparisons and contrasts} Other kinds of Poincaré-Birkhoff type theorems have appeared recently in the literature. We would like to compare and contrast the present work with some of these results. The author does not claim any bibliographical completeness, and apologizes for any omissions.

In the paper by B. Gürel \cite{Gurel_PJM2014}, Hamiltonian diffeomorphism which are equal to an autonomous hyperbolic linear symplectic diffeomorphism outside a compact set are studied. There it is proven that if the Hamiltonian diffeomorphism admits an isolated, homologically visible fixed point whose mean Conley-Zehnder index is not zero, then there are infinitely many periodic points. Since an asymptotically hyperbolic and autonomous Hamiltonian system always has zero mean index at infinity, assuming the existence of a fixed point with non-vanishing mean Conley-Zehnder index is a twist condition in the sense we introduced in this paper. The proof schema of the main theorem of this paper builds upon the proof found in Gürel's paper. The difference is that we admit a larger class of Hamiltonians, and that that asymptotically unitary Hamiltonian diffeomorphisms behave in a very different manner than asymptotically hyperbolic ones. This leads to complications involving the existence and asymptotic behaviour of continuation morphisms on Floer homology, which required original ideas to be overcome.

Another interesting development is found in A. Moreno and O. Van Koert's paper \cite{MorenoVanKoert2022_PoincBirk}, where a kind of Poincaré-Birkhoff theorem is proven for certain ``twist'' Hamiltonian diffeomorphisms in the completion of Liouville domains which have infinite dimensional symplectic homology. Recently, A. Limoge and Moreno \cite{limogemoreno2025poincarebirkhofftheoremc0hamiltonianmaps} extended the results in the aforementioned paper to a more general notion of twist Hamiltonians, and also provide a relative version for Lagrangian chords. The present work compares to these sorts of Poincaré-Birkhoff-like theorems, because one could interpret our class of Hamiltonian systems as such kind of ``twist'' Hamiltonian diffeomorphims in the completion of an ellipsoid in $\bR^{2n}$. The difference is that the symplectic homology of the ball vanishes, so our techniques cover an orthogonal case.

\subsection*{Structure of the paper} In section \ref{sec:preliminaries} we introduce ALHDs and their elementary properties, and discuss briefly the Conley-Zehnder index. In section \ref{sec:modifying_linear_and_asymptotically_linear_hamiltonian_systems} we present some techniques to modify linear and asymptotically linear Hamiltonian systems which are fundamental for the technical side of the proof of the main theorem. In section \ref{sec:the_poincare_birkhoff_theorem_for_rapidly_asymptotically_unitary_hamiltonian_systems} we prove the Poincaré-Birkhoff theorem for rapidly asymptotically unitary Hamiltonian diffeomorphisms, conditional to an auxiliary proposition which summarizes the technical constructions of the paper. Section \ref{sec:proof_of_the_auxiliary_proposition} contains the proof the auxiliary proposition. In section \ref{sec:floer_homology_of_asymptotically_linear_hamiltonian_systems} we briefly discuss Floer homology for asymptotically linear Hamiltonian systems, also in its filtered and local versions.

\renewcommand{\theequation}{\arabic{equation}}
\numberwithin{equation}{section}

\section{Preliminaries}
\label{sec:preliminaries}

\subsection{Generalities on Hamiltonian diffeomorphisms}
\label{sub:generalities_on_hamiltonian_diffeomorphisms}

Denote by $S^1=\bR/\bZ$ and by $\langle\cdot,\cdot\rangle$ the standard Euclidean inner product on $\bR^m$. Consider $\bR^{2n}$ with its standard symplectic structure, which is given by
\begin{equation}\label{eq:omega0_eucl_inner_prod}
	\omega_0(v,w)=\left\langle J_0v,w\right\rangle,\quad J_0=\begin{pmatrix}
		\bO_n & -\bI_n\\
		\bI_n & \bO_n
	\end{pmatrix}
\end{equation}
where the block matrix form arises from splitting $\bR^{2n}=\bR^n\oplus\bR^n$.
We fix the following primitive of $\omega_0$, which is sometimes called the \emph{radial primitive}:
\begin{equation}\label{eq:radial_primitive_of_omega_0}
	\lambda_0(z)v=\frac12\left\langle z,J_0v\right\rangle.
\end{equation}

For a Hamiltonian function $H\colon S^1\times\bR^{2n}\to\bR$ we usually denote $H(t,z)=H_t(z)$. Define the Hamiltonian vector field $X_H$ of the Hamiltonian $H$ by the identity
\begin{equation}
	i_{X_H}\omega_0=dH
\end{equation}
or, in coordinates, $X_H(t,z)=-J_0\nabla H_t(z)$, where the $\nabla$ denotes the gradient in the $z$-coordinates only.
\begin{rmk}
	Since the Hamiltonian vector field of a time-dependent Hamiltonian $H$ is non-autonomous, the flow it generates is a non-autonomous flow $(t_0,t)\mapsto\phi^{t_0,t}_H$ \cite[Definition 2.2.23]{AbrahamMarsden1978_Foundations}. We denote $\phi^{0,t}_H=\phi^t_H$ throughout the rest of the paper.
\end{rmk}

A diffeomorphism $\upphi\colon\bR^{2n}\to\bR^{2n}$ is said to be Hamiltonian when $\upphi=\phi^1_H$ for some $H\in C^\infty\left(S^1\times\bR^{2n}\right)$. We denote the space of Hamiltonian diffeomorphisms by $\Ham$.

\subsubsection{Fixed points and action functional}
\label{ssub:fixed_points_and_action_functional}

Denote by
\begin{equation}\label{eq:fixed_pts}
	\Fix\upphi=\left\{z\in\bR^{2n}:\upphi(z)=z\right\}
\end{equation}
the set of fixed points of a diffeomorphism $\upphi$. Fixed points of Hamiltonian diffeomorphisms have an \emph{action}, which is determined by a choice of generating Hamiltonian. Let $\upphi\in\Ham$, $H$ a generating Hamiltonian, and $z_0\in\Fix\upphi$. Define $\gamma\colon\bR/\bZ\to\bR^{2n}$ by $\gamma(t)=\phi^t_H(z_0)$. The \emph{action} of $z_0$ is defined to be
\begin{equation}\label{eq:action_of_fixed_pt}
	\cA_H\left(z_0\right)=\int_0^1\frac12\left\langle\gamma(t),J_0\dot\gamma(t)\right\rangle-H_t\left(\gamma(t)\right)dt=\int_{S^1}\gamma^*\lambda_0-H_t\circ\gamma dt.
\end{equation}
We can view $\cA_H$ as a functional on the space of loops of $\bR^{2n}$. The critical points of this functional are exactly the 1-periodic orbits of $X_H$.

\subsubsection{Combining and iterating Hamiltonians}
\label{ssub:combining_and_iterating_hamiltonians}

Choose a $\tau\in C^\infty([0,1])$ such that $\tau(t)=0$ near $0$, $\tau(t)=1$ near 1 and $0\leq \tau'(t)<2$. If $F,G$ are two Hamiltonians, define
\begin{equation}\label{eq:composition_of_hamilts}
	\begin{split}
		\left(F\# G\right)_t(z)&=F_t(z)+G_t\left(\left(\phi^t_F\right)^{-1}(z)\right),\\
		\left(F\wedge G\right)_t(z)&=
		\begin{cases}
			2\tau'(2t)G_{\tau(2t)}(z), &t\in\left[0,\frac12\right]\\
			2\tau'(2t-1)F_{\tau(2t-1)}(z), &t\in\left[\frac12,1\right].
		\end{cases}
	\end{split}
\end{equation}
We call $F\#G$ the \emph{composition} of $F$ after $G$ and $F\wedge G$ the \emph{conactenation} of $F$ after $G$.
Denote by $\overline{F}_t(z)=-F_t\left(\phi^t_F(z)\right)$. These Hamiltonians generate the following flows:
\begin{equation}\label{eq:comp_cat_inv_hams_flows}
	\phi^t_{\overline{F}}=\left(\phi^t_F\right)^{-1},\quad\phi^t_{F\#G}=\phi^t_F\circ\phi^t_G,\quad\phi^t_{F\wedge G}=
	\begin{cases}
		\phi^{\tau(2t)}_G, & t\in\left[0,\frac12\right]\\
		\phi^{\tau(2t-1)}_F\circ\phi^1_G,&t\in\left[\frac12,1\right].
	\end{cases}
\end{equation}
In particular $\phi^1_{F\#G}=\phi^1_F\circ\phi^1_G=\phi^1_{F\wedge G}$ and $\phi^1_{\overline F}=\left(\phi^1_F\right)^{-1}$. The flows $\phi^\cdot_{F\wedge G}$ and $\phi^\cdot_{F\#G}$ seen as paths in $\Ham$ are homotopic with fixed end-points.

These calculations show that if $\upphi\in\Ham$ then $\upphi^{-1}\in\Ham$ and if further $\uppsi\in\Ham$ then $\upphi\circ\uppsi\in\Ham$. The identity $\id_{\bR^{2n}}$ is generated by the zero Hamiltonian. Therefore $\Ham$ is a group under composition. In particular, if $\upphi\in\Ham$, then $\upphi^k\in\Ham$ for all $k\in\bZ$. In this case though, given a generating Hamiltonian $H$ for $\upphi$, there is a choice of Hamiltonian generating $\upphi^k$ which is simpler than the $k$-fold composition or $k$-fold concatenation of $H$. For $k\in\bZ$, we define
\begin{equation}\label{eq:def_of_iterated_Hamiltonian}
	H^{\times k}(t,z)=kH(kt,z).
\end{equation}
It is easy to see that $\phi^1_{H^{\times k}}=\phi^k_H=\upphi^k$ using the 1-periodicity of the coefficients of the Hamiltonian. We will call $H^{\times k}$ the $k$-\emph{fold iteration} of the Hamiltonian $H$.

\subsection{Asymptotically Linear Hamiltonian Systems}
\label{sub:asymptotically_linear_hamiltonian_systems}

Let $f$ be a $C^1$ function from $[0,1]\times\bR^{2n}$ to $\bR$. Define the ``tail functions'' $\sigma^f_0,\sigma^f_1\colon[0,\infty)\to[0,\infty]$ of $f$ by
\begin{align}\label{eq:def_of_decay_moduli}
	\sigma^f_1(R)=\sup_{t\in[0,1],\ |z|\geq R}\frac{\left|\nabla f_t(z)\right|}{|z|},\\
	\sigma^f_0(R)=\sup_{t\in[0,1],\ |z|\geq R}\frac{\left| f_t(z)\right|}{|z|^2}.
\end{align}
The function $\sigma^f_1$ is a measure of the ``linear growth'' of the gradient, while $\sigma^f_0$ a measure of the ``quadratic growth'' of the function:
\begin{equation}\label{eq:decay_moduli_growth_moduli}
	\left|\nabla f_t(z)\right|\leq\sigma^f_1\left(|z|\right)|z|,\quad \left|f_t(z)\right|\leq\sigma^f_0\left(|z|\right)|z|^2.
\end{equation}

\begin{defn}\label{def:asymptotically_quadratic_hamiltonians}
	Consider a smooth Hamiltonian function $H\colon S^1\times\bR^{2n}\to\bR$. We say that $H$ is \emph{asymptotically quadratic} if there exists a smooth path $A\colon S^1\to\Sym(2n)$ of symmetric matrices such that setting $Q_t(z)=\frac12\langle A_tz,z\rangle$ and
	\begin{equation}\label{eq:def_of_perturb}
		h_t(z)=H_t(z)-Q_t(z)
	\end{equation}
	we have that
	\begin{enumerate}
		\item\label{item:asyquad_def_asspts_hessian_bded} The Hessian of $h$ is bounded:
		\begin{equation}\label{eq:hessian_bded}
			c_2=\left\|\nabla^2h\right\|_{L^\infty\left([0,1]\times\bR^{2n}\right)}<\infty.
		\end{equation}
		\item\label{item:asyquad_def_asspts_sublin_grad} The gradient of $h$ has sub-linear growth:
			\begin{equation}\label{eq:asyquad_def_asspt_sublin_grad}
				\sigma^h_1(R)\to 0\quad\text{as}\quad R\to\infty.
			\end{equation}
	\end{enumerate}
	We refer to the time-dependent quadratic form $Q_t(z)=\frac12\langle A_tz,z\rangle$ as the \emph{quadratic Hamiltonian at infinity} of $H$. We call the remainder $h$ the \emph{sub-quadratic part} of $H$.
\end{defn}

Notice that equation \eqref{eq:hessian_bded} implies that the Hessian of $H$ is bounded and gives the automatic estimates
\begin{equation}\label{eq:decay_moduli_decayish}
	\begin{multlined}
		\left\|\nabla^2H\right\|_{L^\infty}\leq c_2+\|A\|_{L^\infty}=C_2,\\
		\sigma^h_1(R)\leq \frac{\max_t|\nabla h_t(0)|}{R}+c_2,\\
		\sigma^h_0(R)\leq \frac{\max_t|h_t(0)|}{R^2}+\sigma^h_1(R).
	\end{multlined}
\end{equation}
In particular, equation \eqref{eq:asyquad_def_asspt_sublin_grad} implies that also $\sigma_0^h(R)\to 0$ as $R\to\infty$.
For the sake of estimation, we can and will replace the tail functions $\sigma^h_1$ and $\sigma^h_0$ with their smallest majorant which is non-increasing: we shall denote for $j=0,1$
\begin{equation}\label{eq:def_of_non_incr_majorant_tail}
	\bar\sigma^h_j(R)=\sup_{r\geq R}\sigma^h_j(r)\geq\sigma_j^h(R).
\end{equation}
Clearly if $R'\geq R$ then $\bar\sigma^h_j(R')\leq\bar\sigma^h_j(R)$ and
\begin{equation}
	\lim_{R\to\infty}\bar\sigma_j^h(R)=\limsup_{R\to\infty}\sigma_j^h(R)=0
\end{equation}
because the limit of $\sigma_j^h(R)$ as $R\to\infty$ exists and is zero.

\begin{lemma}\label{lem:asyquad_welldef}
	If $H$ is an asymptotically quadratic Hamiltonian, then its quadratic Hamiltonian at infinity $Q$ and its sub-quadratic part $h$ are uniquely defined.
\end{lemma}
\begin{proof}
	Assume by contradiction that $H=Q+h=Q'+h'$, where $Q_t(z)=\frac12\left\langle A_tz,z\right\rangle$, $Q'_t(z)=\frac12\left\langle A'_tz,z\right\rangle$, $\sigma^{h}_1$ and $\sigma^{h'}_1$ go to $0$ as $R\to\infty$ but $A\neq A'$. Since $A\neq A'$, there is some $a>0$ such that $|(A-A')z|\geq a|z|$ for all $z\in\bR^{2n}$. Writing $h'=Q-Q'+h$,
	\begin{equation}
		\sigma^{h'}_1(|z|)\geq\frac{\left|\nabla h'_t(z)\right|}{|z|}=\frac{\left|\left(A_t-A'_t\right)z-\nabla h_t(z)\right|}{|z|}\geq a-\sigma^h_1(|z|)
	\end{equation}
	Passing to the limit $|z|\to\infty$ we contradict the sub-linear growth bound on $\nabla h'$. We conclude that $A'=A$, and therefore also $h=h'$.
\end{proof}

\begin{defn}\label{def:ALHD}
	Let $\upphi\in\Ham$ be a Hamiltonian diffeomorphism. If there exists an asymptotically quadratic Hamiltonian $H$ such that $\upphi=\phi^1_H$, we say $\upphi$ is a \emph{asymptotically linear Hamiltonian diffeomorphism} (ALHD), and we call $H$ the \emph{generating Hamiltonian} of $\upphi$.
\end{defn}
If $\upphi,\uppsi$ are ALHDs and $F,G$ are asymptotically quadratic Hamiltonians such that $\phi^1_F=\upphi$ and $\phi^1_G=\uppsi$, then $F\wedge G$ is asymptotically quadratic (inspect equation \eqref{eq:composition_of_hamilts}) and $\upphi\circ\uppsi=\phi^1_{F\wedge G}$. Therefore also $\upphi\circ\uppsi$ is an ALHD. If $\upphi$ is an ALHD, then $\upphi^k$ is also an ALHD for any fixed $k\in\bZ$, since if an asymptotically quadratic $H$ generates $\upphi$, then $H^{\times k}$ is also asymptotically quadratic and generates $\upphi^k$.

If $H=Q+h$ is asymptotically quadratic, then $\phi^t_Q$ is a linear symplectomorphism for all $t\in\bR$. We are thus led to the following
\begin{defn}\label{def:lin_map_at_infty}
	Let $\upphi$ be an ALHD and $H=Q+h$ be a generating asymptotically quadratic Hamiltonian for $\upphi$. The linear symplectomorphism $\upphi_\infty=\phi^1_Q\in\Sp(2n)$ is called the \emph{linear map at infinity}.
\end{defn}
\begin{prop}\label{prop:lin_map_infty_welldef}
	Let $\upphi$ be an ALHD. The linear map at infinity $\upphi_\infty$ of $\upphi$ does not depend on the chosen generating asymptotically quadratic Hamiltonian.
\end{prop}
\begin{proof}
	Let $H=Q+h$ be some asymptotically quadratic generating Hamiltonian for $\upphi$. Let's first show that
	\begin{equation}\label{eq:linmap_infty_sublin_close_to_ALHD}
		\left|\phi^1_H(z)-\phi^1_Qz\right|=o\left(|z|\right)\quad\text{as}\quad|z|\to\infty.
	\end{equation}
	Let $\tau\in[0,1]$ and $z\in\bR^{2n}$ be fixed. Set $x_t=\phi^t_Qz$ and $y_t=\phi^t_H(z)$.
	\begin{equation}\label{eq:1per_estimate_dist_of_flows-1}
		\begin{split}
			\left|y_\tau-x_\tau\right|&=\left|\int_0^\tau J_0A_tx_t-J_0\nabla H_t(y_t)dt\right|=\\
			&=\left|\int_0^\tau A_tx_t-\nabla H_t(x_t)+\nabla H_t(x_t)-\nabla H_t(y_t)dt\right|\leq\\
			&\leq \int_0^\tau\left|\nabla h_t(x_t)\right|dt+\left\|\nabla^2H\right\|_{L^\infty}\int_0^\tau\left|x_t-y_t\right|dt
		\end{split}
	\end{equation}
	We want to estimate the first term.
	Since $x_t=\phi^t_Qz$ is a linear flow, there exist constants $b,a>0$ such that
	\begin{equation}
		e^{-bt}|z|\leq |x_t|\leq e^{at}|z|\quad\forall t\in[0,1].
	\end{equation}
	Using this and the monotonicity of $\bar\sigma^h_1$ (see \eqref{eq:def_of_non_incr_majorant_tail}), we estimate
	\begin{equation}
		\begin{split}
			\int_0^\tau\left|\nabla h_t(x_t)\right|dt&\leq\int_0^\tau \sigma^h_1\left(\left|x_t\right|\right)\left|x_t\right|dt\leq \tau e^{a\tau}|z|\max_{t\in[0,\tau]}\bar\sigma^h_1\left(e^{-bt}|z|\right)\leq\\
			&\leq\tau e^{a\tau}\bar\sigma^h_1\left(e^{-b\tau}|z|\right)|z|
		\end{split}
	\end{equation}
	We found that
	\begin{equation}\label{eq:1per_estimate_dist_of_flows-2}
		\left|y_\tau-x_\tau\right|\leq \tau e^{a\tau}\bar\sigma^h_1\left(e^{-b\tau}|z|\right)|z|+\left\|\nabla^2H\right\|_{L^\infty}\int_0^\tau\left|y_t-x_t\right|dt
	\end{equation}
	which, using Grönwall's lemma \cite[Chapter 2, Lemma 2]{AbrahamMarsden1978_Foundations} at $\tau=1$ gives the estimate
	\begin{equation}\label{eq:1per_estimate_dist_of_flows-final}
		\left|\phi^1_H(z)-\phi^1_Qz\right|\leq e^{C_2+a}\bar\sigma^h_1\left(e^{-b}|z|\right)|z|,\quad C_2=\left\|\nabla^2H\right\|_{L^\infty}.
	\end{equation}
	\begin{rmk}
		We can take $a=\|A\|_{L^\infty}$ and then $C_2+a\leq 2C_2$. In the case of unitary system at infinity, we can take $a=b=0$.
	\end{rmk}
	The inequality \eqref{eq:1per_estimate_dist_of_flows-final} shows \eqref{eq:linmap_infty_sublin_close_to_ALHD}. Now, if we have $\upphi=\phi^1_H=\phi^1_{H'}$ for asymptotically quadratic Hamiltonians
	\begin{equation}
		H=Q+h,\quad H'=Q'+h',
	\end{equation}
	we obtain that
	\begin{equation}
		\left|\phi^1_Q(z)-\phi^1_{Q'}(z)\right|\leq\left|\phi^1_H(z)-\phi^1_Q(z)\right|+\left|\phi^1_{H'}(z)-\phi^1_{Q'}(z)\right|=o(|z|)\quad \text{as}\quad |z|\to\infty.
	\end{equation}
	But this is possible if and only if the two linear maps coincide.
\end{proof}
\begin{defn}\label{def:asy_quad_ndg_at_infty}
	Let $H=Q+h$ be an asymptotically quadratic Hamiltonian. We say that $H$ is \emph{non-resonant at infinity} if
	\begin{equation}\label{eq:def_of_nonres_at_infty}
		\det\left(\phi^1_Q-\bI\right)\neq 0.
	\end{equation}
	Let $\upphi$ be an ALHD and $\upphi_\infty$ its linear map at infinity. We say that $\upphi$ is \emph{non-degenerate at infinity} if
	\begin{equation}\label{eq:def_of_ndg_at_infty}
		\det\left(\upphi_\infty-\bI\right)\neq 0.
	\end{equation}
\end{defn}
Proposition \ref{prop:lin_map_infty_welldef} implies that an ALHD is non-degenerate at infinity if and only if it can be generated by an asymptotically quadratic Hamiltonian which is non-resonant at infinity, and then all its asymptotically quadratic generating Hamiltonians will be non-resonant at infinity.

\subsection{Fixed points of ALHDs}
\label{sub:fixed_points_of_alhds}

In this section we prove a first important consequence of non-degeneracy at infinity.
\begin{lemma}\label{lem:1-per_orbits_of_nondeg_at_infty_asyquad_are_unif_Linfty_bounded}
	Let $\upphi$ be an ALHD which is non-degenerate at infinity. There exists an $R_1>0$ such that $\Fix\upphi\subset B^{2n}_{R_1}(0)$.
\end{lemma}
\begin{proof}
	Let $\upphi_\infty$ be the linear map at infinity for $\upphi$. Since $\upphi$ is non-degenerate at infinity, \eqref{eq:def_of_ndg_at_infty} implies that
	\begin{equation}\label{eq:nonres_at_infty_fund_estimate}
		\left|\upphi_\infty z-z\right|\geq \nu_\infty|z|,\quad \nu_\infty=\frac{1}{\left|\left(\upphi_\infty-\bid\right)^{-1}\right|}>0.
	\end{equation}
	Let $H=Q+h$ be an asymptotically quadratic generating Hamiltonian for $\upphi$. Recall the estimate \eqref{eq:1per_estimate_dist_of_flows-final}. Putting it together with \eqref{eq:nonres_at_infty_fund_estimate}, we see
	\begin{equation}\label{eq:1per_estimate_below}
		\begin{split}
			\left|\phi^1_H(z)-z\right|&\geq\left|\phi^1_Qz-z\right|-\left|\phi^1_H(z)-\phi^1_Qz\right|\geq\\
			&\geq\left[\nu_\infty-e^{C_2+a}\bar\sigma^h_1\left(e^{-b}|z|\right)\right]|z|=\nu(|z|)|z|.
		\end{split}
	\end{equation}
	Notice that $\nu(|z|)\to \nu_\infty>0$ monotonically non-decreasing as $|z|\to\infty$.
	Since $\bar\sigma^h_1(R)\to 0$ as $R\to\infty$, we can define
	\begin{equation}\label{eq:fund_radius}
		R_1=\max\left\{R:\nu\left(R\right)\leq 0\right\}=\max\left\{R:\bar\sigma^h_1\left(e^{-b}R\right)\geq e^{-(C_2+a)}\nu_\infty\right\}.
	\end{equation}
	If $|z|>R_1$ then $\nu(|z|)>0$ so
	\begin{equation}
		\left|\phi^1_H(z)-z\right|\geq\nu(|z|)|z|>0.
	\end{equation}
	In particular any fixed point of $\phi^1_H$ has norm less than $R_1$.
\end{proof}
\begin{rmk}
	When the system at infinity is unitary, the bound on the fixed points simplifies slightly, because the linear flow at infinity preserves the norm:
	\begin{equation}\label{eq:fund_radius_unitary}
		R_1=\max\left\{R:\bar\sigma^h_1\left(R\right)\geq e^{-C_2}\nu_\infty\right\}.
	\end{equation}
\end{rmk}
\begin{defn}\label{def:resonance_proximity_cst}
	If $\upphi$ is an ALHD which is non-degenerate at infinity, we will call
	\begin{equation}\label{eq:resonance_proximity_cst}
		\nu_\infty=\left|\left(\upphi_\infty-\bid\right)^{-1}\right|^{-1}>0
	\end{equation}
	the \emph{resonance proximity constant}.
\end{defn}

\subsection{The Conley-Zehnder index}
\label{sub:the_conley_zehnder_index}

Denote by $\Sp(2n)$ the group of symplectic $2n\times 2n$-matrices, and by $\operatorname{SP}(2n)$ the space of continuous paths $M\colon[0,1]\to\Sp(2n)$ such that $M(0)=\bI$. If $\det(M(1)-\bI)\neq 0$ we say that $M$ is a \emph{non-degenerate path}. The set of non-degenerate paths is denoted by $\operatorname{SP}^*(2n)$. This is an open dense subset of $\operatorname{SP}(2n)$.

The Conley-Zehnder index \cite{GelfandLidskii1955_Stability,YakubovichStarzhinskii1975_LODE,ConleyZehnder1984_CZ,SalamonZehnder1992_MorseTheoryMaslovIndex,Abbondandolo2001_MorseHamilt, Long2002_IndexTheorySp,Gutt2012_CZ} is an integer associated to non-degenerate paths of symplectic matrices, which we denote by
\begin{equation}
	\begin{split}
		\CZ\colon\operatorname{SP}^*(2n)&\to\bZ\\
		M&\mapsto\CZ(M_t).
	\end{split}
\end{equation}
Two non-degenerate paths $M$ and $M'$ are homotopic through non-degenerate paths if and only if $\CZ(M_t)=\CZ(M'_t)$ \cite[Theorem 3.3(ii)]{SalamonZehnder1992_MorseTheoryMaslovIndex}.

In principle, the Conley-Zehnder index is only defined for non-degenerate paths. In practice, it important to make sense of the Conley-Zehnder index of a \emph{degenerate} path, i.e. one whose endpoint does contain 1 in its spectrum. Following \cite[Section 1.3.7]{Abbondandolo2001_MorseHamilt} or \cite[Section 5.4]{Long2002_IndexTheorySp} we can extend $\CZ$ to the whole space of paths based at the identity, setting
\begin{equation}\label{eq:lsc_cz}
	\CZ\colon\operatorname{SP}(2n)\to\bZ,\quad \CZ(M_t)=\liminf_{\stackrel{M^*\in\operatorname{SP}^*(2n),}{M^*\to M}}\CZ\left(M^*_t\right).
\end{equation}

The Conley-Zehnder index of a path of symplectic matrices is a symplectic invariant of the path \cite[Theorem 1.3.11]{Abbondandolo2001_MorseHamilt}. Indeed, if $M\in\operatorname{SP}(2n)$ and $\Phi\in\Sp(2n)$ is a linear symplectomorphism, then
\begin{equation}\label{eq:CZ_sympl_invariant}
	\CZ\left(M_t\right)=\CZ\left(\Phi^{-1}M_t\Phi\right).
\end{equation}
The parity of the Conley-Zehnder index of a path $M\in\operatorname{SP}(2n)$ depends only on the end-point $M_1$, as can be seen using the following formula, shown in \cite[Theorem 3.3 (iii)]{SalamonZehnder1992_MorseTheoryMaslovIndex}
\begin{equation}\label{eq:parity_CZ}
	\operatorname{sign}\det\left(M_1-\bI\right)=\left(-1\right)^{\CZ\left(M_t\right)-n}.
\end{equation}
Loops of symplectic matrices act on paths of symplectic matrices by composition. Let $M\in\operatorname{SP}(2n)$, and consider a loop $\Phi\colon S^1\to\Sp(2n)$ based at $\Phi(0)=\bI$. Then  $t\mapsto\Phi_tM_t$ is a path in $\operatorname{SP}(2n)$ and the following ``loop composition formula'', which can be found e.g. in \cite[Section 2.4]{Salamon1999_LecturesFH}, holds:
\begin{equation}\label{eq:CZ_loop_composition_formula}
	\CZ\left(\Phi_tM_t\right)=\CZ\left(M_t\right)+2\Mas\left(\Phi_t\right),
\end{equation}
where $\Mas\colon\pi_1\left(\Sp(2n)\right)\to\bZ$ denotes the Maslov index of a loop of linear symplectomorphisms \cite[Theorem 2.2.12]{McDuffSalamon2017_Intro}.

A path $M\in\operatorname{SP}(2n)$ can be extended periodically to a path defined on  the whole $\bR$ by setting
\begin{equation}
	M(t+1)=M(t)M(1)\quad\forall t\in\bR.
\end{equation}
For $k\in\bZ$, we define the $k$-fold iterate $M^{\times k}\in\operatorname{SP}(2n)$ of the path $M$ by extending it periodically and then setting $M^{\times k}(t)=M(kt)$ for $t\in[0,1]$. Notice that by construction $M^{\times k}(1)=M(1)^k$. We are interested in the Conley-Zehnder index of the iterated path, because we will need to understand the index of iterated fixed points. It is easy to see \cite[Section 1.4.1]{Abbondandolo2001_MorseHamilt} that
\begin{equation}
	k\mapsto \CZ\left(M^{\times k}_t\right)
\end{equation}
grows at most linearly as $k\to\infty$. Define the \emph{mean Conley-Zehnder index} of the path $M$ to be the rate of growth of its index under iteration:
\begin{equation}
	\meanCZ\left(M_t\right)=\lim_{k\to\infty}\frac{\CZ\left(M^{\times k}_t\right)}{k}.
\end{equation}
We have the following useful bound on the Conley-Zehnder index of an iterated path in terms of the mean index \cite[Theorem 1.4.3]{Abbondandolo2001_MorseHamilt}:
\begin{equation}\label{eq:iterated_CZ_meanCZ_bound}
	\left|\CZ\left(M^{\times k}_t\right)-k\meanCZ\left(M_t\right)\right|\leq n.
\end{equation}
The inequality is strict when the $k$th iterate is non-degenerate.

\subsection{A twist condition}
\label{sub:a_twist_condition}

Linearising the flow, we can associate a Conley-Zehnder index to a 1-periodic orbit of a Hamiltonian vector field. Let $H\in C^\infty(S^1\times\bR^{2n})$ be any Hamiltonian and $z_0\in\Fix\phi^1_H$. Then $[0,1]\ni t\mapsto d\phi^t_H(z_0)$ is a path of linear symplectomorphisms. Since the Conley-Zehnder index is a symplectic invariant \eqref{eq:CZ_sympl_invariant}, we can choose an arbitrary symplectic basis to represent the linear flow as a path of symplectic matrices and denote
\begin{equation}
	\begin{split}
		\CZ(z_0,H)=\CZ\left(d\phi^t_H(z_0)\right),\\
		\meanCZ(z_0,H)=\meanCZ\left(d\phi^t_H(z_0)\right).
	\end{split}
\end{equation}
This integer depends on the chosen generating Hamiltonian, but its parity does not.

Notice that $\meanCZ\left(z_0,H^{\times k}\right)=k\meanCZ\left(z_0,H\right)$. From equation \eqref{eq:iterated_CZ_meanCZ_bound} we obtain that
\begin{equation}\label{eq:iterated_CZ_meanCZ_bound-fixedpt}
	\left|\CZ\left(z_0,H^{\times k}\right)-k\meanCZ\left(z_0,H\right)\right|\leq n.
\end{equation}

If $H$ defines an asymptotically quadratic Hamiltonian system, there is a well defined linear Hamiltonian system at infinity. We will use this to define two central notions for the paper.
\begin{defn}\label{def:index_at_infty}
	Given $H=Q+h$ asymptotically quadratic, define the \emph{index at infinity} $\ind_\infty(H)$ and the \emph{mean index at infinity} $\overline{\ind_\infty}(H)$ of $H$ as
	\begin{equation}
		\ind_\infty(H)=\CZ\left(\phi^t_Q\right)\in\bZ,\quad\overline{\ind}_\infty(H)=\meanCZ\left(\phi^t_Q\right)\in\bR
	\end{equation}
	where we have chosen an arbitrary symplectic basis of $\bR^{2n}$ to represent the linear flow at infinity as a path of symplectic matrices.
\end{defn}
This integer again depends on the generating Hamiltonian, but its parity does not.
As was discussed in the introduction, the twist condition we have in mind is given by the comparison between the mean index at infinity and the mean index of a fixed point.
\begin{defn}\label{def:twist_fixed_pt}
	Let $\upphi$ be an ALHD and $z_0\in\Fix\upphi$ be a fixed point. We say that $z_0$ is a \emph{twist fixed point} if there exists a generating Hamiltonian $H$ for $\upphi$ for which
	\begin{equation}\label{eq:twist_fixed_pt_def}
		\overline{\ind}_\infty(H)\neq\meanCZ\left(z_0,H\right).
	\end{equation}
\end{defn}
Even though the quantities entering the twist condition depend on the chosen generating Hamiltonian, the twist condition itself does not.
\begin{lemma}\label{lem:twisted_fixed_pt_well_def}
	Let $\upphi$ be an ALHD and $z_0\in\Fix\upphi$ be a fixed point. Let $F=Q+f$ and $G=P+g$ be asymptotically quadratic Hamiltonians generating $\upphi$.
	\begin{equation}\label{eq:twisted_fixed_pt_well_def}
		\overline{\ind}_\infty F-\meanCZ\left(z_0,F\right)=\overline{\ind}_\infty G-\meanCZ\left(z_0,G\right).
	\end{equation}
\end{lemma}
\begin{proof}
	Since $\upphi=\phi^1_F=\phi^1_G$, the flows of $F$ and $G$ are related by a loop of asymptotically linear Hamiltonian diffeomorphisms. For example, we can take the concatenation
	\begin{equation}
		K=F\wedge\overline{G}=R+k,\quad \phi^\cdot_{F\wedge\overline{G}}=\phi^\cdot_F\wedge\left(\phi^\cdot_G\right)^{-1}.
	\end{equation}
	The quadratic Hamiltonian $R$, which is just the concatenation of $Q$ after $P$, generates a loop of linear symplectomorphisms relating the linear flows of $Q$ and $P$. In lemma \ref{lem:loops_ALHDs_maslov} it is shown that any Hamiltonian $K$ generating a loop of asymptotically linear Hamiltonian diffeomorphisms has the following property: there is an integer $\Mas K\in\bZ$ such that
	\begin{equation}
		\Mas K=\Mas\left(d\phi^t_K(z)\right)=\Mas\left(\phi^t_R\right)\quad\forall z\in\bR^{2n}.
	\end{equation}
	Now, since $\upphi^l=\phi^1_{F^{\times l}}=\phi^1_{G^{\times l}}$ there is a Hamiltonian $K_l=R_l+k_l$ generating a loop of asymptotically linear Hamiltonian diffeomorphisms relating the corresponding flows. Set $\kappa_l=\Mas K_l\in\bZ$. Using the loop composition formula \eqref{eq:CZ_loop_composition_formula},
	\begin{equation}
		\begin{multlined}
			\overline{\ind}_\infty G-\meanCZ\left(z_0,G\right)=\lim_{l\to\infty}\frac{1}{l}\left[\CZ\left(\phi^{lt}_P\right)-\CZ\left(d\phi^{lt}_G(z_0)\right)\right]=\\
			=\lim_{l\to\infty}\frac{1}{l}\left[\CZ\left(\phi^{lt}_Q\right)+2\Mas\left(\phi^t_{R_l}\right)-\CZ\left(d\phi^{lt}_H(z_0)\right)-2\Mas\left(d\phi^t_{K_l}(z_0)\right)\right]=\\
			=\lim_{l\to\infty}\frac{1}{l}\left[\CZ\left(\phi^{lt}_Q\right)-\CZ\left(d\phi^{lt}_H(z_0)\right)+2\kappa_l-2\kappa_l\right]=\overline{\ind}_\infty H-\meanCZ\left(z_0,H\right).
		\end{multlined}
	\end{equation}
	This concludes the proof.
\end{proof}

\section{Modifying linear and asymptotically linear Hamiltonian systems}
\label{sec:modifying_linear_and_asymptotically_linear_hamiltonian_systems}

In this section we explain three constructions which will be important for the proof of the Poincaré-Birkhoff theorem for rapidly asymptotically unitary Hamiltonian diffeomorphisms.

The first construction allows us to change the index at infinity of an asymptotically linear Hamiltonian system. It is an elementary application of composition with loops of linear symplectomorphisms. We will focus on iterated Hamiltonians, in order to construct a Hamiltonian generating a given iterate of an ALHD but with shifted index at infinity. The action of the fixed points calculated with this ``re-indexed'' Hamiltonian remains the same, but their index is shifted by the same amount as the shifting of index at infinity.

The second construction is an interpolation of non-resonant linear unitary Hamiltonian systems. We are given two quadratic Hamiltonians which generate linear, unitary flows and we assume they have the same index. The goal is to construct a semi-quadratic Hamiltonian which interpolates between the two given quadratic Hamiltonians. The interpolation is done slowly on a large annular region, so that no additional 1-periodic orbits are created in the process.

The last construction is a delicate truncation of the sub-quadratic part of an asymptotically quadratic Hamiltonian with unitary flow at infinity. We are given an asymptotically quadratic Hamiltonian which is non-resonant at infinity, and whose linear flow at infinity is unitary. The result of the truncation is a Hamiltonian which is a compact perturbation of its quadratic Hamiltonian at infinity. The truncation happens on an annular region and again conditions are given so that no new 1-periodic orbits are created.

All these constructions will be applied to certain iterated Hamiltonian systems in later sections. We give quantitative estimates on the truncation and interpolation radii, in order to be able to control them along iterations of a Hamiltonian system.

\subsection{Linear loops acting on asymptotically linear Hamiltonian systems}
\label{sub:linear_loops_acting_on_asymptotically_linear_hamiltonian_systems}

We start with a lemma listing some elementary properties of composition of an asymptotically quadratic Hamiltonian with a quadratic Hamiltonian generating a loop of linear symplectomorphisms.
\begin{lemma}\label{lem:linear_flow_composition_preserves_action}
	Let $G\colon S^1\times \bR^{2n}\to\bR$, $G_t(z)=R_t(z)+g_t(z)$ be an asymptotically quadratic Hamiltonian and $P_t(z)=\frac12\left\langle B_tz,z\right\rangle$ a time-dependent quadratic Hamiltonian generating a loop in $\Sp(2n)$. Denote by $\sigma^g_1,\sigma^g_0$ the tail functions of $g$ \eqref{eq:decay_moduli_growth_moduli}. Define $H=P\#G$.
	\begin{enumerate}
		\item\label{item:linflow_comp_props_preserves_asyquad_ndg} If $\phi^t_P\in\Un(n)\subset\Sp(2n)$ for all $t$, then $H$ is asymptotically quadratic. In fact writing $H=Q+h$, we have $\sigma^h_j=\sigma^g_j$ for $j=0,1$. Moreover, if $G$ is non-resonant at infinity, then so is $H$.
		\item\label{item:linflow_comp_props_preserves_action_of_fixedpts} For every $z\in\Fix\phi^1_{H}=\Fix\phi^1_G$, we have
			\begin{equation}\label{eq:looping_preserves_action}
				\cA_{H}(z)=\cA_G(z)
			\end{equation}
	\end{enumerate}
\end{lemma}
\begin{proof}
	Set $\phi^t_P=:N_t\in\Sp(2n)$. We are assuming that $N_0=N_1=\bI$.
	\begin{enumerate}
		\item By the definition of the composition of Hamiltonians \eqref{eq:composition_of_hamilts}, we have
			\begin{equation}
				P\#G_t=P\#\left(R_t+g_t\right)=P_t+R_t\circ\left(N_t\right)^{-1}+g_t\circ\left(N_t\right)^{-1}
			\end{equation}
			which prompts us to set $U_t=\left(N_t\right)^{-1}$ and write
			\begin{equation}
				Q_t=P_t+R_t\circ U_t,\quad h_t=g_t\circ U_t.
			\end{equation}
			If $R_t(z)=\frac12\left\langle A'_tz,z\right\rangle$ then $Q_t(z)=\frac12\left\langle A_tz,z\right\rangle$ where $A=B+U^TA'U$.
			The statement on non-resonance at infinity follows from the fact that $N_t$ is a loop.
			The asymptotic growth bounds follow simply from the chain rule and the norm preservation of unitary maps. Indeed, since $N_t,U_t\in\Un(n)$,
			\begin{equation}
				\left\|\nabla^2 h\right\|_{L^\infty}=\left\|N_t\left(\nabla^2 g\circ U_t\right)U_t\right\|_{L^\infty}=\left\|\nabla^2 g\right\|_{L^\infty}<\infty.
			\end{equation}
			Similarly,
			\begin{equation}
				\begin{multlined}
					\sigma^h_1(R)=\sup_{|z|\geq R}\frac{\left|\nabla h(z)\right|}{|z|}=\sup_{|z|\geq R}\frac{\left|N_t\nabla g_t\left(U_tz\right)\right|}{|z|}=\\
					\sup_{|U_tz|\geq R}\frac{\left|\nabla g_t\left(U_tz\right)\right|}{|U_tz|}=\sigma_1^g(R)
				\end{multlined}
			\end{equation}
			and same for $\sigma_0$.
		\item Recall the definition of the action of a fixed point in \eqref{eq:action_of_fixed_pt} and that $\phi^t_{H}=N_t\phi^t_G$. Here $N_t\in\Sp(2n)$ doesn't have to be unitary. We calculate
			\begin{equation*}
				\begin{multlined}
					\cA_{H}\left(z\right)=\int_0^1\frac12\left\langle N_t\phi^t_G(z), J_0\frac{d}{dt}\left(N_t\phi^t_G(z)\right)\right\rangle-P\#G_t\left(N_t\phi^t_G(z)\right)dt=\\
					=\int_0^1\frac12\left\langle N_t\phi^t_G(z), B_tN_t\phi^t_G(z)\right\rangle-P_t\left(N_t\phi^t_G(z)\right)dt\,+\\
					\qquad+\int_0^1\frac12\left\langle N_t\phi^t_G(z), J_0 N_t\dot\phi^t_G(z)\right\rangle-G_t\left(\phi^t_G(z)\right)dt=\\
					=\int_0^1\frac12\left\langle\phi^t_G(z),J_0\dot\phi^t_G(z)\right\rangle-G_t\left(\phi^t_G(z)\right)dt=\cA_G(z)
				\end{multlined}
			\end{equation*}
			since $N_t$ is a symplectomorphism for all $t$, so $(N_t)^TJ_0N_t=J_0$.
	\end{enumerate}
\end{proof}
\begin{rmk}
	We are tacitly imposing a \emph{normalization condition} on linear symplectomorphisms: we don't allow constant terms in the quadratic functions defining quadratic Hamiltonians. This implies that the action of the origin as the fixed point of a linear symplectomorphism is always zero.
\end{rmk}

\subsubsection{Re-indexing iterated Hamiltonians}
\label{ssub:re_indexing_iterated_hamiltonians}

Let $H_t=Q_t+h_t$ be a asymptotically quadratic Hamiltonian, $\upphi=\phi^1_H$, and $\upphi_\infty=\phi^1_Q$. Here we are interested in composing the quadratic Hamiltonians at infinity of $H^{\times k}$ with loops which change its index at infinity and the effect this has on the index and action of $k$th iterates of fixed points of $\upphi$.

The following lemma is a well known property of the iterated Conley-Zehnder index, which can be proven using the parity calculation \cite[Theorem 3.3 (iii)]{SalamonZehnder1992_MorseTheoryMaslovIndex}, see equation \eqref{eq:parity_CZ} above.
\begin{lemma}\label{lem:CZ_parity}
	If $k,l\in\bZ$ have the same parity, then $\ind_\infty\left(H^{\times k}\right)$ and $\ind_\infty\left(H^{\times l}\right)$ have the same parity.
\end{lemma}

Fix two odd integers $k>l\geq 1$. The previous lemma implies in particular that
\begin{equation}\label{eq:difference_indinfty_odd_iterates_is_even}
	\ind_\infty\left(H^{\times k}\right)-\ind_\infty\left(H^{\times l}\right)=2\mu
\end{equation}
for some $\mu\in\bZ$.
Using the estimate \eqref{eq:iterated_CZ_meanCZ_bound} (see \cite[Theorem 1.4.3]{Abbondandolo2001_MorseHamilt}), we see that
\begin{equation}
	\left|\ind_\infty H^{\times j}-j\overline{\ind}_\infty H\right|\leq n\quad\forall j\in\bZ.
\end{equation}
We can thus estimate $\sigma=2\mu$ in terms of the mean index at infinity, namely
\begin{equation}\label{eq:reloop_index_shift_estimate}
	\left|\sigma-(k-l)\overline\ind_\infty H\right|\leq 2n.
\end{equation}
Since the Maslov index induces an isomorphism between $\pi_1\left(\Sp(2n)\right)$ and $\bZ$ \cite[Theorem 2.2.12]{McDuffSalamon2017_Intro}, and $\Sp(2n)$ is homotopy equivalent to $\Un(n)$ \cite[Proposition 2.2.4]{McDuffSalamon2017_Intro}, there always exists a loop of unitary matrices based at $\bI$ with Maslov index $\mu$. This loop can always be generated by a time-dependent quadratic Hamiltonian $P^\mu_t(z)=\frac12\left\langle B^\mu_tz,z\right\rangle$. Using composition $\#$ and concatenation $\wedge$ of Hamiltonians (recall their definition in equation \eqref{eq:composition_of_hamilts}), define the following Hamiltonian
\begin{equation}\label{eq:def_of_relooping_at_infty}
	H^{k\ominus l}=\left(\overline{P^\mu}\#H^{\times(k-l)}\right)\wedge H^{\times l}.
\end{equation}
\begin{rmk}
	The basic reasoning which led to this formula for $H^{k\ominus l}$ is to consider $H^{\times (k-l)}\wedge H^{\times l}$, which generates a flow homotopic to the one generated by $H^{\times k}$, but which separates the part with ``excessive'' index $H^{\times(k-l)}$ from $H^{\times l}$. Then one kills the index of the ``excessive part'' via the loop generated by $\overline{P^\mu}$.
\end{rmk}

\begin{lemma}\label{lem:relooped_hamilt_properties}
	The Hamiltonian $H^{k\ominus l}$ is an asymptotically quadratic Hamiltonian which generates $\upphi^k$. It has the following properties.
	\begin{enumerate}
		\item $\ind_\infty\left(H^{k\ominus l}\right)=\ind_\infty\left(H^{\times l}\right)$. Moreover if $\bar z\in\Fix\phi^k_H=\Fix\phi^1_{H^{k\ominus l}}$ then
			\begin{equation}\label{eq:reloop_index_shift_k-ppoint}
				\CZ\left(\bar z,H^{k\ominus l}\right)=\CZ\left(\bar z,H^{\times k}\right)-2\mu
			\end{equation}
		\item If $z_0\in\Fix\phi^1_H$ is seen as a $k$-periodic point, then
			\begin{equation}
				\cA_{H^{k\ominus l}}(z_0)=\cA_{H^{\times k}}(z_0)=k\cA_H(z_0).
			\end{equation}
	\end{enumerate}
\end{lemma}
\begin{proof}
	That $H^{k\ominus l}$ generates $\upphi^k$ follows immediately from definition, and that it is asymptotically quadratic follows from lemma \ref{lem:linear_flow_composition_preserves_action}, point \ref{item:linflow_comp_props_preserves_asyquad_ndg}.
	\begin{enumerate}
		\item Notice that $H^{k\ominus l}$ generates a path which is homotopic with fixed endpoints to the path generated by $\overline{P^\mu}\#H^{\times k}$. The quadratic Hamiltonian at infinity of this Hamiltonian is $\overline P^\mu\#Q^{\times k}$. Therefore, by the invariance of the Conley-Zehnder index and the loop composition formula \eqref{eq:CZ_loop_composition_formula},
			\begin{equation}
				\begin{multlined}
					\ind_\infty\left(H^{k\ominus l}\right)=\CZ\left(\left(\phi^t_{P^\mu}\right)^{-1}\circ\phi^t_{Q^{\times (k-l)}}\wedge\phi^t_{Q^{\times l}}\right)=\\
					=\CZ\left(\left(\phi^t_{P^\mu}\right)^{-1}\circ\phi^t_{Q^{\times k}}\right)=\\
					=\CZ\left(\phi^t_{Q^{\times k}}\right)+2\Mas\left(\left(\phi^t_{P^\mu}\right)^{-1}\right)=\\
					=\ind_\infty\left(H^{\times k}\right)-2\mu=\ind_\infty\left(H^{\times l}\right).
				\end{multlined}
			\end{equation}
			The calculation of the index of 1-periodic orbits of $H^{k\ominus l}$ is completely analogous.
		\item Notice that $z_0\in\Fix\phi^1_H\implies z_0\in\Fix\phi^l_H\cap\Fix\phi^{k-l}_H$. So we are reduced to the situation of two Hamiltonians $F,G$ and a point $z_0\in\Fix\phi^1_F\cap\Fix\phi^1_G$. Then $z_0\in\Fix\phi^1_F\circ\phi^1_G$ and
			\begin{equation}
				\begin{multlined}
					\cA_{F\wedge G}(z_0)=\int_0^{\frac12}\left(\phi^{\rho(2\,\cdot\,)}_G(z_0)\right)^*\lambda_0-2\rho'(2t)G_{\rho(2t)}\circ\phi^{\rho(2t)}_G(z_0)dt+\\
					+\int_{\frac12}^1\left(\phi^{\rho(2\,\cdot\,-1)}_F\circ\phi^1_G(z_0)\right)^*\lambda_0-2\rho'(2t)F_{\rho(2t-1)}\circ\phi^{\rho(2t-1)}_F\circ\phi^1_G(z_0)dt\\
					=\int_0^1\left(\phi^\cdot_G(z_0)\right)^*\lambda_0-G_t\circ\phi^t_G(z_0)dt+\int_0^1\left(\phi^\cdot(z_0)\right)^*\lambda_0-F_t\circ\phi^t_F(z_0)dt=\\
					=\cA_G(z_0)+\cA_F(z_0)
			\end{multlined}
		\end{equation}
		Using lemma \ref{lem:linear_flow_composition_preserves_action}, point \ref{item:linflow_comp_props_preserves_action_of_fixedpts}, we calculate
			\begin{equation}
				\begin{multlined}
				\cA_{H^{k\ominus l}}(z_0)=\cA_{H^{\times l}}(z_0)+\cA_{\overline P^\mu\#H^{\times(k-l)}}(z_0)=\\
				=l\cA_H(z_0)+\cA_{H^{\times(k-l)}}(z_0)=l\cA_H(z_0)+(k-l)\cA_{H}(z_0)=k\cA_H(z_0).
				\end{multlined}
			\end{equation}
	\end{enumerate}
\end{proof}
\begin{rmk}
	Notice that in general the action of the $k$-periodic orbits which are not $k$-fold iterates of 1-periodic orbits might not be preserved by this procedure.
\end{rmk}

\subsection{Interpolation and truncation}
\label{sub:interpolation_and_truncation}

In this section we develop the interpolation and truncation constructions which we briefly introduced above. Both these constructions rely on the linear system at infinity being unitary.

\subsubsection{Interpolation of quadratic Hamiltonians}
\label{ssub:interpolation_of_quadratic_hamiltonians}

First, a preliminary definition. Let $Q^0,Q^1$ be quadratic Hamiltonians and $\upphi_0,\upphi_1\in\Sp(2n)$ the linear symplectomorphisms which they define. Assume that $\det(\upphi_j-\bid)\neq 0$ and
\begin{equation}
	\CZ\left(\phi^t_{Q^0}\right)=\CZ\left(\phi^t_{Q^1}\right).
\end{equation}
Then there is a path $\mathcal{Q}\colon [0,1]\times S^1\times\bR^{2n}\to\bR$, $\mathcal{Q}=\mathcal{Q}^s_t$ of quadratic Hamiltonians, say
\begin{equation}\label{eq:nonres_homotopy_lin_systems}
	\mathcal{Q}^s_t(z)=\frac12\left\langle\bA^s_tz,z\right\rangle,
\end{equation}
such that
\begin{equation}\label{eq:nonres_homotopy_lin_systems_props}
	\mathcal{Q}^0=Q^0,\quad\mathcal{Q}^1=Q^1,\quad \det\left(\phi^1_{\mathcal{Q}^s}-\bid\right)\neq 0\quad\forall s\in[0,1].
\end{equation}
We call such a homotopy of quadratic Hamiltonians a \emph{non-resonant homotopy} of linear Hamiltonian systems. If moreover $\phi^t_{Q^0},\phi^t_{Q^1}\in\Un(n)\subset\Sp(2n)$ for all $t\in[0,1]$, it is always possible to chose $\mathcal{Q}$ in such a way that we also have $\phi^t_{\mathcal{Q}^s}\in\Un(n)$ for all $(s,t)\in[0,1]\times[0,1]$. We call such a homotopy a non-resonant homotopy of unitary Hamiltonian systems.

\begin{prop}\label{prop:nonresonant_interpolation_quadratic_forms}
	Let $Q^0,Q^1\colon S^1\times\bR^{2n}\to\bR$ be two 1-periodic time-dependent quadratic Hamiltonians such that $\det\left(\phi^1_{Q^i}-\bI\right)\neq 0$ for $i=0,1$ and
	\begin{equation}\label{eq:nonresonant_interpolation-sameindex}
		\CZ\left(\phi^t_{Q^0}\right)=\CZ\left(\phi^t_{Q^1}\right).
	\end{equation}
	Assume that $\phi^t_{Q^i}\in\Un(n)$ for all $t$ and all $i=0,1$. Then for every $R_0>0$ and every non-resonant homotopy of unitary Hamiltonian systems, there exists an $R_1>R_0$ and a Hamiltonian $K\in C^\infty(S^1\times \bR^{2n})$ with the following properties:
	\begin{enumerate}
		\item $K$ interpolates between $Q^0$ and $Q^1$:
			\begin{equation}
				K_t=Q^0_t,\ \forall |z|\leq\sqrt{R_0},\quad K_t=Q^1_t\ \forall |z|\geq\sqrt{R_1}.
			\end{equation}
		\item $K$ has no non-trivial 1-periodic orbits:
			\begin{equation}
				\Fix\phi^1_K=\{0\}
			\end{equation}
	\end{enumerate}
\end{prop}
\begin{proof}
	Since the Conley-Zehnder indices of the flows of $Q^0$ and $Q^1$ are the same, the two quadratic Hamiltonians are non-resonant homotopic. Choose a non-resonant homotopy of unitary Hamiltonian systems defined by a path of quadratic Hamiltonians $\mathcal{Q}=\mathcal{Q}^s_t$, $s\in[0,1]$. Since the flow of $\mathcal{Q}^s$ is unitary for all $s$, its corresponding symmetric matrix \eqref{eq:nonres_homotopy_lin_systems} satisfies $-J_0\bA^s_t\in\operatorname{\mathfrak{u}}(n)\subset\operatorname{\mathfrak{o}}(2n)$ for all $(s,t)$.

	Given an arbitrary $R_0>0$, let $R_1>R_0$ and let $\chi\colon[0,\infty)\to[0,1]$ be a non-decreasing smooth function such that $\chi(r)=0$ for all $r<R_0$ and $\chi(r)=1$ for all $r>R_1$. The Hamiltonian $K$ is defined as follows:
	\begin{equation}\label{eq:def_of_interpolation_at_infty}
		K(t,z)=\mathcal{Q}^{\chi\left(|z|^2\right)}_t(z)=\frac12\left\langle\bA^{\chi\left(\left|z\right|^2\right)}_tz,z\right\rangle.
	\end{equation}
	We want to determine the function $\chi$ and the radius $R_1$ in terms of $\mathcal{Q}$ and $R_0$ so that the Hamiltonian system defined by $K$ has only $0\in\bR^{2n}$ as 1-periodic orbit. We calculate the Hamiltonian vector field
	\begin{equation}\label{eq:interpolation_at_infty_HamVF}
		X_K(t,z)=-J_0\bA^{\chi\left(|z|^2\right)}_tz-\chi'\left(|z|^2\right)\left\langle\left.\frac{\partial\bA^s_t}{\partial s}\right|_{s=\chi\left(|z|^2\right)}z,z\right\rangle J_0z
	\end{equation}
	The Hamilton equations for an integral curve $x\colon[0,T]\to\bR^{2n}$ of $K$ can thus be written as
	\begin{equation*}
		\dot x=X_K(x)\iff \dot x+J_0\bA^{\chi\left(|x|^2\right)}_tx=-\chi'\left(|x|^2\right)\left\langle\left.\frac{\partial\bA^s_t}{\partial s}\right|_{s=\chi\left(|x|^2\right)}x,x\right\rangle J_0x
	\end{equation*}
	
	Let's show that the norm of an integral curve of $K$ is constant. Indeed, if $x\colon [0,T]\to\bR^{2n}$ is an integral curve of $K$,
	\begin{equation}\label{eq:interp_hamilt_unitary_case_norm_cst}
		\frac12\frac{d}{dt}|x|^2=\left\langle x,\dot{x}\right\rangle=-\left\langle x,J_0\bA^{\chi\left(|x|^2\right)}_tx\right\rangle-\chi'\left(|x|^2\right)\left\langle\left.\frac{\partial\bA^s_t}{\partial s}\right|_{s=\chi\left(|x|^2\right)}x,x\right\rangle\left\langle x,J_0x\right\rangle=0
	\end{equation}
	because $J_0\bA^s_t\in\operatorname{\mathfrak{u}}(n)\subset\operatorname{\mathfrak{o}}(2n)$ which is the set of skew symmetric matrices. Therefore, if we set $r_0=|x(0)|^2$, $\chi_0=\chi(r_0)$, $\chi'_0=\chi'(r_0)$, we can use the fact that $\bA^s$ gives a non-degenerate Hamiltonian linear system for all $s$ to invert the operator $\frac{d}{dt}+J_0\bA^{\chi_0}_t$ using the variation of constants method, see \cite[\S III, Proposition 2]{Ekeland1990_ConvexityMethods}. We obtain the integral expression
	\begin{equation}
		x_t=\bM^{\chi_0}_t\left[x_0-\int_0^t\chi'_0\left\langle\left.\frac{\partial\bA^s_t}{\partial s}\right|_{s=\chi_0}x_\tau,x_\tau\right\rangle\left(\bM^{\chi_0}_\tau\right)^{-1}J_0x_\tau d\tau\right]
	\end{equation}
	where $\bM^s\colon[0,1]\to\Un(n)\subset\Sp(2n)$ is  $\bM^s_t=\phi^t_{\mathcal{Q}^s}$. We use this formula to estimate
	\begin{equation}\label{eq:interpolation_at_infty_perorb_closure_estimate}
		\begin{split}
			\left|x_1-x_0\right|&=\left|\bM^{\chi_0}_1 x_0-x_0-\int_0^1\chi_0'\left\langle\left.\frac{\partial\bA^s_t}{\partial s}\right|_{s=\chi_0}x_\tau,x_\tau\right\rangle\bM^{\chi_0}_{1-\tau}x_\tau d\tau\right|\geq\\
			&\geq \left|\left|\bM^{\chi_0}_1 x_0-x_0\right|-\left|\int_0^1\chi_0'\left\langle\left.\frac{\partial\bA^s_t}{\partial s}\right|_{s=\chi_0}x_\tau,x_\tau\right\rangle\bM^{\chi_0}_{1-\tau}x_\tau d\tau\right|\right|
		\end{split}
	\end{equation}
	Since $\mathcal{Q}$ is a non-resonant homotopy,
	\begin{equation}
		C_0=\min_{s\in[0,1],\,z\in\bR^{2n}\setminus0}\frac{\left|\phi^1_{\mathcal{Q}^s}z-z\right|}{|z|}>0.
	\end{equation}
	Indeed, define
	\begin{equation}
		\nu_\infty(s)=\left|\left(\phi^1_{\mathcal{Q}^s}-\bid\right)^{-1}\right|^{-1}>0\quad\forall s\in[0,1].
	\end{equation}
	Notice that
	\begin{equation}
		\left|\phi^1_{\mathcal{Q}^s}z-z\right|\geq\left|\left(\phi^1_{\mathcal{Q}^s}-\bid\right)^{-1}\right|^{-1}|z|
	\end{equation}
	which implies that
	\begin{equation}
		C_0\geq\min_{s\in[0,1]}\nu_\infty(s)>0.
	\end{equation}
	Using $C_0$ we can estimate the first term in the second line of equation \ref{eq:interpolation_at_infty_perorb_closure_estimate} by
	\begin{equation}\label{eq:parametric_nnres_perorb_estimate}
		\left|\bM^s_1z-z\right|\geq C_0|z|\quad\forall z\in\bR^{2n}\quad \forall s\in[0,1].
	\end{equation}
	with equality if and only if $z=0$.
	We are finished if we can bound the integral in the second line of equation \eqref{eq:interpolation_at_infty_perorb_closure_estimate} from above by $C_0|x_0|$. We start with
	\begin{equation}
		\left|\int_0^1\chi_0'\left\langle\left.\frac{\partial\bA^s_t}{\partial s}\right|_{s=\chi_0}x_\tau,x_\tau\right\rangle\bM^{\chi_0}_{1-\tau}x_\tau d\tau\right|\leq \int_0^1C_1\chi'\left(|x_0|^2\right)|x_0|^3d\tau=C_1\chi'(r_0)r_0^{3/2}
	\end{equation}
	where $C_1=\left\|\partial_s\bA\right\|_{L^\infty}$. We are led to impose the point-wise constraint
	\begin{equation*}
		\chi'(r)\leq\frac{C_0}{C_1r}
	\end{equation*}
	\begin{figure}[t]
		\centering
		\begin{overpic}[width=.8\textwidth]{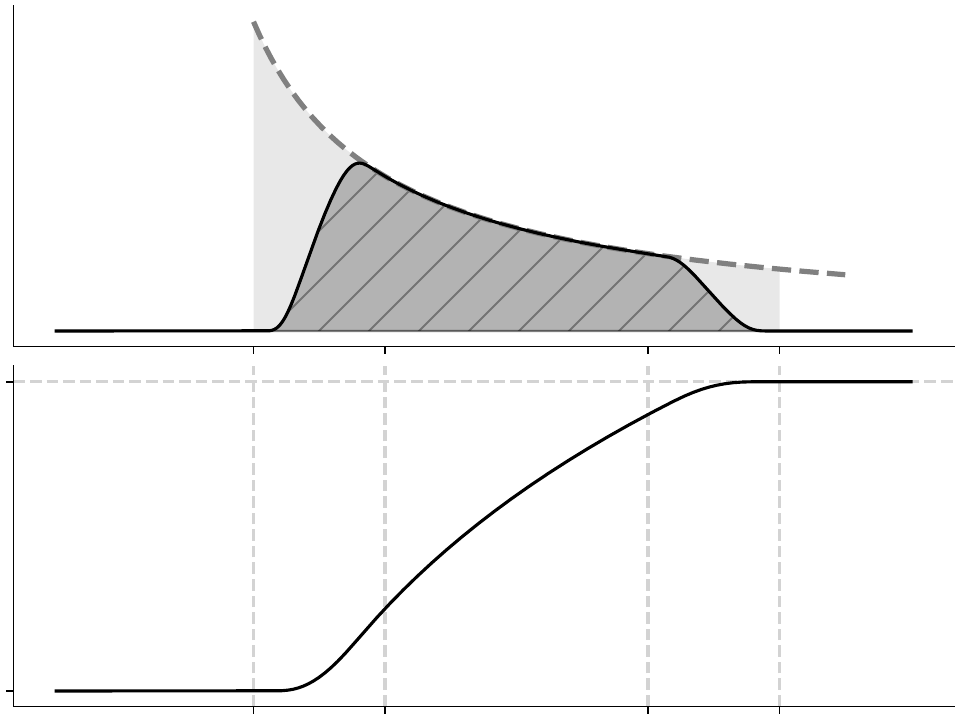}
			\put (25.5,-2.5) {$R_0$}
			\put (35,-2.5) {$R_0+\epsilon$}
			\put (63,-2.5) {$R_1-\epsilon$}
			\put (79,-2.5) {$R_1$}
			\put (-1.5,2) {$0$}
			\put (-1.5,34) {$1$}
			\put (30,70) {\color{darkgray}$\displaystyle\frac{C_0}{C_1r}$}
			\put (50,54) {$\chi'$}
			\put (50,25) {$\chi$}
		\end{overpic}
		\caption{Construction of the step function $\chi$. The smoothing is done in such a way that the light gray area contributes $\delta$ to the total area and the dark gray area contributes $1$ to the total area.}
		\label{fig:bumps}
	\end{figure}
	Observe that we must have $\int_\bR\chi'=1$, while $r^{-1}$ has diverging integral. Fix a small $\delta>0$. Given any fixed $R_0>0$, let $R_1>R_0$ be such that
	\begin{equation}\label{eq:def_of_last_interp_radius}
		\int_{R_0}^{R_1}\frac{C_0}{C_1r}dr=1+\delta\iff R_1=e^{(1+\delta)\frac{C_1}{C_0}}R_0
	\end{equation}
	For $\epsilon>0$ to be determined, define $\chi'$ to be a non-negative function such that $\chi'(r)=\frac{C_0}{C_1r}$ for $r\in[R_0+\epsilon,R_1-\epsilon]$ and $\chi'(r)=0$ for $r\notin[R_0,R_1]$ (see Figure \ref{fig:bumps}). We can assume that $\epsilon>0$ is chosen in such a way that $\int_\bR\chi'=1$. The constant $\epsilon$ is determined by $\delta$ and by how rapidly $\chi'$ interpolates between $0$ and $\frac{C_0}{C_1r}$. Imposing the boundary condition that $\chi(R_0)=0$ gives an unique primitive $\chi$ of our $\chi'$ which will have the sought-for properties.
\end{proof}
We gather the formulas for the two constants $C_0$ and $C_1$ which enter the definition of $R_1$, since their behaviour under iteration is important in the rest of the paper:
\begin{equation}\label{eq:interpolation_constants}
	C_0=\min_{s\in[0,1],\,z\in\bR^{2n}}\frac{\left|\phi^1_{\mathcal{Q}^s}z-z\right|}{|z|},\quad C_1=\left\|\partial_s\bA\right\|_{L^\infty\left([0,1]\times S^1,\Sym2n\right)}=\left\|\partial_s\nabla^2_z\mathcal Q\right\|_{L^\infty}.
\end{equation}
Recall that
\begin{equation}\label{eq:interp_cst_C0_est_below_reso_prox}
	C_0\geq\min_{s\in[0,1]}\nu_\infty(s)>0
\end{equation}
since the non-resonance condition is equivalent to the fact that
\begin{equation}\label{eq:reso_prox_cst_homotopy}
	\nu_\infty(s)=\left|\left(\phi^1_{\mathcal{Q}^s}-\bid\right)^{-1}\right|^{-1}>0\quad\forall s\in[0,1].
\end{equation}

\subsubsection{Truncation of the sub-quadratic term}
\label{ssub:truncation_of_the_sub_quadratic_term}

Here we explain how to truncate the sub-quadratic term of an asymptotically quadratic Hamiltonian, non-resonant and unitary at infinity.

\begin{prop}\label{prop:nonresonant_truncation}
	Let $F=P+f$ be an asymptotically quadratic Hamiltonian, non-resonant at infinity. Assume that $\phi^t_P\in\Un(n)$ for all $t\in[0,1]$.
	There exist $0<R_a<R_b$ and a Hamiltonian
	\begin{equation}\label{eq:truncated_hamilt}
		\widetilde{F}_t(z)=P_t(z)+\widetilde{f}_t(z)
	\end{equation}
	with the following properties:
	\begin{enumerate}
		\item $\widetilde F_t(z)=F_t(z)$ when $|z|<\sqrt{R_a}$ and $\widetilde f$ is compactly supported in a ball of radius $\sqrt{R_b}$. In particular $\widetilde F_t(z)=P_t(z)$ for all $|z|>\sqrt{R_b}$.
		\item $\Fix\phi^1_{\widetilde F}=\Fix\phi^1_F$.
	\end{enumerate}
\end{prop}
\begin{proof}
	Assume that the quadratic Hamiltonian at infinity of $F$ is $P_t(z)=\frac12\left\langle B_tz,z\right\rangle$.
	For some chosen $0<R_a<R_b$, let $\rho\colon\bR\to[0,1]$ be an arbitrary smooth function such that
	\begin{equation}\label{eq:trunclem_step_props}
		\rho(r)=1\quad \forall r\leq R_a,\quad \rho(r)=0\quad \forall r\geq R_b,\quad \rho'(r)\leq 0\quad \forall r.
	\end{equation}
	Set
	\begin{equation}
		\widetilde F_t(z)=P_t(z)+\rho\left(|z|^2\right)f_t(z),\quad\rho\left(|z|^2\right)f_t(z)=\widetilde f_t(z).
	\end{equation}
	Recall from lemma \ref{lem:1-per_orbits_of_nondeg_at_infty_asyquad_are_unif_Linfty_bounded} that the fixed points of $\phi^1_F$ are contained in a ball of radius $R_1>0$ given by equation \eqref{eq:fund_radius_unitary}. Since we don't want to destroy these fixed points, start by taking $\sqrt{R_a}>R_1$ and an arbitrary $R_b>R_a$. Notice that if $|z|\leq\sqrt{R_b}$ then $|\phi^t_{\tilde F}(z)|\leq\sqrt{R_b}$ for all $t$ because the flow of $P$ preserves the norm. It suffices to consider $\sqrt{R_a}\leq |z|\leq\sqrt{R_b}$, because for $|z|>\sqrt{R_b}$, $\widetilde F=P$ definitely does not have any 1-periodic orbit. Set $x_t=\phi^t_{\widetilde F}(z)$ and $y_t=\phi^t_F(z)$.
	\begin{equation}
		\begin{multlined}[b]
			\left|x_\tau-y_\tau\right|\leq\int_0^\tau\left|\nabla\widetilde F_t(x_t)-\nabla F_t(y_t)\right|dt=\\
			=\int_0^\tau\left|B_tx_t+\rho\left(\left|x_t\right|^2\right)\nabla f_t(x_t)+2\rho'\left(\left|x_t\right|^2\right)f_t(x_t)x_t-B_ty_t-\nabla f_t(y_t)\right|dt\\
				\leq \|B\|_{L^\infty}\int_0^\tau|x_t-y_t|dt+\int_0^\tau\left|\rho\left(\left|x_t\right|^2\right)\nabla f_t(x_t)-\nabla f_t(y_t)\right|dt+\\
			+2\int_0^\tau\left|\rho'\left(|x_t|^2\right)f_t(x_t)x_t\right|dt= I_0+I_1+I_2.
		\end{multlined}
	\end{equation}
	Now, since we know that $|x_t|\leq\sqrt{R_b}$, we can estimate
	\begin{equation}
		\begin{multlined}
			I_1=\int_{\left\{t:|x_t|<\sqrt{R_a}\right\}}\left|\nabla f_t(x_t)-\nabla f_t(y_t)\right|dt+\\
			+\int_{\left\{t:\sqrt{R_a}\leq|x_t|\leq\sqrt{R_b}\right\}}\left|\rho\left(\left|x_t\right|^2\right)\nabla f_t(x_t)-\nabla f_t(y_t)\right|dt\leq\\
			\leq\left\|\nabla^2f\right\|_{L^\infty\left(B_{\sqrt{R_a}}\right)}\int_0^\tau|x_t-y_t|dt+\\
			+\int_{\left\{t:\sqrt{R_a}\leq|x_t|\leq\sqrt{R_b}\right\}}\left|\rho\left(\left|x_t\right|^2\right)\nabla f_t(x_t)-\nabla f_t(y_t)\right|dt=I_{10}+I_{11}.
		\end{multlined}
	\end{equation}
	Continuing
	\begin{equation}
		\begin{multlined}[t]
			I_{11}=\int_{\left\{t:\sqrt{R_a}\leq|x_t|\leq\sqrt{R_b}\right\}}\left|\rho\left(\left|x_t\right|^2\right)\nabla f_t(x_t)-\nabla f_t(y_t)+\nabla f_t(x_t)-\nabla f_t(x_t)\right|dt\leq\\
			\leq \left\|\nabla^2f\right\|_{L^\infty\left(B_{\sqrt{R_b}}\setminus B_{\sqrt{R_a}}\right)}\int_0^\tau\left|x_t-y_t\right|dt+\\
			+\int_{\left\{t:\sqrt{R_a}\leq|x_t|\leq\sqrt{R_b}\right\}}\left[\rho\left(\left|x_t\right|^2\right)-1\right]\left|\nabla f_t(x_t)\right|dt\leq\\
			\leq \left\|\nabla^2f\right\|_{L^\infty\left(B_{\sqrt{R_b}}\setminus B_{\sqrt{R_a}}\right)}\int_0^\tau\left|x_t-y_t\right|dt+\tau\bar\sigma_1\left(\sqrt{R_a}\right)\sqrt{R_b}.
		\end{multlined}
	\end{equation}
	Now for $I_2$,
	\begin{equation}
		\begin{split}
			I_2&=2\int_{\left\{t:\sqrt{R_a}\leq|x_t|\leq\sqrt{R_b}\right\}}\left|\rho'\left(\left|x_t\right|^2\right)f_t\left(x_t\right)x_t\right|dt\leq\\
			&\leq2\left\|\rho'\right\|_{L^\infty}\int_{\left\{t:\sqrt{R_a}\leq|x_t|\leq\sqrt{R_b}\right\}}\sigma^f_0\left(|x_t|\right)|x_t|^3dt\leq\\
			&\leq2\tau\left\|\rho'\right\|_{L^\infty}\bar\sigma^f_0\left(\sqrt{R_a}\right)R_b^{\frac32}
		\end{split}
	\end{equation}
	where we recall that $\bar\sigma^f_0(R)$ is the smallest non-increasing majorant of $\sigma^f_0$. Putting it all together
	\begin{equation}
		\begin{split}
			\left|x_\tau-y_\tau\right|&\leq\left(\left\|A\right\|_{L^\infty}+\left\|\nabla^2f\right\|_{L^\infty\left(B_{\sqrt{R_a}}\right)}+\left\|\nabla^2f\right\|_{L^\infty\left(B_{\sqrt{R_b}}\setminus B_{\sqrt{R_a}}\right)}\right)\int_0^\tau\left|x_t-y_t\right|dt+\\
			&\qquad+\tau\bar\sigma^f_1\left(\sqrt{R_a}\right)\sqrt{R_b}+2\tau\bar\sigma^f_0\left(\sqrt{R_a}\right)R^{\frac32}_b\left\|\rho'\right\|_{L^\infty}\leq\\
			&\leq3C_2\int_0^\tau\left|x_t-y_t\right|dt+\tau\bar\sigma^f_1\left(\sqrt{R_a}\right)\sqrt{R_b}+2\tau\bar\sigma^f_0\left(\sqrt{R_a}\right)R^{\frac32}_b\left\|\rho'\right\|_{L^\infty}.
		\end{split}
	\end{equation}
	Grönwall's lemma \cite[Chapter 2, Lemma 2]{AbrahamMarsden1978_Foundations} gives
	\begin{equation}\label{eq:truncation_lemma_gronwall_estimate}
		\left|\phi^1_{\widetilde F}(z)-\phi^1_F(z)\right|\leq e^{3C_2}\left[\bar\sigma^f_1\left(\sqrt{R_a}\right)\sqrt{R_b}+2\bar\sigma^f_0\left(\sqrt{R_a}\right)R^{\frac32}_b\left\|\rho'\right\|_{L^\infty}\right].
	\end{equation}
	Recall the resonance proximity constant $\nu_\infty$ from definition \ref{def:resonance_proximity_cst} and the function $\nu(|z|)$ defined in equation \eqref{eq:1per_estimate_below}. Since $\phi^t_P$ is unitary, we have
	\begin{equation}\label{eq:truncation_lemma_below_estimate_cst_unitary}
		\nu\left(|z|\right)=\nu_\infty-e^{C_2}\bar\sigma^f_1\left(|z|\right).
	\end{equation}
	Assuming that $\sqrt{R_a}\leq |z|\leq\sqrt{R_b}$, we can use \eqref{eq:truncation_lemma_gronwall_estimate} to estimate
	\begin{equation}
		\begin{split}
			&\left|\phi^1_{\widetilde F}(z)-z\right|\geq\left|\phi^1_F(z)-z\right|-\left|\phi^1_{\widetilde F}(z)-\phi^1_F(z)\right|\geq\\
			&\geq \nu\left(|z|\right)|z|-e^{3C_2}\left[\bar\sigma^f_1\left(\sqrt{R_a}\right)\sqrt{R_b}+2\bar\sigma^f_0\left(\sqrt{R_a}\right)R^{\frac32}_b\left\|\rho'\right\|_{L^\infty}\right].
		\end{split}
	\end{equation}
	Plugging in \eqref{eq:truncation_lemma_below_estimate_cst_unitary}, using $|z|\geq\sqrt{R_a}$ and that $\nu(|z|)$ is a non-decreasing function,
	\begin{equation}
		\begin{multlined}[c]
			\left|\phi^1_{\widetilde F}(z)-z\right|\geq \nu_\infty\sqrt{R_a}-e^{C_2}\bar\sigma^f_1\left(\sqrt{R_a}\right)\sqrt{R_a}+\\
			-e^{3C_2}\left[\bar\sigma^f_1\left(\sqrt{R_a}\right)\sqrt{R_b}+2\bar\sigma^f_0\left(\sqrt{R_a}\right)R^{\frac32}_b\left\|\rho'\right\|_{L^\infty}\right]=\\
			\geq \nu_\infty\sqrt{R_a}-e^{3C_2}\sqrt{R_b}\left[\left(1+e^{-2C_2}\frac{\sqrt{R_a}}{\sqrt{R_b}}\right)\bar\sigma^f_1\left(\sqrt{R_a}\right)+2R_b\left\|\rho'\right\|_{L^\infty}\bar\sigma^f_0\left(\sqrt{R_a}\right)\right].
		\end{multlined}
	\end{equation}
	We can rule out fixed points in the truncation annulus by imposing that this is positive. We are thus led to the inequality
	\begin{equation}\label{eq:truncation_fundamental_estimate}
		\frac{\sqrt{R_b}}{\sqrt{R_a}}\left[\left(1+e^{-2C_2}\frac{\sqrt{R_a}}{\sqrt{R_b}}\right)\bar\sigma^f_1\left(\sqrt{R_a}\right)+2R_b\left\|\rho'\right\|_{L^\infty}\bar\sigma^f_0\left(\sqrt{R_a}\right)\right]<e^{-3C_2}\nu_\infty.
	\end{equation}
	We now aim to find $R_a,R_b$ and $\rho$ so that the inequality can be achieved for all large enough radii. Define the piece-wise linear function
	\begin{equation}\label{eq:truncation_piecewiserho}
		\tilde\rho(r)=
		\begin{cases}
			1, & r\leq R_a,\\
			\frac{r-R_b}{R_a-R_b}, & R_a\leq r\leq R_b,\\
			0, & r\geq R_b.
		\end{cases}
	\end{equation}
	We can smooth $\tilde\rho$ out to a smooth function $\rho$ that satisfies the properties of \eqref{eq:trunclem_step_props} and additionally
	\begin{equation}
		\left|\rho'(r)\right|\leq\frac{c_\rho}{R_b-R_a}\quad\forall r\in[0,\infty).
	\end{equation}
	where $c_\rho>1$ can be chosen arbitrarily close to $1$, depending on how tightly we are smoothing the piecewise-linear function. A convenient choice is $c_\rho=\sqrt{2}$.
	It is useful to eliminate a parameter by setting $R_b=2R_a$ so that the norm of $\rho'$ decays with $R_a$:
	\begin{equation}
		\left\|\rho'\right\|_{L^\infty}\leq\frac{c_\rho}{R_a}.
	\end{equation}
	With these choices the left-hand side of \eqref{eq:truncation_fundamental_estimate} becomes
	\begin{equation}\label{eq:def_of_sigma01}
		\begin{split}
			&\sqrt2\left[\left(1+\frac{e^{-2C_2}}{\sqrt 2}\right)\bar\sigma^f_1\left(\sqrt{R_a}\right)+4 R_a\left\|\rho'\right\|_{L^\infty}\bar\sigma^f_0\left(\sqrt{R_a}\right)\right]\leq\\
			&\leq\left(\sqrt2+e^{-2C_2}\right)\bar\sigma^f_1\left(\sqrt{R_a}\right)+4\sqrt2 c_\rho\bar\sigma^f_0\left(\sqrt{R_a}\right)=\bar\sigma^f_{01}\left(\sqrt{R_a}\right).
		\end{split}
	\end{equation}
	Using that $F$ is asymptotically quadratic, more precisely the estimates in equation \eqref{eq:decay_moduli_decayish}, we see that $\bar\sigma^f_{01}(R)\to 0$ as $R\to\infty$. Therefore the strict inequality \eqref{eq:truncation_fundamental_estimate} can be reached for $R_a$ large enough. More precisely, set
	\begin{equation}\label{eq:def_of_tildeR_1}
		\tilde{R}_1=\max\left\{R:\bar\sigma^f_{01}\left(R\right)\geq e^{-3C_2}\nu_\infty\right\}.
	\end{equation}
	Then as soon as $\sqrt{R_a}>\max\{R_1,\tilde{R}_1\}$, $\phi^1_{\widetilde F}$ has no fixed points in the annular region $\{\sqrt{R_a}\leq |z|\leq\sqrt{2R_a}\}$, which in turn means that the fixed points of $\phi^1_{\widetilde F}$ are exactly the same as the fixed points of $\phi^1_F$.
\end{proof}
\begin{rmk}
	By scrutinizing the definition of $R_1$ \eqref{eq:fund_radius_unitary} and of $\tilde{R}_1$ \eqref{eq:def_of_tildeR_1}, it's easy to see that typically $\tilde{R}_1\gg R_1$.
\end{rmk}

\section{The Poincaré-Birkhoff theorem for rapidly asymptotically unitary Hamiltonian systems}
\label{sec:the_poincare_birkhoff_theorem_for_rapidly_asymptotically_unitary_hamiltonian_systems}

The proof of the Poincaré-Birkhoff theorem for rapidly asymptotically unitary Hamiltonian diffeomorphisms is roughly divided in two parts. First, there is a preparatory step, which involves applying the constructions of section \ref{sec:modifying_linear_and_asymptotically_linear_hamiltonian_systems}. The output of this preparatory step are two sequences of Hamiltonians and a special sequence of prime iterates, whose r\^ole we will explain below, and whose properties are summarized in proposition \ref{prop:auxiliary_prop}

The results of proposition \ref{prop:auxiliary_prop} are then used to prove the Poincaré-Birkhoff theorem \ref{thm:PB_thm}, adapting an argument of Gürel \cite{Gurel_PJM2014}. The argument relies on an asymptotic study of continuation maps between the filtered Floer homologies of the Hamiltonians in these two sequences. We point the unfamiliar reader to section \ref{sec:floer_homology_of_asymptotically_linear_hamiltonian_systems} for remarks on the construction of Floer homology for ALHDs in its global, filtered and local version, and references to literature on the topic of Floer homology.

The presentation of the proof found here is reversed, first stating proposition \ref{prop:auxiliary_prop} and immediately using it to prove the theorem, and then proving the proposition in the following section \ref{sec:proof_of_the_auxiliary_proposition}. The rationale for this inversion is to show as quickly as possible how the constructions of the proposition are used in the proof of the Poincaré-Birkhoff theorem. Nevertheless, let us introduce and give some motivation to the objects appearing in the auxiliary proposition.

Let $\upphi$ be an asymptotically unitary Hamiltonian diffeomorphism with linear map at infinity $\upphi_\infty\in\Un(n)$.
In the proposition below we introduce a sequence of prime numbers $(p_j)_{j\in\bN}$ and, for a fixed $m\in\bN$, two sequences of Hamiltonians $(G_j)_{j\in\bN}$ and $(F_{j,m})_{j\in\bN}$ which we will then use in the proof of the Poincaré-Birkhoff theorem.

The sequence of prime numbers depends only on $\upphi_\infty$, and it has two crucial properties: the map $\phi^{p_j}_\infty$ is uniformly far from having a resonance, and the gaps $p_{j+m}-p_j$ for fixed $m$ grow much slower than $p_j$ as $j\to\infty$. These two properties are found by applying Vinogradov's equidistribution theorem on prime multiples of irrational numbers mod 1 \cite[Chapter XI]{Vinogradov1954_TrigSums}, see proposition \ref{prop:uniform_nonresonant_iterations_of_linear_symplectic_map}.

The sequences of Hamiltonians are constructed out of a chosen generating Hamiltonian $H$ for the asymptotically unitary Hamiltonian diffeomorphism $\upphi$. The Hamiltonian $G_j$ is related to the $p_j$th iterate of $H$ while $F_{j,m}$ to the $p_{j+m}$th, but cannot be chosen to be simply the iterates of $H$. The reason is that in the proof of the Poincaré-Birkhoff theorem, we must be able to find continuation isomorphisms between the filtered Floer homologies of Hamiltonians generating these iterates, with an estimate on their action shift in terms of the iterates. There are no continuation isomorphisms between the filtered Floer homologies of $H^{\times p_j}$ and $H^{\times p_{j+m}}$ in general. Therefore, one must suitably modify the iterated Hamiltonians at infinity, as we explain next.

The Hamiltonian $G_j$ is obtained by truncation of $0\wedge H^{\times p_j}$. The Hamiltonian $F_{j,m}$ is obtained by truncation of $H^{p_{j+m}\ominus p_j}$ and then by interpolating the linear system at infinity of this truncated Hamiltonian to be the same as the one for $G_j$. This allows continuation isomorphisms to be defined, as the systems at infinity of the two Hamiltonians are the same.
The resulting Hamiltonians satisfy, among other properties, that $\Fix\phi^1_{G_j}=\Fix\upphi^{p_j}$ while $\Fix\phi^1_{F_{j,m}}=\Fix\upphi^{p_{j+m}}$, and $F_{j,m}-G_j$ is a function of support contained in a large ball $B_{j,m}$, with explicit estimates on both the radius of this ball and the $L^\infty$-norm of this function.

In the proof of the Poincaré-Birkhoff theorem, we must understand the asymptotics of these estimates as $j\to\infty$, because the action shift of the continuation morphism between $F_{j,m}$ and $G_j$ depends on these quantities. This boils down to estimating the radii of truncation and interpolation chosen to construct $F_{j,m}$ and $G_j$. In turn, these radii are determined by how rapidly the tails of the sub-quadratic part of the Hamiltonian $H$ behave. In order to gain control on the radii for increasingly large iterates, we must introduce a restricted class of asymptotically quadratic Hamiltonians, which have fast decay of the sub-quadratic part.

\subsection{Rapidly asymptotically quadratic Hamiltonians}
\label{sub:rapidly_asymptotically_quadratic_hamiltonians}

Let $H=Q+h$ be an asymptotically quadratic Hamiltonian and $\sigma^h_0,\sigma^h_1$ the tail functions defined in \eqref{eq:def_of_decay_moduli}.
\begin{defn}\label{def:rapidly_asy_quad}
	We say that $H$ is \emph{rapidly} asymptotically quadratic if
	\begin{equation}\label{eq:def_of_rapidly_asy_quad}
		\sigma^h_1(R)=o\left(e^{-R^2}\right),\quad\sigma^h_0(R)=o\left(e^{-R^2}\right)\quad\text{as}\quad R\to\infty.
	\end{equation}
	We say that an ALHD $\upphi$ is \emph{rapidly} asymptotically linear if it can be generated by a rapidly asymptotically quadratic Hamiltonian.
\end{defn}
\begin{rmk}
	The sub-quadratic term of a rapidly asymptotically quadratic Hamiltonian decays to zero super-exponentially. Nevertheless, it doesn't have to have compact support, so for example it can be an analytic function with rapid enough decay.
\end{rmk}
The exact point where such condition enters is in lemma \ref{lem:estimate_truncation_radius_k}, where it is shown that with this fast type of decay, the truncation radius of $H^{k\ominus l}$ grows much slower than $\sqrt k$.

\subsection{The auxiliary proposition}
\label{sub:the_auxiliary_proposition}

If $z_0\in\Fix\phi^1_H$, denote by $\HFloc_*\left(H,z_0\right)$ the local Floer homology of the fixed point with respect to the generating Hamiltonian $H$.
The reader unfamiliar with this object may consult section \ref{sub:local_floer_homology} for some of its properties and some references to the literature.
\begin{prop}\label{prop:auxiliary_prop}
	Let $\upphi$ be a rapidly asymptotically linear Hamiltonian diffeomorphism whose linear map at infinity $\upphi_\infty$ is non-degenerate and unitary.
	\begin{enumerate}[label=(\alph*)]
		\item\label{item:mainprop_unif_nonres_prime_iterates} There exists an increasing sequence of prime numbers $(p_j)_{j\in\bN}$ and a constant $\nu_\infty>0$ such that
			\begin{enumerate}[label=(\arabic*)]
				\item For every $j$, it holds that
						\begin{equation}\label{eq:mainprop_ndg_at_infty_unif_inv_est}
							\left|\left(\upphi^{p_j}_\infty-\bid\right)^{-1}\right|^{-1}\geq \nu_\infty.
						\end{equation}
						In particular $\det\left(\upphi^{p_j}_\infty-\bid\right)\neq 0$ for all $j$.\label{item:mainprop_ndg_at_infty}
				\item If $m\in\bN$ is fixed, $p_{j+m}-p_j=o(p_j)$. \label{item:mainprop_unif_nonres_gaps}
			\end{enumerate}
		\item\label{item:mainprop_H_and_Gj} There exists a rapidly asymptotically quadratic Hamiltonian $H$ with $\upphi=\phi^1_H$, a sequence of asymptotically quadratic Hamiltonians $\left(G_j\right)_{j\in\bN}$, and for every fixed $m\in\bN$ a sequence of asymptotically quadratic Hamiltonians $\left(F_{j,m}\right)_{j\in\bN}$ such that, denoting
	\begin{equation}\label{eq:mainprop_def_of_sigma_jm}
		\sigma_{j,m}=\ind_\infty H^{\times p_{j+m}}-\ind_\infty H^{\times p_j},
	\end{equation}
	we have the following properties:
	\begin{enumerate}[label=(\arabic*)]
		\setcounter{enumii}{2}
		\item $\left|\left(p_{j+m}-p_j\right)\overline\ind_\infty H-\sigma_{j,m}\right|\leq 2n$. \label{item:mainprop_index_shift_estimate}
		\item It holds that $\Fix\phi^1_{F_{j,m}}=\Fix\upphi^{p_{j+m}}$. Moreover, for any $z\in\Fix\upphi\subset\Fix\phi^1_{F_{j,m}}$, we have
			\begin{equation}\label{eq:Fjm_action_meanCZ_CZ_iterated_fixed_pts}
				\begin{split}
					&\cA_{F_{j,m}}(z)=p_{j+m}\cA_{H}(z),\\
					&\CZ\left(z,F_{j,m}\right)=\CZ\left(z,H^{\times p_{j+m}}\right)-\sigma_{j,m}
				\end{split}
			\end{equation}
			and, if $z\in\Fix\phi^1_{F_{j,m}}$ is isolated as a fixed point of $\phi^1_{F_{j,m}}$, then
			\begin{equation}\label{eq:Fjm_iterated_fixed_pt_HFloc_shifted}
					\HFloc_*\left(F_{j,m},z\right)\cong\HFloc_{*+\sigma_{j,m}}\left(H^{\times p_{j+m}},z\right).
			\end{equation} \label{item:mainprop_Fjm_dynquants}
		\item It holds that $\Fix\phi^1_{G_j}=\Fix\upphi^{p_j}$. Moreover, for any $z\in\Fix\upphi\subset\Fix\phi^1_{G_j}$, we have
			\begin{equation}\label{eq:Gj_fixed_pts_dynamical_quantities}
				\begin{split}
					&\cA_{G_j}(z)=p_j\cA_H(z),\\
					&\meanCZ\left(z,G_j\right)=p_j\meanCZ\left(z,H\right),\\
					&\CZ\left(z,G_j\right)=\CZ(z,H^{\times p_j}),\\
				\end{split}
			\end{equation}
			and, if $z\in\Fix\phi^1_{G_j}$ is isolated as a fixed point of $\phi^1_{G_j}$, then
			\begin{equation}\label{eq:Gj_iterated_fixed_pt_HFloc_preserved}
					\HFloc_*\left(G_j,z\right)\cong\HFloc_{*}\left(H^{\times p_{j}},z\right).
			\end{equation} \label{item:mainprop_Gj_dynquants}
		\item $G_j$ and $F_{j,m}$ are equal outside a compact set $B_{j,m}$ depending on $j,m$. We have the estimate
			\begin{equation}\label{eq:action_shift_estimate}
				\left\|G_j-F_{j,m}\right\|_{L^\infty\left(S^1\times B_{j,m}\right)}=o\left(p_{j+m}\right)\text{ as } j\to\infty.
			\end{equation} \label{item:mainprop_action_shift_estimate}
		\end{enumerate}
	\end{enumerate}
\end{prop}
The proof of this proposition is deferred to section \ref{sec:proof_of_the_auxiliary_proposition}.

\subsection{Proof of the Poincaré-Birkhoff theorem}
\label{sub:proof_of_the_poincare_birkhoff_theorem}

In this section we prove the Poincaré-Birkhoff theorem for rapidly asymptotically unitary Hamiltonian diffeomorphisms using proposition \ref{prop:auxiliary_prop}.

Let $\upphi$ be an ALHD. Recall that a fixed point $z_0\in\Fix\upphi$ is said to be \emph{twist}, when $\meanCZ(z_0,H)\neq\overline\ind_\infty H$ for some generating asymptotically quadratic Hamiltonian $H$, and that if it is isolated, it is said to be homologically visible when $\HFloc_*(H,z_0)\neq\{0\}$ for some generating asymptotically quadratic Hamiltonian $H$. Both these conditions do not depend on the choice of generating Hamiltonian.

\setcounter{thm}{0}
\begin{thm}\label{thm:PB_thm}
	Let $\upphi$ be a rapidly asymptotically linear Hamiltonian diffeomorphism with non-degenerate and unitary linear map at infinity $\upphi_\infty\in\Un(n)\subset\Sp(2n)$. Assume that $\upphi$ admits an isolated homologically visible twist fixed point $z_0\in\Fix\upphi$. Then $\upphi$ has infinitely many fixed points or infinitely many periodic points with increasing primitive period.
\end{thm}
\begin{proof}
	The strategy of the proof extends Gürel's \cite{Gurel_PJM2014}. Assume by contradiction that $\upphi$ has finitely many fixed points and finitely many integers appear as primitive periods of periodic points of $\upphi$. Since $\upphi$ has finitely many fixed points, the action value $a_0=\cA_H(z_0)$ of the 1-periodic orbit corresponding to the fixed point $z_0$ is isolated. Therefore there exists an $\epsilon>0$ such that $(a_0-\epsilon,a_0+\epsilon]$ contains only $a_0$ as critical value of the action functional.
	
	Now, we apply proposition \ref{prop:auxiliary_prop} to $\upphi$. We obtain a sequence of prime numbers $\left(p_j\right)_{j\in\bN}$, a sequence of Hamiltonians $\left(G_j\right)_j\in\bN$, and for $m\in\bN$ fixed arbitrarily a sequence of Hamiltonians $\left(F_{j,m}\right)_{j\in\bN}$. Recall from point \ref{item:mainprop_H_and_Gj} of proposition \ref{prop:auxiliary_prop} that the fixed points of $\phi^1_{F_{j,m}}$ are exactly all the $p_{j+m}$-periodic points of $\upphi$, and the fixed points of $\phi^1_{G_j}$ are exactly all the $p_j$-periodic points of $\upphi$. We can assume that $p_0$ is larger than the largest primitive period of any periodic point of $\phi$. It follows that $p_j$-periodic points of $\upphi$ are iterated fixed points of $\upphi$ for all $j$. In particular, we have  $\Fix\phi^1_{F_{j,m}}=\Fix\upphi$ and also $\Fix\phi^1_{G_j}=\Fix\upphi$.

	Since there are only finitely many fixed points, the critical value $p_{j+m}a_0$ of $\cA_{F_{j,m}}$ is isolated. In fact, if $\bS\left(F_{j,m}\right)\subset\bR$ denotes the set of critical values of $\cA_{F_{j,m}}$,
	\begin{equation}
		\bS\left(F_{j,m}\right)\cap\left(p_{j+m}\left(a_0-\epsilon\right),p_{j+m}\left(a_0+\epsilon\right)\right]=\left\{p_{j+m}a_0\right\}.
	\end{equation}

	Now we consider the filtered Floer homologies (see section \ref{sub:filtered_floer_homology}) of the Hamiltonians $F_{j,m}$ and $G_j$. The plan is to first analyze the filtered Floer homology of $F_{j,m}$ in action windows around the critical value $p_{j+m}a_0$ and identify action windows for which a certain inclusion-quotient morphism (see section \ref{ssub:inclusion_quotient_morphism} for the definition) is non-vanishing. Then we study the continuation morphism to the filtered Floer homology of $G_j$, in particular, the shift on the action filtration it induces (see section \ref{sub:action_shift_of_continuation_morphisms}, equations \eqref{eq:def_of_action_shift_constant} and \eqref{eq:contin_with_action_shift}).

	Set $C_{j,m}=\left\|F_{j,m}-G_j\right\|_{L^\infty}$. Since $C_{j,m}=o(p_{j+m})$ by point \ref{item:mainprop_action_shift_estimate} of proposition \ref{prop:auxiliary_prop}, we know that for any $\delta>0$ there exists a $j_0>0$ such that $\delta p_{j+m}>\delta p_j>6C_{j,m}$ for every $j>j_0$. So take $\delta=\epsilon$ as above and set
	\begin{equation}\label{eq:special_action_window}
		I=\left(p_{j+m}\left(a_0-\frac{\epsilon}{3}\right),p_{j+m}\left(a_0+\frac{\epsilon}{3}\right)\right]
	\end{equation}
	\begin{figure}[t]
		\centering
		\begin{overpic}[width=.85\textwidth]{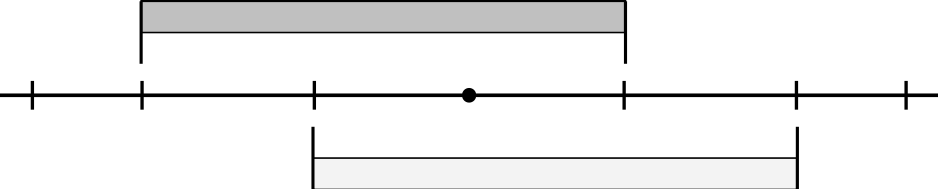}
			\put (-5,6) {$p_{j+m}(a_0-\epsilon)$}
			\put (14,6) {$a_-$}
			\put (25,13) {$a_-+2C_{j,m}$}
			\put (45,6) {$p_{j+m}a_0$}
			\put (65,6) {$a_+$}
			\put (78,13) {$a_++2C_{j,m}$}
			\put (90,5) {$p_{j+m}\left(a_0+\epsilon\right)$}
			\put (40,17.5) {$I$}
			\put (55,.5) {$I+2C_{j,m}$}
		\end{overpic}
		\caption{Special action windows. Here $a_{\pm}=p_{j+m}\left(a_0\pm\frac{\epsilon}{3}\right)$}
		\label{fig:action_windows}
	\end{figure}
	Then for any $j>j_0$ we have (see figure \ref{fig:action_windows})
	\begin{equation}\label{eq:special_action_window_is_good}
		\begin{split}
			I\cup\left(I+C_{j,m}\right)\cup\left(I+2C_{j,m}\right)\subset\left(p_{j+m}(a_0-\epsilon),p_{j+m}(a_0+\epsilon)\right]\\
			p_{j+m}a_0\in I\cap\left(I+C_{j,m}\right)\cap\left(I+2C_{j,m}\right)
		\end{split}
	\end{equation}
	We can use local Floer homology to compute the filtered Floer homology in these action windows. Indeed, since $p_{j+m}a_0$ is an isolated critical value and the $1$-periodic orbits of $F_{j,m}$ are finite in number, we can use lemma \ref{lem:local_to_nbh_HF} for $J=I$ or $J=I+2C_{j,m}$ to conclude that
	\begin{equation}\label{eq:calc_of_HFFjm_around_J}
			\HF^J_*\left(F_{j,m}\right)\cong\bigoplus\left\{\HFloc_*\left(F_{j,m},z\right):z\in\Fix\phi^1_{F_{j,m}}=\Fix\phi,\ \cA_{F_{j,m}}(z)=p_{j+m}a_0\right\}.
	\end{equation}
	
	We want to find non-zero classes in this filtered homology. To do this, we will use a theorem of Ginzburg and Gürel \cite[Theorem 1.1]{GinzburgGuerel2010_LFHAG} on the iterated local Floer homology to show that the local Floer homology of $z_0$ contributes non-trivially to the direct sum \eqref{eq:calc_of_HFFjm_around_J}. Here are the details.
	
	We denote the set of degrees in which the local Floer homology is not trivial by $\deg\supp\HFloc_*$ (see equation \eqref{eq:def_of_degsupp} in section \ref{sub:local_floer_homology}).
	As stated in lemma \ref{lem:degsupp_of_HFloc}, $\deg\supp\HFloc_*$ is contained in an interval of length $2n$ around the mean Conley-Zehnder index of the fixed point. This implies that as we take $j>j_0$ large enough, the supports in degree of the local Floer homology summands corresponding to fixed points with different mean Conley-Zehnder index appearing in \eqref{eq:calc_of_HFFjm_around_J} become disjoint. This lets us conclude that for every $m$ and all $j$ large enough, for $J=I$ or $J=I+2C_{j,m}$,
	\begin{equation}\label{eq:calc_of_HFFjm_around_I_degreeloc}
		\begin{split}
			\HF^J_s\left(F_{j,m}\right)=\bigoplus\left\{\HFloc_s\left(F_{j,m},z\right):
				\begin{split}
					&z\in\Fix\upphi,\\ &\cA_{F_{j,m}}(z)=p_{j+m}a_0,\\ &\meanCZ\left(z,H\right)=\meanCZ\left(z_0,H\right)
				\end{split}\right\}
		\end{split}
	\end{equation}
	for all $s\in\deg\supp\HFloc\left(F_{j,m},z_0\right)$. Thus we look for non-zero classes in these degrees.

	Without loss of generality, we can assume that $p_0$ is larger than the order of any of root of unity in the spectrum of $d\upphi(z_0)$. In particular, $p_{j+m}$ is never a multiple of the order of a root of unity in the spectrum of $d\upphi(z_0)$.  Since $z_0$ is homologically visible, we can apply \cite[Theorem 1.1]{GinzburgGuerel2010_LFHAG} (see lemma \ref{lem:HFloc_iterated} for the exact statement we need here) to conclude that
	\begin{equation}
		\HFloc_{*}\left(H^{\times p_{j+m}},z_0\right)\neq\{0\}.
	\end{equation}
	for all fixed $m\in\bN$ and all $j$.
	Recall from equation \eqref{eq:Fjm_iterated_fixed_pt_HFloc_shifted} in point \ref{item:mainprop_Fjm_dynquants} in proposition \ref{prop:auxiliary_prop} that
	\begin{equation}
		\HFloc_*\left(F_{j,m},z_0\right)\cong\HFloc_{*+\sigma_{j,m}}\left(H^{\times p_{j+m}},z_0\right).
	\end{equation}
	In particular there are degrees where the local Floer homology of $F_{j,m}$ at $z_0$ is not trivial. Apply lemma \ref{lem:degsupp_of_HFloc} to the local homology of $H^{\times p_{j+m}}$ at $z_0$ and use this isomorphism to find the support in degree of $\HFloc(F_{j,m},z_0)$:
	\begin{equation}\label{eq:calc_of_degsupp_HFlocFjmz0}
		\begin{multlined}
			\varnothing\neq\deg\supp\HFloc\left(F_{j,m},z_0\right)=\deg\supp\HFloc\left(H^{\times p_{j+m}},z_0\right)-\sigma_{j,m}\subseteq\\
			\subseteq\left[p_{j+m}\meanCZ\left(z_0,H\right)-\sigma_{j,m}-n,p_{j+m}\meanCZ\left(z_0,H\right)-\sigma_{j,m}+n\right]=\Delta(z_0,j,m).
		\end{multlined}
	\end{equation}
	Applying this to the calculation \eqref{eq:calc_of_HFFjm_around_I_degreeloc}, we see that when $J=I$ or $J=I+2C_{j,m}$,
	\begin{equation}\label{eq:HFII2C_neq_zero}
		\exists s\in\Delta\left(z_0,j,m\right)\text{ such that }\{0\}\neq \HF^J_s\left(F_{j,m}\right)=\HFloc_s\left(F_{j,m},z_0\right)\oplus\cdots
	\end{equation}
	Consider the inclusion-quotient morphism from action window $I$ to action window $I+2C_{j,m}$, defined in section \ref{ssub:inclusion_quotient_morphism}, equation \eqref{eq:def_of_iq_morphism}:
	\begin{equation}\label{eq:iq_I-I+2C}
		\Phi:\HF^I_*\left(F_{j,m}\right)\to\HF^{I+2C_{j,m}}_*\left(F_{j,m}\right).
	\end{equation}
	By lemma \ref{lem:iq_commutes_with_HFloc_inclusion}, the classes corresponding to $\HFloc_*\left(F_{j,m},z_0\right)$ are sent to themselves under $\Phi$. Since there are non-zero classes in this local Floer homology, which appears as a summand in both filtered homologies, we conclude
	\begin{equation}\label{eq:iq_I-I+2C_neq_0}
		\exists s\in\Delta\left(z_0,j,m\right):\Phi\neq 0.
	\end{equation}

	Now, by point \ref{item:mainprop_action_shift_estimate} in proposition \ref{prop:auxiliary_prop}, $F_{j,m}-G_j$ is compactly supported in a ball $B_{j,m}$. Therefore, there are continuation morphisms between their filtered Floer homologies with action shift $C_{j,m}=\left\|F_{j,m}-G_j\right\|_{L^\infty}$ (section \ref{sub:action_shift_of_continuation_morphisms}):
	\begin{equation}\label{eq:continuations_Fjm_Gj}
		\begin{cases}
			\cC&\colon\HF^J_*\left(F_{j,m}\right)\to\HF^{J+C_{j,m}}_*\left(G_j\right),\\
			\overline\cC&\colon\HF^J_*\left(G_j\right)\to\HF^{J+C_{j,m}}_*\left(F_{j,m}\right)
		\end{cases}
	\end{equation}
	The inclusion quotient morphism is factored by the composition of these morphisms, see lemma \ref{lem:contin_forth_and_back_factors_IQ}:
	\begin{equation}\label{eq:crucial_triangle}
		\begin{tikzcd}
			&\HF^{I+C_{j,m}}_*\left(G_j\right)\ar[dr,bend left,"\overline\cC"]&\\
			\HF^I_*\left(F_{j,m}\right)\ar[rr,"\Phi"]\ar[ru,bend left,"\cC"]& &\HF^{I+2C_{j,m}}_*\left(F_{j,m}\right)
		\end{tikzcd}
	\end{equation}
	The contradiction we are seeking arises from this factorization. The idea is to compare the degrees where $\Phi$ is certainly non-zero with the degrees where $\cC,\overline\cC$ are both certainly zero.

	As observed above \eqref{eq:iq_I-I+2C_neq_0}, $\Phi$ is certainly non-zero for some degrees $s\in\Delta(z_0,j,m)$ for all $m$ and all $j$ large enough. On the other hand, $\cC$ and $\overline\cC$ are non-zero only in the degrees corresponding to the supports in degree of local Floer homologies of fixed points of $G_j$. Using again lemma \ref{lem:degsupp_of_HFloc} and the calculations of the mean index of iterated fixed points with $G_j$ found in equation \eqref{eq:Gj_fixed_pts_dynamical_quantities}, point \ref{item:mainprop_Gj_dynquants} of proposition \ref{prop:auxiliary_prop},
	\begin{equation}\label{eq:degsuppHFlocGj}
		\deg\supp\HFloc_*\left(G_j,z\right)\subseteq\left[p_j\meanCZ(z,H)-n,p_j\meanCZ(z,H)+n\right]=\Gamma(z,j).
	\end{equation}
	There are two cases: either $z\in\Fix\upphi$ is such that $\meanCZ(z,H)=\meanCZ(z_0,H)$, or not. In the first case, we use the twist condition. If $\Delta(z_0,j,m)\cap\Gamma(z,j)\neq\varnothing$, then
	\begin{equation}\label{eq:degree_calc_case_CZzeqCZz0}
		\left|\left(p_{j+m}-p_j\right)\meanCZ(z_0,H)-\sigma_{j,m}\right|\leq 2n
	\end{equation}
	Combining this with point \ref{item:mainprop_index_shift_estimate} in proposition \ref{prop:auxiliary_prop} we have
	\begin{equation}
		p_{j+m}-p_j\leq\frac{3n}{\left|\meanCZ(z_0)-\overline\ind_\infty H\right|}
	\end{equation}
	But $p_{j+m}-p_j>2m$ for all $j$, so there exists an $m\in\bN$ such that $\Delta(z_0,j,m)$ and $\Gamma(z,j)$ are disjoint for every $j$.

	Assume instead that $\meanCZ(z,H)\neq\meanCZ(z_0,H)$. We have
	\begin{equation}\label{eq:degree_calc_case_CZzneqCZz0_1}
			\Delta(z_0,j,m)\cap\Gamma(z,j)\neq\varnothing\implies \left|p_{j+m}\meanCZ(z_0,H)-\sigma_{j,m}-p_j\meanCZ(z,H)\right|\leq 2n.
	\end{equation}
	Notice that point \ref{item:mainprop_unif_nonres_gaps} in proposition \ref{prop:auxiliary_prop} implies that $p_{j+m}/p_j\to 1$ and point \ref{item:mainprop_index_shift_estimate} that $\sigma_{j,m}/p_j\to0$. Therefore dividing everything by $p_j$ we see
	\begin{equation}\label{eq:degree_calc_case_CZzneqCZz0_2}
		\frac{1}{p_j}\left|p_{j+m}\meanCZ(z_0,H)-\sigma_{j,m}-p_j\meanCZ(z,H)\right|\xrightarrow[j\to\infty]{}\left|\meanCZ(z_0,H)-\meanCZ(z,H)\right|>0
	\end{equation}
	But since $p_j$ is unbounded in $j$, also $\left|p_{j+m}\meanCZ(z_0,H)-\sigma_{j,m}-p_j\meanCZ(z,H)\right|$ must be, and hence eventually larger than $2n$. So for every $m\in\bN$ there exists a $j$ large enough such that $\Delta(z_0,j,m)\cap\Gamma(z,j)=\varnothing$.

	This is a contradiction, because in both cases we found that there exists an $m$ such that for every $j$ large enough, $\Phi=\overline\cC\circ\cC$ is non-zero for some degrees in $\Delta(z_0,j,m)$ while $\cC$ and $\overline\cC$ are both zero for all degrees in $\Delta(z_0,j,m)$.
\end{proof}

\section{Proof of the auxiliary proposition}
\label{sec:proof_of_the_auxiliary_proposition}

\subsection{Point \protect\ref{item:mainprop_unif_nonres_prime_iterates}}
We construct a sequence $(p_j)_{j\in\bN}$ of special prime iterates for any non-degenerate linear symplectomorphism. This result is of independent interest and uses the Vinogradov equidistribution theorem on prime iterates of irrational numbers modulo 1. We first recall some elementary facts about equidistribution modulo 1.

Recall that an \emph{elementary} set in $\bR^q$ is a finite union of $q$-dimensional parallelepids. Elementary sets have an obviously defined $q$-volume. Recall that the outer volume of a set $C\subset\bR^q$ is the infimum of the $q$-volumes of elementary sets containing $C$, while the inner volume is the supremum of the $q$-volumes of elementary sets contained in $C$. A subset is Jordan-measurable if its inner and outer volumes coincide, so that it has a well defined $q$-volume. The $q$-volume of a Jordan-measurable set $C\subset\bR^q$ will be denoted by $|C|$. It is the same thing as its Lebesgue measure, but there are Lebesgue measurable sets which are not Jordan-measurable.

We give a definition from \cite[Chapter 1,\S6]{KuipersNiederreiter1974_UnifDist}. If $v\in\bR^q$, we denote by $\{v\}\subset[0,1)^q$ the vector given by the fractional parts of the coordinates of $v$. Let $(v_j)\subset\bR^q$ be a sequence in $\bR^q$. We say that $(v_j)$ is \emph{equidistributed} modulo 1 if for every Jordan-measurable set $C$ one has
\begin{equation}\label{eq:equidistribution_def_Jsets}
	\lim_{N\to\infty}\frac{\#\left\{x_j:j\leq N,\ x_j\in C\right\}}{N}=|C|
\end{equation}

It is not hard to see \cite[Theorem 6.1]{KuipersNiederreiter1974_UnifDist} that a sequence $(v_j)\subset\bR^q$ is equidistributed mod 1 if and only if for every \emph{continuous} function $f\colon[0,1]^q\to\bR$ we have
\begin{equation}\label{eq:equidistribution_def_fctns}
	\lim_{N\to\infty}\frac1N\sum_{j=1}^Nf\left(\left\{v_j\right\}\right)=\int_{[0,1]^q}fdx.
\end{equation}

We need the following equidistribution statement, which is a simple consequence of a theorem of Weyl combined with a theorem of Vinogradov.
\begin{lemma}\label{lem:prime_mults_irrat_vec}
	Let $a_1,\dots,a_q\in\bR\setminus\bQ$ be rationally independent irrational numbers. Let $(P_j)_j\subset\bZ$ be the sequence of prime numbers indexed increasingly. Define $v_j=(P_ja_1,\dots,P_ja_q)\in\bR^q$. Then $(v_j)\subset\bR^q$ is equidistributed mod 1.
\end{lemma}
\begin{proof}
	Recall that $a_1,\dots,a_q\in\bR$ are rationally independent irrational numbers if the set $\{1,a_1,\dots,a_q\}$ spans a $\bQ$-subspace of $\bR$ of dimension $q$. Denote by $\vec a=(a_1,\dots,a_q)\in\bR^q$ and notice $v_j=P_j\vec a$.
	We shall make use of a theorem of Weyl, which can be found in \cite[Theorem 6.3]{KuipersNiederreiter1974_UnifDist}. It states that a sequence $(w_j)\subset\bR^q$ is equidistributed mod 1 if and only if for every non-zero lattice point $h\in\bZ^q\subset\bR^q$ the sequence $\left(\langle w_j,h\rangle\right)\subset \bR$ is equidistributed mod 1. If $h\neq 0$ is such a lattice point, we have $\left\langle h,v_j\right\rangle=P_j\left\langle h,\vec a\right\rangle$. Now, since $a_1,\dots,a_q$ are rationally independent irrational numbers, the number $\alpha=\left\langle h,\vec a\right\rangle$ is irrational. Therefore we are reduced to studying the sequence $(P_j\alpha)\subset\bR$ with $\alpha\in\bR\setminus\bQ$ irrational. We conclude using a theorem of Vinogradov \cite[Chapter XI]{Vinogradov1954_TrigSums}, which states that if $\alpha\in\bR\setminus\bQ$ is an irrational number, then the sequence $(P_j\alpha)\subset\bR$ is equidistributed mod 1.
\end{proof}

\begin{prop}\label{prop:uniform_nonresonant_iterations_of_linear_symplectic_map}
	Let $\upphi_\infty\in\Sp(2n)$ be a linear symplectic map with $\det\left(\upphi_\infty-\bI\right)\neq 0$. Then there exists a $\nu_\infty>0$ and an increasing sequence $(p_j)_{j\in\bN}$ of prime numbers such that:
	\begin{enumerate}
		\item The spectrum of the $p_j$th iterate of $\upphi_\infty$ remains uniformly away from $1$:
			\begin{equation}\label{eq:unif_nonres_iters_stay_away_from_eigenvalue_1}
				\left|\left(\upphi_\infty^{p_j}-\bid\right)^{-1}\right|^{-1}\geq \nu_\infty\quad\forall j\in\bN.
			\end{equation}
		\item The sequence $(p_j)_{j\in\bN}$ satisfies
			\begin{equation}\label{eq:gaps_in_unif_nonres_iters}
				p_{j+m}-p_j=o\left(p_j\right)\quad\text{as }j\to\infty\quad\forall m\in\bN.
			\end{equation}
	\end{enumerate}
\end{prop}
\begin{proof}
	Notice that the only part of the spectrum that may possibly approach $1$ as we iterate $\upphi$ is given by the eigenvalues on $\Un(1)$. So let $\{e^{\pm i\alpha_1},\dots,e^{\pm i\alpha_l}\}$ be the eigenvalues on $\Un(1)$, listed repeating the multiple eigenvalues, when necessary. Let $\alpha_1,\dots,\alpha_l$ be choices of arguments of these eigenvalues, again repeated according to their multiplicity, and $a_j=\alpha_j/2\pi$. These are well defined numbers mod 1. There are two cases: the set $\{1,a_1,\dots,a_l\}$ spans a $\bQ$-subspace of $\bR$ of rank $1$, or it spans a $\bQ$-subspace of $\bR$ of rank at least 2.

	In the first case, each $a_j$ is a rational number, or equivalently we have that every eigenvalue of $\upphi_\infty$ on the unit circle is a root of unity. In this case it suffices to take a prime number $p_0\gg2$ larger than the largest prime factor of any order of these roots of unity, and $p_i,i\geq 1$ will be all the prime numbers larger than $p_0$ listed in increasing order. By the prime number theorem, this sequence satisfies the estimate on gaps \eqref{eq:gaps_in_unif_nonres_iters}, as we shall see below for the harder case of rank $\geq2$. By construction, none of the eigenvalues of $\upphi_\infty^{p_j}$ is 1 for every $j$, so the spectrum of $\upphi_\infty^{p_j}$ is uniformly bounded away from $1$, meaning that $\upphi_\infty^{p_j}-\bid$ is uniformly invertible.
	
	Now, assume that $\{1,a_1,\dots,a_l\}$ spans a $\bQ$-subspace of rank $q\geq 2$. This is equivalent to saying that there are exactly $q$ rationally independent irrational numbers in $\{a_1,\dots,a_l\}$. Without loss of generality we may assume that $a_1,\dots,a_q$ are such numbers, and that $a_i=\sum_{j=1}^qr_{ij}a_j$ for all $i\geq q+1$ where $r_{ij}\in\bQ$ for all $i\geq q+1$ and $j\leq q$. Set $r_{ij}=\delta_{ij}$ when $i,j\leq q$. 

	Since $\{a_1,\dots,a_q\}$ is a rationally independent set of irrational numbers, we can use lemma \ref{lem:prime_mults_irrat_vec} in order to conclude that the sequence $\left(P\vec a\right)_{P\text{ prime}}\subset\bR^q$ is equidistributed mod 1, where $\vec a=(a_1,\dots,a_q)$. By definition, for every Jordan-measurable set $C\subset[0,1]^q$ of $q$-volume $|C|$ it holds that
	\begin{equation}\label{eq:def_of_equidistribution}
		\lim_{n\to\infty}\frac{\#\left\{P\text{ prime }:P\leq n,\ P\vec a \bmod 1\in C\right\}}{\pi(n)}=|C|.
	\end{equation}
	Here $\pi\colon\bN\to\bN$ is the prime counting function, $\pi(n)=\#\{P\text{ prime}:P\leq n\}$.

	We look for a Jordan measurable set $C\subset\bR^q/\bZ^q$ of positive volume and a $c>0$ such that
	\begin{equation}\label{eq:unif_nonres_prime_iterates_wincon}
		P\vec{a}\bmod 1\in\ C\implies \mathrm{dist}_{\Un(1)}\left(\sigma(\upphi_\infty^P)\cap\Un(1),1\right)>c.
	\end{equation}
	Notice that $1\in\sigma(\upphi_\infty^P)$ if and only if $Pa_i=0\bmod 1$ for at least one $i\in\{1,\dots,l\}$. Consider the following collection of hyperplanes in $\bR^q$:
	\begin{equation}
		\Pi_i=\left\{x\in\bR^q:\sum_{j=1}^q r_{ij}x_j=0\right\},\quad\Pi=\bigcup_{i=1}^n\Pi_i
	\end{equation}

	\begin{figure}[h]
		\centering
		\begin{overpic}[width=.6\textwidth]{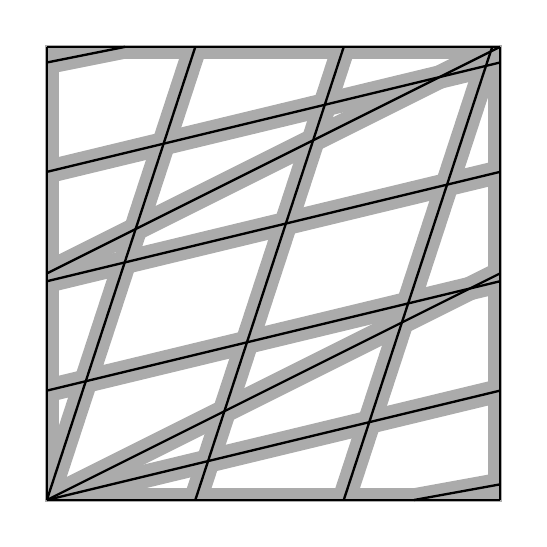}
			\put (45,2) {$\bR^q/\bZ^q$}
			\put (50,50) {$\Pi/\bZ^q$}
			\put (33,43) {$C^b_e$}
		\end{overpic}
		\caption{The black lines are $\Pi/\bZ^q$. Notice that the ``boundary box'' is included, since $r_{ij}=\delta_{ij}$ when $i,j\leq q$. The gray area represents their $e$-thickening. The white area represents $C^b_e$, which by construction has volume $\geq 1-b$. As long as this is positive, the prime iterates which stay uniformly away from resonances are equidistributed mod 1.}
		\label{fig:resonances}
	\end{figure}
	The set $\Pi$ is the union of finitely many hyperplanes defined by equations with \emph{rational} coefficients. Therefore, its projection $\Pi/\bZ^q\subset\bR^q/\bZ^q$ is a \emph{finite} collection (see Figure \ref{fig:resonances}) of hyperplanes such that
	\begin{equation}\label{eq:unif_nonres_iterates_wincon_neg}
		P\vec a\bmod 1\in\Pi/\bZ^q\iff1\in\sigma(\upphi_\infty^P).
	\end{equation}
	Moreover, since $\Pi/\bZ^q$ can be seen as a finite union of $(q-1)$-dimensional parallelepids in $[0,1)^q$, it has zero $q$-volume. It follows that for any $0<b<1$ there exists an $e>0$ such that
	\begin{equation}
		C^b_e=\left(\bR^q/\bZ^q\right)\setminus \bigcup_{x\in\Pi/\bZ^q}B_e(x)
	\end{equation}
	has $q$-volume $\left|C^b_e\right|\geq 1-b$. Clearly $e\to 0$ as $b\to 0$, so we can choose $b$ small enough so that $e<1/2$. By construction
	\begin{equation}\label{eq:unif_nonres_prime_iterates-unif_invert}
		P\vec a\bmod 1\in C^b_e\implies \mathrm{dist}_{\Un(1)}\left(\sigma(\upphi_\infty^P)\cap\Un(1),1\right)>2\pi e.
	\end{equation}
	By equidistribution, since $|C|=|C^b_e|\geq 1-b>0$, there is an increasing sequence $(p_j)_{j\geq 1}$ of prime numbers such that $p_j\vec a\bmod 1\in C^b_e$ for every $j$. By equation \eqref{eq:unif_nonres_prime_iterates-unif_invert} we have that $\upphi_\infty^{p_j}-\bid$ is uniformly invertible, so the first point is satisfied for an uniform $\nu_\infty>0$.

	To estimate the gaps, notice that from \eqref{eq:def_of_equidistribution} the cumulative distribution function of the sequence $\left(p_j\right)_j$ satisfies
	\begin{equation}
		P(n)=\#\left\{j\in\bN:p_j\leq n\right\}=|C|\pi(n)+o(\pi).
	\end{equation}
	In fact this estimate is much stronger than the estimate on the gaps we need: our sequence of primes is distributed like the sequence of \emph{all} primes, up to an irrelevant multiplicative constant.
	In particular, $(p_j)_{j\in\bN}$ satisfies a form of prime number theorem:
	\begin{equation}\label{eq:pnt_for_unif_nonres}
		P(n)=|C|\frac{n}{\log n}+o\left(\frac{n}{\log n}\right).
	\end{equation}
	In particular,
	\begin{equation}\label{eq:pnt_for_unif_nonres-useful}
		\lim_{j\to\infty}\frac{p_j}{|C|j\log j}=1.
	\end{equation}
	Write $p_{j+m}-p_j=\sum_{k=0}^{m-1}p_{j+k+1}-p_{j+k}$. By induction, we are thus reduced to show that $p_{j+1}-p_j=o(p_j)$. Now we compute, using \eqref{eq:pnt_for_unif_nonres-useful},
	\begin{equation}
		\lim_{j\to\infty}\frac{p_{j+1}}{p_j}=\lim_{j\to\infty}\frac{p_{j+1}}{|C|(j+1)\log(j+1)}\cdot\frac{|C|j\log j}{p_j}\cdot\frac{|C|(j+1)\log(j+1)}{|C|j\log j}=1.
	\end{equation}
	It follows that
	\begin{equation}
		\lim_{j\to\infty}\frac{p_{j+1}-p_j}{p_j}=\lim_{j\to\infty}\frac{p_{j+1}}{p_j}-1=0
	\end{equation}
	which was our claim.
\end{proof}

We apply this lemma to the linear map at infinity $\upphi_\infty$ of the ALHD $\upphi$ and we obtain an increasing sequence of prime numbers $(p_j)_j$ which satisfies items \ref{item:mainprop_ndg_at_infty} and \ref{item:mainprop_unif_nonres_gaps} in of point \ref{item:mainprop_unif_nonres_prime_iterates} of proposition \ref{prop:auxiliary_prop}.

\begin{rmk}
	Since $p_{j+m}/p_j\to 1$ as $j\to\infty$ we can interchange asymptotics in $p_j$ with asymptotics in $p_{j+m}$ without loss of information.
\end{rmk}

\subsection{Point \protect\ref{item:mainprop_H_and_Gj}, construction of \texorpdfstring{$H$}{H}}
First, we show that any ALHD $\upphi$ with $\upphi_\infty\in\Un(n)$ can be generated by an asymptotically quadratic Hamiltonian $H=Q+h$ such that $\phi^t_Q\in\Un(n)$ for all $t$.
\begin{lemma}\label{lem:relooping_to_auton_unitary}
	Let $H'_t=Q'_t+h'_t$ be an asymptotically quadratic Hamiltonian where $\phi^1_{Q'}\in\Un(n)$. There exists a time-dependent quadratic Hamiltonian $P$ generating a loop of unitary maps such that $H_t=P\#H'_t=Q+h_t$ is an asymptotically quadratic Hamiltonian with \emph{time-independent} quadratic Hamiltonian at infinity such that $\phi^t_Q\in\Un(n)$ for all $t$ and $\phi^1_H=\phi^1_{H'}$.
\end{lemma}
\begin{proof}
	Let $\phi^1_{Q'}=U\in\Un(n)\subset\Sp(2n)$. Since $U$ is unitary, it has a logarithm $b=\log U\in\mathfrak{u}(n)$. Notice that $J_0b=bJ_0$ and $b^T=-b$, since $\mathfrak{u}(n)=\mathfrak{o}(2n)\cap\mathfrak{gl}(n,\bC)$. Set $B=J_0b$. Then $B$ is symmetric: $B^T=-b^TJ_0=B$. Define
	\begin{equation}
		Q(z)=\frac12\left\langle Bz,z\right\rangle,\quad P=Q\#\overline{Q'}
	\end{equation}
	Notice $\phi^t_Q\in\Un(n)$ and $\phi^1_Q=U$ by construction. Then $\phi^1_P=\phi^1_Q\circ\left(\phi^1_{Q'}\right)^{-1}=\id$ so $\phi^t_P$ is a loop of unitary matrices. Finally
	\begin{equation}
		\begin{split}
			P\#H_t&=P_t+H_t\circ\left(\phi^t_P\right)^{-1}=Q-Q'_t\circ\phi^t_{Q'}\circ\phi^{-t}_Q+Q'_t\circ\phi^t_{Q'}\circ\phi^{-t}_Q+h'_t\circ\phi^t_{Q'}\circ\phi^{-t}_Q=\\
			&=Q+h_t
		\end{split}
	\end{equation}
	where $h_t=h'_t\circ\left(\phi^t_P\right)^{-1}$.
\end{proof}
Let $Q$ be an autonomous quadratic Hamiltonian generating a flow of unitary maps. Since every Hermitian matrix can be unitarily diagonalized, we can find a basis of $\bC^n$ for which $Q$ is the following diagonal quadratic form:
\begin{equation}\label{eq:quad_form_at_infty_diagonal}
	Q(z)=\frac12\sum_{r=1}^{n}\alpha_r\left|z_r\right|^2,\quad z=(z_1,\dots,z_n)\in\bC\oplus\dots\oplus\bC=\bC^n,\quad \alpha_r\in\bR.
\end{equation}
Notice that we are not assuming that $Q$ is positive-definite.

In conclusion, let $\upphi$ be a rapidly asymptotically linear Hamiltonian diffeomorphism with $\upphi_\infty\in\Un(n)$. Choose any rapidly asymptotically quadratic Hamiltonian $H'$ generating $\upphi$. We can compose $H'$ with a quadratic Hamiltonian $P$ so that the Hamiltonian $H=P\#H=Q+h$ is rapidly asymptotically quadratic (see lemma \ref{lem:linear_flow_composition_preserves_action}, point \ref{item:linflow_comp_props_preserves_asyquad_ndg}) and has autonomous diagonal quadratic form at infinity. This is the Hamiltonian $H$ we need in point \ref{item:mainprop_H_and_Gj}.

\subsection{Point \protect\ref{item:mainprop_H_and_Gj}, construction of \texorpdfstring{$F_{j,m}$}{Fjm}, step 1: truncation}

The Hamiltonian $F_{j,m}$ is constructed by truncation of $H^{p_{j+m}\ominus p_j}$ and then interpolation of its quadratic Hamiltonian at infinity from $Q^{p_{j+m}\ominus p_j}$ to $0\wedge Q^{\times p_j}$. Here we explain the first step and how the rapidity hypothesis gives an estimate on the growth of the truncation radius in terms of the iteration.

First, we use a very explicit sequence of loops to define the re-indexed Hamiltonian $H^{p_{j+m}\ominus p_j}$. Recall that the sequence $\sigma_{j,m}$ is defined as follows:
\begin{equation}\label{eq:def_of_sigmaj_sequence_of_shifts}
	\sigma_{j,m}=\ind_\infty H^{\times p_{j+m}}-\ind_\infty H^{\times p_j}.
\end{equation}
The estimate on the growth of $\sigma_{j,m}$ in point \ref{item:mainprop_index_shift_estimate} is just the estimate already given in section \ref{ssub:re_indexing_iterated_hamiltonians}, equation \eqref{eq:reloop_index_shift_estimate}, with $k=p_{j+m},l=p_j$ and $2\mu=\sigma_{j,m}$.

For brevity set $k=p_{j+m}$ and $l=p_j$. Given the formula for the quadratic Hamiltonian at infinity \eqref{eq:quad_form_at_infty_diagonal}, we can find an explicit formula for the quadratic Hamiltonian $P^{k,l}$ generating the loop we need to reduce the index at infinity:
\begin{equation}\label{eq:quad_form_delooper}
	P^{k,l}(z)=\frac12\sum_r2\pi\floor{\frac{(k-l)\alpha_r}{2\pi}}\left|z_r\right|^2
\end{equation}
where $\floor{\alpha}=\max\{j\in\bZ:j\leq\alpha\}$.

We calculate
\begin{equation}\label{eq:calc_of_short_connecting_path_hamilt}
	\begin{split}
		\overline{P^{k,l}}\#H^{\times (k-l)}_t(z)&=\frac12\sum_r2\pi\left[\frac{(k-l)\alpha_r}{2\pi}-\floor{\frac{(k-l)\alpha_r}{2\pi}}\right]\left|z_r\right|^2+(k-l)h_{(k-l)t}\circ\phi^t_{P^{k,l}}(z)=\\
		&=R(z)+h^{\times(k-l)}_t\circ\phi^t_{P^{k,l}}(z).
	\end{split}
\end{equation}
It is extremely important to remark that for any $k>l$ odd integers, by construction we have the estimate
\begin{equation}\label{eq:quadHam_short_connecting_path_hamilt_unif_bded_coeffs}
	\left|R(z)\right|\leq \pi |z|^2\implies \left\|\nabla^2R\right\|_{L^\infty}\leq \pi.
\end{equation}
Substituting these formulas in the definition of the re-looped Hamiltonian \eqref{eq:def_of_relooping_at_infty},
\begin{equation}\label{eq:full_preinterp_F}
	H^{k\ominus l}=\left(\overline{P^{k,l}}\#H^{\times(k-l)}\right)\wedge H^{\times l}=R\wedge Q^{\times l}+\left(h^{\times(k-l)}\circ\phi^\cdot_{P^{k,l}}\right)\wedge h^{\times l}=Q^{k\ominus l}+h^{k\ominus l}.
\end{equation}
By construction $\ind_\infty(H^{k\ominus l})=\ind_\infty(H^{\times l})$.

Truncate the Hamiltonian $H^{k\ominus l}$ following proposition \ref{prop:nonresonant_truncation}, where now the lower bound on the radius given in equation \eqref{eq:def_of_tildeR_1}, say $\tilde R_{k,l}>0$, depends on $k=p_{j+m}$ and $l=p_j$ through the tails of the function $h^{k\ominus l}$. In the construction of the truncation, we can choose for example $\sqrt{R_a}=\tilde R_{k,l}+1$, and obtain a new Hamiltonian $\widetilde{H^{k\ominus l}}$ which is equal to $H^{k\ominus l}$ in a ball of radius $\sqrt{R_a}$, and equal to $Q^{k\ominus l}$ outside a ball of radius $\sqrt{2R_a}$. Let us write
\begin{equation}
	\widetilde{H^{k\ominus l}_t}(z)=Q^{k\ominus l}(z)+\widetilde{h^{k\ominus l}_t}(z)
\end{equation}
where $\widetilde{h^{k\ominus l}}$ is compactly supported in a ball of radius $\sqrt{2R_a}$.

Here is the first crucial consequence of the rapidity condition on asymptotically quadratic Hamiltonians.
\begin{lemma}\label{lem:estimate_truncation_radius_k}
	If $H$ is rapidly asymptotically quadratic, then $\tilde R_{p_{j+m},p_j}=o\left(\sqrt{p_{j+m}}\right)$ as $j\to\infty$.
\end{lemma}
\begin{proof}
	The truncation radius $\tilde R_{k,l}$ depends on the tail functions of $h^{k\ominus l}$ via \eqref{eq:def_of_tildeR_1}, which we now estimate in terms of the tail functions of $h$. Henceforth, we will set $\sigma_j=\sigma_j^h$ for $j=0,1$ and $\sigma^{k\ominus l}_0,\sigma^{k\ominus l}_1$ the corresponding tails of $h^{k\ominus l}$. Recall from equation \eqref{eq:full_preinterp_F} that $h^{k\ominus l}$ is a concatenation:
	\begin{equation}
		h^{k\ominus l}_t(z)=
			\begin{cases}
				2\tau'(2t)h^{\times(k-l)}_{\tau(2t)}\left(\phi^{\tau(2t)}_{P^{k,l}}z\right),& t\in\left[0,\frac12\right]\\
				2\tau'(2t-1)h^{\times l}_{\tau(2t-1)}(z), & t\in\left[\frac12,1\right].
			\end{cases}
	\end{equation}
	This gives us an explicit formula for the tail functions.
	\begin{equation}\label{eq:tails_of_relooped}
		\begin{multlined}[b]
			\sigma^{k\ominus l}_1(R)=\sup_{t\in[0,1],|z|\geq R}\frac{\left|\nabla h^{k\ominus l}_t(z)\right|}{|z|}=\\
			\left.\max
			\begin{cases}
				\sup_{t\in\left[0,\frac12\right],|z|\geq R}\frac{2l\tau'\left(2t\right)\left|\nabla h_{l\tau(2t)}(z)\right|}{|z|},\\
				\sup_{t\in\left[\frac12,1\right],|z|\geq R}\frac{2(k-l)\tau'(2t-1)\left|\left(\phi^{2\tau(t)-1}_{P^{k,l}}\right)^{-1}\nabla h_{(k-l)\tau(2t-1)}\left(\phi^{2\tau(t)-1}_{P^{k,l}}z\right)\right|}{|z|}
			\end{cases}\right\}.
		\end{multlined}
	\end{equation}
	Starting from the last supremum, we can use that the loop is unitary:
	\begin{equation}\label{eq:tails_of_relooped_last_sup}
		\begin{multlined}[b]
			\sup_{t\in\left[\frac12,1\right],|z|\geq R}\frac{2(k-l)\tau'(2t-1)\left|\left(\phi^{2\tau(t)-1}_{P^{k,l}}\right)^{-1}\nabla h_{(k-l)\tau(2t-1)}\left(\phi^{2\tau(t)-1}_{P^{k,l}}z\right)\right|}{|z|}\leq\\
			\leq4(k-l)\sup_{t\in\left[\frac12,1\right],|z|\geq R}\frac{\left|\nabla h_{(k-l)\tau(2t-1)}\left(\phi^{2\tau(t)-1}_{P^{k,l}}z\right)\right|}{\left|\phi^{2\tau(t)-1}_{P^{k,l}}z\right|}\leq\\
			\leq4(k-l)\sup_{\tau\in[0,1],\left|\zeta\right|\geq R}\frac{\left|\nabla h_\tau\left(\zeta\right)\right|}{|\zeta|}=4(k-l)\sigma_1(R).
		\end{multlined}
	\end{equation}
	The first supremum is easy to estimate:
	\begin{equation}
		\sup_{t\in\left[0,\frac12\right],|z|\geq R}\frac{2l\tau'\left(2t\right)\left|\nabla h_{l\tau(2t)}(z)\right|}{|z|}\leq 4l\sigma_1(R).
	\end{equation}
	All in all, we can estimate the max by the sum of its terms and obtain
	\begin{equation}\label{eq:tail1_of_relooped_4k}
		\sigma^{k\ominus l}_1(R)\leq 4k\sigma_1(R).
	\end{equation}
	In a completely analogous way,
	\begin{equation}\label{eq:tail0_of_relooped_4k}
		\sigma_0^{k\ominus l}(R)\leq 4k\sigma_0(R).
	\end{equation}
	
	Now that we estimated the tails, let's study the definition of the lower bound on the truncation radius \eqref{eq:def_of_tildeR_1} to understand the dependence of $\tilde R_{k,l}$ on $k,l$. This is defined as follows:
	\begin{equation}\label{eq:def_of_Rk_truncrad_iter}
		\tilde R_{k,l}=\max\left\{R:\bar\sigma^{k\ominus l}_{01}\left(R\right)\geq e^{-3C^{k\ominus l}_2}\nu_\infty^{k\ominus l}\right\}.
	\end{equation}
	We used the shortened notations
	\begin{equation}
		\bar\sigma^{k\ominus l}_{01}(R)=\left(\sqrt2+e^{-C_2^{k\ominus l}}\right)\bar\sigma_1^{k\ominus l}(R)+4\sqrt2 c_\chi\bar\sigma_0^{k\ominus l}(R), \quad C_2^{k\ominus l}=\left\|\nabla^2H^{k\ominus l}\right\|_{L^\infty},
	\end{equation}
	and
	\begin{equation}\label{eq:resoprox_cst_reloop}
		\nu^{k\ominus l}_\infty=\left|\left(\phi^1_{Q^{k\ominus l}}-\bid\right)^{-1}\right|^{-1}.
	\end{equation}
	First, let's estimate the left-hand side of the inequality in \eqref{eq:def_of_Rk_truncrad_iter} from above. To do this, we estimate the Hessian. It's easy to see that
	\begin{equation}
		\left\|\nabla^2 h^{k\ominus l}\right\|_{L^\infty}\leq 4k\left\|\nabla^2 h\right\|_{L^\infty}=4kc_2.
	\end{equation}
	Combining the definition \eqref{eq:full_preinterp_F} with the estimate \eqref{eq:quadHam_short_connecting_path_hamilt_unif_bded_coeffs} we obtain that
	\begin{equation}\label{eq:relooped_hessian_est}
		C^{k\ominus l}_2=\left\|\nabla^2H^{k\ominus l}\right\|_{L^\infty}\leq\pi l+4kc_2\leq (\pi+4c_2)k
	\end{equation}
	All in all there is some constant $\beta>0$ such that for all $k=p_{j+m}$ and $l=p_j$,
	\begin{equation}\label{eq:barsigma01kominl_estimate}
		\begin{multlined}[b]
			\bar\sigma^{k\ominus l}_{01}\left(R\right)\leq 4k\left[\left(\sqrt 2+e^{-\left(\pi+4c_2\right)k}\right)\bar\sigma_1\left(R\right)+4\sqrt{2}c_\rho \bar\sigma_0\left(R\right)\right]\leq\\
			\leq \beta k\left[\bar\sigma_1\left(R\right)+\bar\sigma_0\left(R\right)\right].
		\end{multlined}
	\end{equation}
	Now we estimate the right-hand side of the inequality in \eqref{eq:def_of_Rk_truncrad_iter} from below. We have to calculate $\nu^{k\ominus l}_\infty$. By its very definition \eqref{eq:resoprox_cst_reloop}, this constant depends only on the linear map at infinity of the system. This is
	\begin{equation}
		\phi^1_{Q^{k\ominus l}}=\phi^1_{\left(\overline{P^{k,l}}\# Q^{\times(k-l)}\right)\wedge Q^{\times l}}=\phi^{-1}_{P^{k,l}}\circ\phi^{k-l}_Q\circ\phi^l_Q=\phi^k_Q
	\end{equation}
	so by definition
	\begin{equation}
		\nu^{k\ominus l}_\infty=\left|\left(\upphi_\infty^k-\bid\right)^{-1}\right|^{-1}\geq \nu_\infty>0
	\end{equation}
	where $\nu_\infty>0$ is the constant, independent of $k=p_{j+m}$, coming from point \ref{item:mainprop_unif_nonres_prime_iterates}.
	We already estimated the Hessian in \eqref{eq:relooped_hessian_est}. Therefore,
	\begin{equation}\label{eq:defRk_lhs_est_below}
		e^{-3C^{k\ominus l}_2}\nu^{k\ominus l}_\infty\geq e^{-3\left(\pi+4c_2\right)k}\nu_\infty
	\end{equation}
	Putting everything together we can estimate $\tilde R_{k,l}$ from above by a similar quantity with a simpler definition: there are positive constants $\nu'_\infty=\beta^{-1} \nu_\infty,\gamma=3(\pi+4c_2)$ depending only on $H$ (and therefore not on the iterates $l=p_j$ and $k=p_{j+m}$) such that, setting
	\begin{equation}\label{eq:def_of_rk}
		r_k=\nu'_\infty\frac{e^{-\gamma k}}{k},
	\end{equation}
	we have
	\begin{equation}\label{eq:Rk_est_above}
		\tilde R_{k,l}\leq R_k=\max\left\{R:\bar\sigma_0\left(R\right)+\bar\sigma_1\left(R\right)\geq r_k\right\}\quad\forall l.
	\end{equation}
	Now we estimate $R_k$ from above. We are presented with the problem of studying $R_k=\mathbf{R}(r_k)$ where, setting $\mathbf{r}(R)=\bar\sigma_0\left(R\right)+\bar\sigma_1\left(R\right)$,
	\begin{equation}
		\mathbf{R}(r)=\max\left\{R:\mathbf{r}(R)\geq r\right\}.
	\end{equation}
	Since $\mathbf{r}$ is non-increasing and continuous, $\mathbf{R}$ is well defined as a function and is a right inverse of $\mathbf{r}$: for $r\in(0,\mathbf{r}(0))$ it holds that $\mathbf{r}(\mathbf{R}(r))=r$. Moreover it is non-increasing and lower semi-continuous. Since $\mathbf{r}(R)\to 0$ as $R\to\infty$, we have that $\mathbf{R}(r)\to\infty$ as $r\to 0$.
	\begin{rmk}
		In fact since $\mathbf{r}$ is non-increasing and continuous we also know that $\mathbf{R}(\mathbf{r}(R))\geq R$.
	\end{rmk}
	The rapidity condition on $H$ implies that
	\begin{equation}\label{eq:rapidity_condition_Rk_estimate}
		\mathbf{r}(R)=\bar\sigma_0\left(R\right)+\bar\sigma_1\left(R\right)=o\left(e^{-R^2}\right)\text{ as }R\to\infty.
	\end{equation}
	This implies that $\mathbf{R}(r)=o\left(|\log r|^{\frac 12}\right)$ as $r\to 0$.
	In other words, for $r$ small, $\mathbf{R}(r)=\epsilon(r)|\log r|^\frac12$ with $\epsilon$ converging to $0$ as $r\to 0$. This gives, as $k\to\infty$,
	\begin{equation}
		\begin{split}
			\frac{R_k}{\sqrt k}=\epsilon\left(r_k\right)\frac{\left|\log r_k\right|^\frac12}{\sqrt k}&=\epsilon\left(r_k\right)\sqrt{\frac{\left|\log\left(M'_\infty e^{-\gamma k}k^{-1}\right)\right|}{k}}=\\
			&=\epsilon(r_k)\sqrt{\frac{\left|\log M'_\infty-\log k-\gamma k\right|}{k}}\leq\\
			&\leq\underbrace{\epsilon(r_k)}_{\to 0}\underbrace{\sqrt{\frac{\left|\log M'_\infty\right|}{k}+\frac{\log k}{k}+\gamma}}_{\to\sqrt\gamma}\to 0
		\end{split}
	\end{equation}
	Therefore $R_k=o\left(\sqrt k\right)$ as $k\to\infty$ which implies $\tilde{R}_{p_{j+m},p_j}=o\left(\sqrt{p_{j+m}}\right)$ as $j\to\infty$.
\end{proof}

\subsection{Point \protect\ref{item:mainprop_H_and_Gj}, construction of \texorpdfstring{$F_{j,m}$}{Fjm}, step 2: quadratic transplant}

Up to this point, we have a Hamiltonian $\widetilde{H^{k\ominus l}}$ which is equal to $Q^{k\ominus l}$ outside a  ball whose radius grows much slower than $\sqrt k$. Now, we will interpolate this quadratic form using the construction in section \ref{ssub:interpolation_of_quadratic_hamiltonians} with a very specific non-resonant homotopy of unitary Hamiltonian systems.

Recall the quadratic Hamiltonian $R$ from the formula for $H^{k\ominus l}$, \eqref{eq:full_preinterp_F}. Set for notational convenience $R(z)=\frac12\sum_j\beta_r|z_r|^2$. Notice that $\beta_r\in(0,2\pi)$, which again implies the important estimate \eqref{eq:quadHam_short_connecting_path_hamilt_unif_bded_coeffs}. We homotope $R(z)$ to the zero quadratic Hamiltonian, defining
\begin{equation}\label{eq:special_short_unitary_nonres_homotp_firstpart}
	R^s(z)=\frac12\sum_r\beta^s_r|z_r|^2
\end{equation}
where the angular velocities $\beta_r^s$ are determined in the following way: fix an $r$ and notice that the flow of $Q^{\times l}$ fixes the $r$th complex line in the symplectic decomposition $\bC^n=\bC\oplus\dots\oplus\bC$, and on this line it restricts to the map
\begin{equation}
	z_r\mapsto e^{-l\alpha_rit}z_r.
\end{equation}
Set $\zeta_r=e^{-l\alpha_ri}$ i.e. the image of the vector $1\in\bC$ under the map above at time $t=1$. Now consider the arc
\begin{equation}\label{eq:short_arc_jth}
	\gamma\colon[0,1]\to S^1\subset\bC,\quad\gamma(s)=e^{-i\beta_r(1-s)}\zeta_r
\end{equation}
We know by hypothesis that $\gamma(1)\neq 1$. If there is no $s\in(0,1)$ such that $\gamma(s)=1$, then we set $\beta_r^s=(1-s)\beta_r$. Otherwise $\beta_r^s=(1-s)\left(2\pi-\beta_r\right)$, i.e. we trace the complementary arc in the unit circle. Notice that the coefficients $\beta^r_s\in\bR$ may have any sign. By construction, $R^0=R$ and $R^1=0$. We define
\begin{equation}\label{eq:special_short_unitary_nonres_homotp}
	\mathcal{Q}^s= R^s\wedge Q^{\times l}.
\end{equation}
This is a homotopy between
\begin{equation}\label{eq:quadratic_transplant_qHams}
	\mathcal{Q}^0=R\wedge Q^{\times l}=Q^{k\ominus l},\quad \mathcal Q^1=0\wedge Q^{\times l}.
\end{equation}
For every $s$, the linear system defined by $\mathcal{Q}^s$ is non-degenerate and unitary by construction, exactly because we avoid the possible resonance which might happen if the arc traced by $\gamma$ at some $s$ hits $1$ for some $r$. By construction
\begin{equation}\label{eq:short_homotopy_is_O(1)}
	0\leq\left|\beta_r^s\right|<2\pi\implies \max_{(s,z)\in[0,1]\times B_\rho(0)}\left|R^s(z)\right|<\pi \rho^2.
\end{equation}

We want to apply proposition \ref{prop:nonresonant_interpolation_quadratic_forms} to $Q^0=Q^{k\ominus l}$ and $Q^1=0\wedge Q^{\times l}$ using the homotopy $\mathcal{Q}$ to obtain the interpolating semi-quadratic Hamiltonian $K$, defined by equation \eqref{eq:def_of_interpolation_at_infty}. Recall that the first radius of interpolation $R_0$ can be chosen arbitrarily, while the second $R_1>R_0$ depends on $R_0$ and the homotopy $\mathcal{Q}$. But with the uniformly non-resonant sequence of prime iterates $(p_j)_{j\in\bN}$ and the explicit homotopy $\mathcal{Q}$ we can control $R_1$.
\begin{lemma}\label{lem:second_interp_radius_unif_bded_along_unif_nnres_iters}
	Let $k=p_{j+m}$ and $l=p_j$ be uniformly non-resonant prime iterates. Let $Q^0=Q^{k\ominus l}$, $Q^1=0\wedge Q^{\times l}$ and $\mathcal{Q}$ the non-resonant homotopy of unitary Hamiltonian systems defined in \eqref{eq:special_short_unitary_nonres_homotp}. Let $R_0>0$ be fixed arbitrarily. There exists a $\overline{C}>0$, independent of $k$ and $l$, such that if $R_1$ is defined by equation \eqref{eq:def_of_last_interp_radius}, then $R_1\leq\overline{C}R_0$.
\end{lemma}
\begin{proof}
	Recall from \eqref{eq:def_of_last_interp_radius} and \eqref{eq:interpolation_constants} that
	\begin{equation}
		R_1=e^{(1+\delta)\frac{C_1}{C_0}}R_0,\quad C_0=\min_{s\in[0,1],z\in\bR^{2n}}\frac{\left|\phi^1_{\mathcal{Q}^s}z-z\right|}{|z|}>0,\quad C_1=\left\|\partial_s\nabla^2_z\mathcal Q\right\|_{L^\infty}.
	\end{equation}
	The easiest is $C_1$. By definition of $R^s$ \eqref{eq:special_short_unitary_nonres_homotp_firstpart},
	\begin{equation}\label{eq:second_interp_radius_unif_bded_C_1}
		C_1=\left\|\partial_s\nabla^2_z\mathcal Q\right\|_{L^\infty}=\max_{r=1,\dots,n}\frac{\left|\partial_s\beta^s_r\right|}{2}=\max_{r=1,\dots,n}\frac{\left|\beta_r\right|}{2}<\pi.
	\end{equation}
	To estimate $C_0$, recall from equations \eqref{eq:reso_prox_cst_homotopy} and \eqref{eq:interp_cst_C0_est_below_reso_prox} that
	\begin{equation}
		\left|\phi^1_{\mathcal{Q}^s}z-z\right|\geq\left|\left(\phi^1_{\mathcal{Q}^s}-\bid\right)^{-1}\right|^{-1}|z|=\nu_\infty(s)|z|\implies C_0\geq\min_{s\in[0,1]}\nu_\infty(s).
	\end{equation}
	By construction of the path $R^s$ it follows that
	\begin{equation}
		\nu_\infty(s)\geq \min\left\{\nu_\infty(0),\nu_\infty(1)\right\}.
	\end{equation}
	But since $k=p_{j+m}$ and $l=p_j$ are uniformly non-resonant prime iterations, by item \ref{item:mainprop_ndg_at_infty} in point \ref{item:mainprop_unif_nonres_prime_iterates}
	\begin{equation}
		\nu_\infty(0)=\left|\left(\upphi_\infty^k-\bid\right)^{-1}\right|^{-1}\geq\nu_\infty,\quad \nu_\infty(1)=\left|\left(\upphi_\infty^l-\bid\right)^{-1}\right|^{-1}\geq\nu_\infty.
	\end{equation}
	In particular $C_0\geq \nu_\infty>0$. To finish just set 
	\begin{equation}
		\overline{C}=\max\left\{e^{(1+\delta)\frac{\pi}{\nu_\infty}},2\right\}.
	\end{equation}
	Since $\nu_\infty$ does not depend on $j$ and $m$, the same is true for $\overline{C}$.
\end{proof}

We are finally ready to define $F_{j,m}$. Recall the definition of the Hamiltonian $K$ in equation \eqref{eq:def_of_interpolation_at_infty}. Choose $\sqrt{R_0}=2(\tilde R_{k,l}+2)>\sqrt{2R_a}=\sqrt{2}(\tilde R_{k,l}+1)$ and define
\begin{equation}\label{eq:def_of_Fjm}
	F_{j,m}(t,z)=K_t(z)+\widetilde{h^{p_{j+m}\ominus p_j}_t}(z)=0\wedge Q^{\times p_j}+f_{j,m},\quad f_{j,m}=K_t(z)-0\wedge Q^{\times p_j}+\widetilde{h^{p_{j+m}\ominus p_j}_t}(z).
\end{equation}
In particular
\begin{equation}\label{eq:Fjm_cpt_pert}
	\supp f_{j,m}=B_{\sqrt{R_1}}(0)=B_{j,m},\quad \sqrt{R_1}\leq\sqrt{2\overline{C}}\left(\tilde{R}_{p_{j+m},p_j}+2\right).
\end{equation}
We must prove that $F_{j,m}$ has the properties claimed in item \ref{item:mainprop_Fjm_dynquants} of point \ref{item:mainprop_H_and_Gj} in proposition \ref{prop:auxiliary_prop}.
\begin{lemma}\label{lem:Fjm_works}
	It holds that $\Fix\phi^1_{F_{j,m}}=\Fix\upphi^{p_{j+m}}$. If $z$ is any $k$-fold iterated fixed point of $\upphi$, seen as a fixed point of $\phi^1_{F_{j,m}}$, we have
	\begin{equation}
		\begin{split}
			&\cA_{F_{j,m}}(z)=p_{j+m}\cA_H(z),\\
			&\CZ\left(z,F_{j,m}\right)=\CZ\left(z,H^{\times p_{j+m}}\right)-\sigma_{j,m}.
		\end{split}
	\end{equation}
	If $z_0\in\Fix\upphi$ is isolated as a fixed point of $\phi^1_{F_{j,m}}$, then
	\begin{equation}
		\HFloc_*\left(F_{j,m},z_0\right)\cong\HFloc_{*+\sigma}\left(H^{\times p_{j+m}},z_0\right).
	\end{equation}
\end{lemma}
\begin{proof}
	Notice that $F_{j,m}$ equals $H^{p_{j+m}\ominus p_{j}}$ on a ball containing all the fixed points of its time-1 map, and outside that ball it has no fixed points. Therefore
	\begin{equation}
		\Fix\phi^1_{F_{j,m}}=\Fix\phi^1_{H^{p_{j+m}\ominus p_j}}=\Fix\upphi^{p_{j+m}}.
	\end{equation}
	We have already calculated the action and index of the iterated fixed points in lemma \ref{lem:relooped_hamilt_properties}. Set $p_{j+m}=k$ and $p_j=l$ as usual. Let $z_0$ be an isolated fixed point of $\phi^1_{F_{j,m}}$. By the same token,
	\begin{equation}
		\HFloc_*\left(F_{j,m},z_0\right)\cong\HFloc_{*}\left(H^{p_{j+m}\ominus p_j},z_0\right).
	\end{equation}
	The flows of $H^{k\ominus l}$ and $\overline{P^{k,l}}\#H^{\times k}$ have the same end-points at time 1 and are homotopic with fixed end-points in $\Ham$. Therefore they are related by a contractible loop in $\Ham$. We can apply lemma \ref{lem:HFloc_invariant_contractible_loop_compos} and obtain
	\begin{equation}
		\HFloc_{*}\left(H^{k\ominus l},z_0\right)\cong\HFloc_{*}\left(\overline{P^{k,l}}\#H^{\times k},z_0\right).
	\end{equation}
	Lemma \ref{lem:HFloc_action_of_linear_loops_shifts_grading} gives us
	\begin{equation}
		\HFloc_{*}\left(\overline{P^{k,l}}\#H^{\times k},z_0\right)\cong\HFloc_{*+\sigma}\left(H^{\times k},z_0\right)
	\end{equation}
	and this concludes the proof.
\end{proof}

\subsection{Point \protect\ref{item:mainprop_H_and_Gj}, construction of \texorpdfstring{$G_j$}{Gj}}
The Hamiltonian $G_j$ is obtained by truncation of $0\wedge H^{\times p_j}$ as in proposition \ref{prop:nonresonant_truncation}. The result is a Hamiltonian of the form
\begin{equation}
	G_j=0\wedge Q^{\times p_j}+\widetilde{0\wedge h^{\times p_j}}=0\wedge Q^{\times p_j}+g_j
\end{equation}
with $g_j$ compactly supported in a ball of radius $\sqrt{R'_a}=\tilde R_{p_j}+1$ where $\tilde R_{p_j}$ is defined as in equation \eqref{eq:def_of_tildeR_1}, determined by the tails of $0\wedge h^{\times p_j}$. Inspecting \eqref{eq:tails_of_relooped}, $\sigma^{k\ominus l}_1$ is the maximum between the $\sigma_1$-tail of $0\wedge h^{\times p_j}$ and some other function. The same is true for $\sigma_0^{k\ominus l}$. In particular, we will always have that $\tilde R_{p_j}\leq \tilde R_{p_{j+m},p_j}$. We are thus allowed to use the same truncation radius \emph{also} for $0\wedge H^{\times l}$, without creating additional fixed points in the interpolation region. This makes the action shift constant estimates simpler.

Notice that since we chose $\sqrt{R'_a}=\tilde{R}_{p_{j+m},p_j}+1<\sqrt{2\overline{C}}(\tilde{R}_{p_{j+m},p_j}+2)$,
\begin{equation}\label{eq:Gj_cpct_pert}
	\supp g_j\subset B_{\sqrt{R'_a}}(0)\subset\supp f_{j,m}\subset B_{j,m}.
\end{equation}

The fact that $G_j$ has the properties claimed in item \ref{item:mainprop_Gj_dynquants} of point \ref{item:mainprop_H_and_Gj} in proposition \ref{prop:auxiliary_prop} follows from the same reasoning we used to prove the analogous statement for $F_{j,m}$ in lemma \ref{lem:Fjm_works}.

\subsection{Point \protect\ref{item:mainprop_H_and_Gj}, item \protect\ref{item:mainprop_action_shift_estimate}}

We constructed Hamiltonians $F=F_{j,m}$ and $G=G_j$, which are compactly supported perturbations of \emph{the same quadratic form at infinity} $0\wedge Q^{\times l}$. Therefore they are at finite uniform distance, and this is what we have to estimate here. The goal is to show that
\begin{equation}
	\left\|F_{j,m}-G_j\right\|_{L^\infty}=o\left(p_{j+m}\right)\text{ as }j\to\infty.
\end{equation}

The truncation and interpolation radii chosen are:
\begin{equation}\label{eq:all_interp_radii}
	\sqrt{R_a}=\tilde{R}_{p_{j+m},p_j}+1,\quad R_b=2R_a,\quad \sqrt{R_0}=2\tilde{R}_{p_{j+m},p_j}+4,\quad R_1\leq\overline{C}R_0.
\end{equation}
Recall the definition of the truncated Hamiltonian \eqref{eq:truncated_hamilt} and the calculation of the re-indexed Hamiltonian \eqref{eq:full_preinterp_F}. Since we use the same truncation radii in the construction of $F_{j,m}$ and $G_j$, we use the same truncation step function $\rho$ to define them. We calculate
\begin{equation}\label{eq:hamdiff}
	\begin{multlined}
		F-G=f-g=\\
		=K-0\wedge Q^{\times l}+\rho\cdot \left[h^{k\ominus l}- 0\wedge h^{\times l}\right]=\\
		= R^{\chi\left(|z|^2\right)}\wedge 0+\rho\cdot\left[\left(h^{\times(k-l)}\circ\phi^\cdot_{P^{k,l}}\right)\wedge h^{\times l}-0\wedge h^{\times l}\right]=\\
		=\left[ R^{\chi\left(|z|^2\right)}+\rho\cdot\left(h^{\times(k-l)}\circ\phi^\cdot_{P^{k,l}}\right)\right]\wedge 0
	\end{multlined}
\end{equation}
We used that
\begin{equation}\label{eq:hamdiff_fourth_col}
	K-0\wedge Q^{\times l}=R^{\chi\left(|z|^2\right)}\wedge Q^{\times l}-0\wedge Q^{\times l}=R^{\chi\left(|z|^2\right)}\wedge 0
\end{equation}
with $\chi$ the special step function constructed in the proof of the proposition \ref{prop:nonresonant_interpolation_quadratic_forms}. This function is compactly supported on a ball of radius $\sqrt{R_1}$, and we can estimate
\begin{equation}
	\left\|R^{\chi}\wedge0\right\|_{L^\infty}=\left\|R^{\chi}\right\|_{L^\infty}\leq \pi R_1=2\overline{C}\left(\tilde{R}_{p_{j+m},p_j}+2\right).
\end{equation}
The second term can be estimated as
\begin{equation}
\left\|\rho\cdot\left(h^{\times(k-l)}\circ\phi^\cdot_{P^{k,l}}\right)\right\|_{L^\infty\left(S^1\times\bR^{2n}\right)}\leq\left\|h^{\times(k-l)}\right\|_{L^\infty\left(S^1\times B_{\sqrt{R_1}}\right)}=(k-l)\left\|h\right\|_{L^\infty\left(S^1\times B_{j,m}\right)}
\end{equation}
The uniform distance can thus be estimated as
\begin{equation}\label{eq:hamdiff_dist_est}
	\left\|F_{j,m}-G_j\right\|_{L^\infty}\leq \pi R_1+(p_{j+m}-p_j)\left\|h\right\|_{L^\infty\left(S^1\times B_{j,m}\right)}.
\end{equation}
Recall that since $H$ is rapidly asymptotically quadratic, $h$ decays to zero. In particular, it is bounded, so we may set $\|h\|_{L^\infty}=c_0$. Moreover, in lemma \ref{lem:estimate_truncation_radius_k} we found that $\tilde R_{p_{j+m},p_j}=o\left(p_{j+m}^{\frac12}\right)$ as $j\to\infty$. Finally, in proposition \ref{prop:uniform_nonresonant_iterations_of_linear_symplectic_map} we saw that $p_{j+m}-p_j=o(p_j)$ as $j\to\infty$. Using these facts, we can estimate
\begin{equation}
	\begin{multlined}
		\left\|F_{j,m}-G_j\right\|_{L^\infty}\leq\\
		\leq 2\pi\overline{C}\left(\tilde R_{p_{j+m},p_j}+2\right)^2+(p_{j+m}-p_j)\left\|h\right\|_{L^\infty\left(S^1\times B_{2\overline{C}\left(\tilde R_{p_{j+m},p_j}+2\right)}\right)}\leq\\
		\leq 2\pi\overline{C}\left(\tilde R_{p_{j+m},p_j}+2\right)^2+(p_{j+m}-p_j)\left\|h\right\|_{L^\infty\left(S^1\times\bR^{2n}\right)}=\\
		=2\pi\overline{C}\left(\tilde R_{p_{j+m},p_j}+2\right)^2+(p_{j+m}-p_j)c_0=\\
		=o\left(\sqrt{p_{j+m}}\right)^2+o\left(p_{j+m}\right)=o\left(p_{j+m}\right).
	\end{multlined}
\end{equation}
This concludes the proof of proposition \ref{prop:auxiliary_prop}.

\section{Floer Homology of Asymptotically Linear Hamiltonian Systems}
\label{sec:floer_homology_of_asymptotically_linear_hamiltonian_systems}

Floer homology, in its most elementary form, detects periodic orbits of Hamiltonian systems, and ties their existence to symplectic topology. It was introduced by Floer \cite{Floer1987-MorseTheorySymplFixedPts,Floer1988-MorseTheoryLagrInters,Floer1989_SympFixedPtsHolSph} in order to solve the Arnol'd conjectures on fixed points of Hamiltonian diffeomorphisms and intersections of Lagrangian submanifolds. Floer homology resembles formally the Morse homology of a smooth function (see e.g. \cite[Part I]{AudinDamian2014_MorseFloer}), where the smooth function in question is the classical Hamiltonian action functional on loop space. It is the homology of a chain complex which is generated by periodic orbits of fixed period and whose differential counts certain gradient descent-like trajectories between them, called Floer trajectories. We give a brief overview of this construction in section \ref{sub:floer_chain_complex_and_floer_homology}.

The construction of Floer homology for asymptotically quadratic Hamiltonians follows the standard recipe, which can be found for example in \cite[Part II]{AudinDamian2014_MorseFloer}, with some precautions. Since the target manifold for the Floer trajectories in this case is $\bR^{2n}$, one must find some mechanism which guarantees compactness of the moduli spaces of Floer trajectories, which is necessary for the definition of the Floer chain complexes and the morphisms between them. This issue is addressed in section \ref{sub:uniform_Linfty_estimates_for_the_floer_equation}. There it is shown that the hypothesis of non-resonance at infinity is necessary for compactness. In fact, non-resonance at infinity implies that the corresponding Hamiltonian action functional is Palais-Smale \cite[Section 3.3]{HoferZehnder2011_SymplecticInvariants} (see also lemma \ref{lem:ndg_at_infty_a_priori_bound_L2}).

The power of Hamiltonian Floer homology relies substantially on its invariance properties. In the case of asymptotically quadratic Hamiltonians, it is useful to introduce the concept of ``non-resonant homotopy'': two asymptotically quadratic Hamiltonians $F_0$ and $F_1$, which are non-resonant at infinity, are said to be \emph{non-resonant homotopic} if there is a path $[0,1]\ni s\mapsto F_s$ of asymptotically quadratic Hamiltonians for which $F_s$ is non-resonant at infinity for all $s\in[0,1]$. If $F_0$ and $F_1$ are non-resonant homotopic, then in \cite{MyPhDThesis} it is shown that there is a chain morphism which induces an isomorphism between the Floer homologies of $F_0$ and $F_1$. The morphism obtained is usually called a \emph{continuation morphism}. Moreover all non-resonant homotopies between two asymptotically quadratic Hamiltonians induce the same continuation morphism. In other words, Floer homology depends only on the \emph{non-resonant homotopy class} of the asymptotically quadratic Hamiltonian.

Using continuation morphisms one may recover the seminal results, mentioned in the introduction, of Amann, Conley and Zehnder \cite{AmannZehnder1980-FinDimRed,AmannZehnder1980-PeriodicSolutionsAsyLin,ConleyZehnder1984_CZ} on the existence and multiplicity of periodic orbits for asymptotically linear Hamiltonian systems. The idea is that an asymptotically quadratic Hamiltonian is non-resonant homotopic to its quadratic Hamiltonian at infinity. The Floer homology of a non-degenerate quadratic Hamiltonian is computed easily from its definition: there is only one generator corresponding to the fixed point at the origin. This implies that the ALHD must have at least one fixed point. Multiple fixed points can be obtained under additional hypotheses on the Conley-Zehnder index of the periodic orbit found.

The existence of a homologically visible twist fixed point is enough to deduce the existence of one periodic orbit with high period, by considering a high enough iterate and then arguing by continuation to the resulting quadratic Hamiltonian at infinity. This sort of argument was already understood by Amann, Conley and Zehnder. We are interested in finding periodic orbits with arbitrarily high period, which can't be found via simple ``global'' continuation. To detect them, in the proof of the Poincaré-Birkhoff theorem we equipped Floer homology with further structure.

We can filter the chain complex using the Hamiltonian action functional, namely, we consider the sub-chain complex generated by orbits with action value below some fixed constant (see section \ref{sub:filtered_floer_homology}). The resulting filtered homology is called \emph{filtered Floer homology} and the filtration is referred to as the \emph{action filtration}. This is a much finer invariant of the Hamiltonian system: at the level of filtered homology, homotopies of Hamiltonians give rise to continuation morphisms only up to a shift in the action filtration, and this \emph{action shift constant} depends in general on the continuation chosen (see section \ref{sub:action_shift_of_continuation_morphisms}). Since we are interested in covering also Hamiltonian systems with possibly degenerate periodic orbits, we have to explain how Floer homology can be extended to these systems by approximation (section \ref{ssub:degenerate_hamiltonians}).  To better understand filtered Floer homology of a degenerate Hamiltonian, we will introduce the local Floer homology of its isolated periodic orbits (section \ref{sub:local_floer_homology}).

\subsection{Compactness for periodic orbits}
\label{sub:compactness_for_periodic_orbits}

Here we show that the action functional associated to an asymptotically quadratic Hamiltonian which is non-resonant at infinity has nice compactness properties which are crucial for the variational study of the periodic orbit problem. This fact is well known and the proof follows the seminal paper \cite{ConleyZehnder1984_CZ}.

First we recall the classical variational set-up for the periodic orbit problem in Hamiltonian systems on $\bR^{2n}$. Let $H\in C^\infty(S^1\times\bR^{2n})$. We consider the Hilbert space $L^2\left(S^1,\bR^{2n}\right)=L^2$ and its dense subspace $W^{1,2}\left(S^1,\bR^{2n}\right)=W^{1,2}$. Define the action functional on $W^{1,2}$ by
\begin{equation}\label{eq:def_of_action_functional}
	\cA_H(x)=\int_{S^1}\frac12\left\langle x(t),J_0\dot x(t)\right\rangle-H_t(x(t))\ dt.
\end{equation}
Suitable growth bounds on the Hamiltonian guarantee that the action functional is a smooth functional \cite[pg. 76]{Abbondandolo2001_MorseHamilt}.
The unregularized gradient of $\cA_H$ is defined by the identity
\begin{equation*}
	\left\langle\nabla_{L^2}\cA_H(x),\xi\right\rangle_{L^2}=d\cA_H|_x\xi\quad\forall\xi\in W^{1,2}.
\end{equation*}
A straightforward calculation shows that
\begin{equation*}
	\nabla_{L^2}\cA_H(x)=J_0\dot x-\nabla H_t\circ x.
\end{equation*}
It follows that the critical points of the action functional are precisely the 1-periodic orbits of the Hamiltonian system associated to $H$.

The central property of the action functional which we will need is the following weak form of the Palais-Smale condition.
\begin{lemma}\label{lem:ndg_at_infty_a_priori_bound_L2}
	Let $H\in C^\infty\left(S^1\times\bR^{2n},\bR\right)$ be a asymptotically quadratic Hamiltonian. If $H$ is non-resonant at infinity, then there exist constants $\nu,\delta>0$ such that
	\begin{equation}\label{eq:ndg_at_infty_implies_L2_properness_nablaA}
		\left\|\nabla_{L^2}\cA_H(x)\right\|_{L^2}=\left\|\dot x-X_H\circ x\right\|_{L^2}\geq\frac\nu2\|x\|_{L^2}-\delta.
	\end{equation}
\end{lemma}
\begin{proof}
	Let $H=Q+h$, where $Q_t(z)=\frac12\left\langle A_tz,z\right\rangle$ for $A\colon S^1\to\Sym(2n)$ a smooth path of symmetric matrices.
	With a small abuse of notation denote by $J_0A$ the operator on $L^2$ given by $x\mapsto J_0A_tx(t)$. Since $H$ is non-resonant at infinity, the operator
	\begin{equation}
		D_A=\frac{d}{dt}+J_0A\colon W^{1,2}\to L^2
	\end{equation}
	is invertible (see e.g. Ekeland \cite[Proposition 2, \S III.1]{Ekeland1990_ConvexityMethods}). Set
	\begin{equation}
		\nu=\frac{1}{\left\|D_A^{-1}\right\|_\op}
	\end{equation}
	We obtain the first estimate
	\begin{equation}\label{eq:ndg_lin_sys_invert_estimate}
		\left\|D_Ax\right\|_{L^2}\geq\nu\|x\|_{L^2}.
	\end{equation}
	Now notice that $\dot x-X_H\circ x-J_0D_Ax=-\nabla H_t\circ x+Ax$. Therefore
	\begin{equation}\label{eq:estimate_L2_properness_nablaA}
		\begin{split}
			\left\|\dot x-X_H\circ x\right\|_{L^2}&=\left\|J_0D_Ax-\left(J_0D_Ax-\dot x+X_H\circ x\right)\right\|_{L^2}\geq\\
			&\geq\left\|D_Ax\right\|_{L^2}-\left\|\nabla H\circ x-Ax\right\|_{L^2}=\left\|D_Ax\right\|_{L^2}-\left\|\nabla h\circ x\right\|_{L^2}.
		\end{split}
	\end{equation}
	Now, since $|\nabla h_t(z)|=o(|z|)$, for every $\epsilon>0$ there exists a $\delta>0$ such that
	\begin{equation}
		\left\|\nabla h\circ x\right\|_{L^2(S^1)}\leq \epsilon\|x\|_{L^2}+\delta.
	\end{equation}
	Set $\epsilon=\frac\nu2$, and combine this estimate with \eqref{eq:ndg_lin_sys_invert_estimate} and \eqref{eq:estimate_L2_properness_nablaA} to get
	\begin{equation*}
		\left\|\dot x-X_H\circ x\right\|_{L^2}\geq \frac\nu2\|x\|_{L^2}-\delta
	\end{equation*}
	as claimed.
\end{proof}
A nice consequence of this lemma is a refinement of lemma \ref{lem:1-per_orbits_of_nondeg_at_infty_asyquad_are_unif_Linfty_bounded} from fixed points to 1-periodic orbits.
\begin{lemma}\label{lem:Linfty_estimate_perorb}
	Let $H\in C^\infty\left(S^1\times\bR^{2n},\bR\right)$ be a asymptotically quadratic Hamiltonian with non-degenerate quadratic Hamiltonian at infinity. There exists a constant $R>0$ such that if $x\colon S^1\to\bR^{2n}$ is a 1-periodic orbit of $H$ then
	\begin{equation}\label{eq:Linfty_estimate_perorb}
		\|x\|_{L^\infty}\leq R.
	\end{equation}
\end{lemma}
\begin{proof}
	Let $x\colon S^1\to\bR^{2n}$ be a 1-periodic orbit of $X_H$.
	From the estimate \eqref{eq:ndg_at_infty_implies_L2_properness_nablaA} applied to $x\in W^{1,2}$, we obtain
	\begin{equation}\label{eq:a_priori_L2_estimate_perorb}
		\|x\|_{L^2}\leq \frac{2\delta}{\nu}=C_0.
	\end{equation}
	Similarly as the previous lemma, notice that
	\begin{equation*}
		D_Ax=J_0\left(A_tx-\nabla H_t\circ x\right).
	\end{equation*}
	Using that $|\nabla h|=o(|z|)$ we obtain
	\begin{equation*}
		\left\|D_Ax\right\|_{L^2}=\left\|\nabla H\circ x-Ax\right\|_{L^2}\leq \|x\|_{L^2}+C_1
	\end{equation*}
	where $C_1>0$ depends only on $H$. Now, $D_A\colon W^{1,2}\to L^2$ is an invertible operator. Moreover any 1-periodic (weak) solution is a bounded function, since $W^{1,2}(S^1,\bR^{2n})$ embeds in $L^\infty(S^1,\bR^{2n})$ by the Sobolev inequalities. Therefore we get
	\begin{equation}\label{eq:Linfty_estimate}
		\begin{split}
			\|x\|_{L^\infty}&\leq C_2\|x\|_{W^{1,2}}\leq C_2\left\|D_A^{-1}\right\|_\op\left\|D_Ax\right\|_{L^2}\leq\\
			&\leq C_3\left(\|x\|_{L^2}+C_1\right)\leq C_3(C_0+C_1)=:R
		\end{split}
	\end{equation}
	where $R$ now depends only on $H$.
\end{proof}
\begin{rmk}
	The example of a quadratic Hamiltonian generating a degenerate symplectic linear map shows that compactness is not to be expected in absence of some non-degeneracy condition.
\end{rmk}

\subsection{Uniform \texorpdfstring{$L^\infty$}{Linfty}-estimates for the Floer equation}
\label{sub:uniform_Linfty_estimates_for_the_floer_equation}

We derive the uniform $L^\infty$ bounds in the most general case we need, that of a Floer cylinder solving the continuation Floer equation between asymptotically quadratic Hamiltonians with non-degenerate quadratic forms at infinity. The proof of these estimates follows Abbondandolo and Kang's more general treatment in \cite[Section 3]{AbbondandoloKang2022_SHClarke}.

\begin{defn}\label{def:asyquad_continuation}
	Consider a smooth function $\cH\colon\bR\times S^1\times \bR^{2n}\to\bR$, $\cH(s,t,z)=\cH^s_t(z)$, with the following properties.
	\begin{enumerate}
		\item The function $\cH$ depends on the $s\in\bR$ coordinate only in a compact interval $\cS\subset\bR$. We set $H^0_t(z)=\cH^s_t(z)$ for some (and then all) $s<\min \cS$ and $H^1_t(z)=\cH^s_t(z)$ for some (and then all) $s>\max \cS$.
		\item There exists a smooth $\bA=\bA^s_t\colon\bR\times S^1\to\Sym(2n)$ such that
			\begin{equation*}
				\sup_{(s,t)\in\bR\times S^1}\left|\nabla\cH^s_t(z)-\bA^s_tz\right|=o(|z|)\quad\text{as }|z|\to\infty.
			\end{equation*}
			Also $\bA$ must depend on $s\in\bR$ only in the interval $\cS\subset\bR$. We set $A^0_t=\bA^s_t$ for some $s<\min \cS$ and $A^1_t=\bA^s_t$ for some $s>\max \cS$.
		\item The asymptotically quadratic Hamiltonians $H^i$ are non-resonant at infinity.
	\end{enumerate}
	Such an $\cH$ is called a \emph{asymptotically quadratic continuation} between asymptotically quadratic Hamiltonians $H^0$ and $H^1$.
\end{defn}
For example, a non-resonant homotopy gives rise naturally to a asymptotically quadratic continuation. But notice that for the time being, if $\cH=\cH^s$ is a asymptotically quadratic continuation, we are \emph{not} assuming that $\cH^s$ is non-resonant at infinity for every $s\in\bR$.

We also fix the behaviour of the families of almost-complex structures required to define the Floer equation.
\begin{defn}\label{def:adequate_ACS}
	A family of almost-complex structures $\cJ\colon\bR\times S^1\times\bR^{2n}\to\End\bR^{2n}$, $\cJ(s,t,z)=\cJ_{s,t}(z)$ is said to be \emph{adequate} if
	\begin{enumerate}
		\item $\cJ$ depends on the $s\in\bR$ coordinate only on a compact interval $\cS\subset\bR$.
		\item $\cJ_{s,t}$ is $\omega_0$-compatible for all $(s,t)\in\bR\times S^1$ \cite[\S 2.5]{McDuffSalamon2017_Intro}.
		\item $\|\cJ\|_{L^\infty}<\infty$.
	\end{enumerate}
	We say that a path of almost-complex structures $J\colon S^1\times\bR^{2n}\to\End\bR^{2n}$ is adequate if the corresponding $s$-constant family $\cJ_{s,t}\equiv J_t$ is adequate.
\end{defn}
Let $\cJ$ be an adequate almost complex structure. Since $\cJ_{s,t}$ is $\omega_0$-compatible for all $(s,t)$, there is a canonically defined family of Riemannian metrics on $\bR^{2n}$, given by
\begin{equation}\label{eq:def_of_g_J}
	g_\cJ(s,t,z)\left(u,v\right)=\omega_0\left(u,\cJ_{s,t}(z)v\right).
\end{equation}
Often we suppress the dependence on $(s,t)$ and simply write $g_\cJ$. The associated family of norms is denoted by $|\cdot|_{g_\cJ}$. Notice that since $\|\cJ\|_{L^\infty}<\infty$, there exists a $b>1$ such that
\begin{equation}\label{eq:g_J_unif_equiv_to_eucl}
	\frac{1}{b}|v|\leq |v|_{g_\cJ(s,t)}\leq b|v|\quad\forall (s,t)\in\bR\times S^1\quad\forall v\in\bR^{2n}.
\end{equation}

A map $u\colon\bR\times S^1\to\bR^{2n}$ is said to solve the \emph{continuation} Floer equation for $\cH,\cJ$ as above when
\begin{equation}\label{eq:fleq_contin}
		\partial_su+\cJ_{s,t}(u)\left[\partial_tu-X_{\cH}(s,t,u)\right]=0.
\end{equation}
A relevant special case of these definitions is when neither the Hamiltonian nor the almost complex structure depend on $s\in\bR$, in other words, when $H,J$ are just a Hamiltonian and a ($t$-dependent) $\omega_0$-compatible and bounded almost complex structure. In this case, a map $u\colon\bR\times S^1\to\bR^{2n}$ is said to be a \emph{Floer trajectory} for $H,J$ when
\begin{equation}\label{eq:fleq_auton}
	\partial_su-J_t(u)\left[\partial_tu-X_H(t,u)\right]=0.
\end{equation}

For a map $u\colon\bR\times S^1\to\bR^{2n}$ we define the $(\cH,\cJ)$-energy as
\begin{equation}\label{eq:H_J_energy}
	E_{\cH,\cJ}(u)=\frac12\int_{\bR\times S^1}\left|\partial_su(s,t)\right|^2_{g_\cJ}+\left|\partial_tu(s,t)-X_{\cH}(s,t,u(s,t))\right|^2_{g_\cJ}dsdt.
\end{equation}
If $u$ solves the continuation Floer equation for $\cH$, $\cJ$, then
\begin{equation}\label{eq:energy_identity_Fcyls}
	E_{\cH,\cJ}(u)=\int_{\bR\times S^1}\left|\partial_su\right|^2_{g_\cJ}dsdt=\|\partial_su\|^2_{L^2(g_\cJ)}.
\end{equation}
Therefore an uniform bound on the energy corresponds to an uniform bound on the $L^2$-norm (with respect to the euclidean metric, by equivalence of norms) of $\partial_su$. Similar considerations hold for Floer trajectories.

\begin{prop}\label{prop:unif_Linfty_bounds}
	Let $\cH\colon\bR\times S^1\times\bR^{2n}\to\bR$ be an asymptotically quadratic continuation between asymptotically quadratic Hamiltonians $H^0,H^1$ which are non-resonant at infinity, and $\cJ\colon\bR\times S^1\times \bR^{2n}\to\End\bR^{2n}$ an adequate family of almost complex structures. For every $E>0$ there exists an $M>0$ with the following significance. If $u\colon\bR\times S^1\to\bR^{2n}$ solves the continuation Floer equation \eqref{eq:fleq_contin} and its $(\cH,\cJ)$-energy of $u$ is bounded by $E$,
	\begin{equation}\label{eq:unif_Linfty_bounds_energy_bd_asspt}
		E_{\cH,\cJ}(u)=\int_{\bR\times S^1}\left|\partial_su\right|^2_{g_\cJ}<E
	\end{equation}
	then we have
	\begin{equation*}
		\left\|u\right\|_{L^\infty}<M.
	\end{equation*}
\end{prop}
\begin{proof}
	The proof follows \cite[Proposition 3.1]{AbbondandoloKang2022_SHClarke}.
	We can assume that the interval $\cS\subset\bR$ where $\cH$ depends on $s$ is the same as the one for $\cJ$.
	By equations \eqref{eq:g_J_unif_equiv_to_eucl}, \eqref{eq:energy_identity_Fcyls} and \eqref{eq:unif_Linfty_bounds_energy_bd_asspt} we conclude
	\begin{equation}\label{eq:energy_estimate_L2_control}
		\left\|\partial_su\right\|_{L^2}^2=\int_{\bR\times S^1}\left|\partial_su\right|^2<E'
	\end{equation}
	where $E'=b^2E$. Since each $H^i$ has non-degenerate quadratic Hamiltonian at infinity, there exist constants $\nu,\delta>0$ such that
	\begin{equation}\label{eq:unif_PS_on_cAcH}
		\left\|\nabla_{L^2}\cA_{\cH^s}(x)\right\|_{L^2(S^1)}\geq \frac\nu2\|x\|_{L^2}-\delta \quad\forall x\in W^{1,2}(S^1,\bR^{2n})\quad \forall s\in\bR\setminus \cS.
	\end{equation}
	Here the weak-$L^2$ gradient is with respect to the $L^2$-metric arising from the euclidean inner product. Notice that $u(s,\cdot)\in W^{1,2}(S^1)$ for all $s\in\bR$ by the regularity theory of the Floer equation. Since $u$ solves the Floer equation, and using again \eqref{eq:g_J_unif_equiv_to_eucl},
	\begin{equation}\label{eq:unif_PS_along_sol}
		\|\partial_su(s,\cdot)\|_{L^2(S^1)}\geq B\left\|\nabla_{L^2}\cA_{\cH^s}\left(u(s,\cdot)\right)\right\|_{L^2(S^1)}\geq \frac{\nu'}{2}\|u(s,\cdot)\|_{L^2(S^1)}-\delta'\quad\forall s\in\bR\setminus \cS.
	\end{equation}
	Consider an arbitrary $\alpha>0$ to be determined later. Define
	\begin{equation*}
		S_\alpha=\left\{s\in\bR:\|\partial_su(s,\cdot)\|_{L^2(S^1)}\leq \alpha\right\}.
	\end{equation*}
	Notice that \eqref{eq:unif_PS_along_sol} implies
	\begin{equation}\label{eq:L2S1_estimate_over_S}
		\|u(s,\cdot)\|_{L^2(S^1)}\leq\frac{2\left(\alpha+\delta'\right)}{\nu'}=:R_\alpha\quad \forall s\in S_\alpha\setminus \cS.
	\end{equation}
	The length of the complement of $S_\alpha$ can be estimated as
	\begin{equation*}
		\left|\bR\setminus S_\alpha\right|\leq \frac{1}{\alpha^2}\int_\bR\|\partial_su(s,\cdot)\|^2_{L^2(S^1)}ds\leq \frac{E'}{\alpha^2}=L_\alpha.
	\end{equation*}
	Therefore if $I\subset\bR$ is an interval of length $|I|=L>L_\alpha$, the previous estimate implies that $S_\alpha\cap I\neq\varnothing$ so there exists some $s_*\in S_\alpha\cap I$. Now, if we fix $\alpha<\sqrt{\frac{E'}{\left|\cS\right|}}$, then $L_\alpha>\left|\cS\right|$ so we can assume that $s_*\in \left(S_\alpha\setminus\cS\right)\cap I$. Using the identity
	\begin{equation*}
		u(s,t)=u(s_*,t)+\int_{s_*}^s\partial_su(\sigma,t)d\sigma
	\end{equation*}
	and estimate \eqref{eq:L2S1_estimate_over_S} we obtain $\forall s\in I$
	\begin{equation*}
		\begin{split}
			\|u(s,\cdot)\|_{L^2(S^1)}^2&=\int_{S^1}|u(s,t)|^2dt\leq 2\left(\int_{S^1}|u(s_*,t)|^2dt+\int_{S^1}\int_{s_*}^s|\partial_su(\sigma,t)|^2d\sigma dt\right)\leq\\
			&\leq 2\left(R_\alpha^2+LE'\right)=B_0^2.
		\end{split}
	\end{equation*}
	Integrating over $I$, we obtain that if $L_\alpha<|I|<\infty$ then
	\begin{equation}\label{eq:unif_L2_estimate_on_bded_interval}
		\|u\|_{L^2(I\times S^1)}\leq \sqrt{|I|}B_0.
	\end{equation}
	Now, let $I\subset\bR$ be an interval of length $L=L_\alpha+1$ and $I'\supset I$ an interval of length at most $2L$. Denote $\delbar_\cJ u=\partial_su+\cJ_{s,t}(u)\partial_tu$ the Cauchy-Riemann operator associated to $\cJ$. By the Calderón-Zygmund inequality \cite[\S B.2]{McDuffSalamon2012_BigJ}, for every $p\in[2,\infty)$ there exists a constant $C_p>0$ such that
	\begin{equation}\label{eq:caldzyg_2}
		\|\nabla u\|_{L^p\left(I\times S^1\right)}\leq C_p\left[\|u\|_{L^p(I'\times S^1)}+\left\|\delbar_\cJ u\right\|_{L^p(I'\times S^1)}\right].
	\end{equation}
	The constant $C_p$ depends only on $p$ and the length of $I$, i.e. $L$. Moreover $\delbar_\cJ u=\nabla_\cJ\cH\circ u$, where $\nabla_\cJ\cH$ is the gradient of $\cH$ with respect to $g_\cJ$. Since $\nabla\cH$ is sub-linear and, by \eqref{eq:g_J_unif_equiv_to_eucl}, $g_\cJ$ is uniformly equivalent to the euclidean metric, we can estimate
	\begin{equation}\label{eq:estimate_Lp_delbarJu}
		\begin{split}
			\left\|\delbar_\cJ u\right\|_{L^p(I'\times S^1)}&=\left\|\left(\nabla_\cJ\cH\right)\circ u\right\|_{L^p(I'\times S^1)}\leq\\
			&\leq \left\|\left(\nabla_\cJ\cH\right)\circ u-\bA u\right\|_{L^p(I'\times S^1)}+\left\|\bA u\right\|_{L^p(I'\times S^1)}\leq\\
			&\leq B_1\left(\|u\|_{L^p(I'\times S^1)}+1\right)
		\end{split}
	\end{equation}
	where $B_1>0$ depends only on $\cH$ and $\cJ$. Now, fixing $p=2$ and using the estimates \eqref{eq:unif_L2_estimate_on_bded_interval}, \eqref{eq:caldzyg_2}, \eqref{eq:estimate_Lp_delbarJu}, we obtain
	\begin{equation}
		\begin{split}
			\|u\|_{W^{1,2}(I\times S^1)}&\leq \|u\|_{L^2(I\times S^1)}+\|\nabla u\|_{L^2(I\times S^1)}\leq\\
			&\leq\sqrt{2L}B_0+C_p\left[\sqrt{2L}B_0+B_1\left(\sqrt{2L}B_0+1\right)\right]=M_1.
		\end{split}
	\end{equation}
	By Sobolev embedding there is some constant $R_p>0$ such that
	\begin{equation}\label{eq:unif_Lp_bound_u}
		\|u\|_{L^p\left(I\times S^1\right)}\leq R_p\|u\|_{W^{1,2}(I\times S^1)}=R_pM_1
	\end{equation}
	where $R_p$ depends only on the length of $I$. So applying Calderón-Zygmund again we have
	\begin{equation}
		\|u\|_{W^{1,p}(I\times S^1)}\leq \|u\|_{L^p(I'\times S^1)}+\|\nabla u\|_{L^p(I'\times S^1)}\leq M_2
	\end{equation}
	with $M_2$ depending only on $p$ and the length $L$ of $I$. We are now allowed to take a fixed $p>2$ and use the Sobolev embedding theorem again to find a constant $B_p>0$ such that:
	\begin{equation*}
		\|u\|_{L^\infty\left(I\times S^1\right)}\leq B_p\|u\|_{W^{1,p}\left(I\times S^1\right)}\leq B_pM_2=M.
	\end{equation*}
	Recall that the Sobolev constant $B_p$ depends only on the length of $I$. Covering $\bR$ by intervals of length $L\leq I\leq 2L$ therefore supplies us with the wanted estimate. All in all, since $L$ depends only on $E$ and $\cJ$, we have that $M$ depends only on $E$, $\cJ$ and $\cH$. The estimate is therefore independent of the particular solution $u$.
\end{proof}

\subsection{Floer chain complex and Floer homology}
\label{sub:floer_chain_complex_and_floer_homology}

The uniform $L^\infty$-estimates of proposition \ref{prop:unif_Linfty_bounds} suffice to guarantee that standard techniques like in \cite[Part II]{AudinDamian2014_MorseFloer} can be applied to construct the Floer homology groups, provided an uniform energy bound on the moduli space of Floer trajectories we are using is obtained. Here we summarize very briefly this construction, pointing the interested reader to references in the literature for additional details.

Recall that a Hamiltonian $H$ is said to be \emph{non-degenerate} when all its 1-periodic orbits are non-degenerate, in the sense that their linearized Poincaré return map does not have $1$ as an eigenvalue. In this section, $H$ always denotes a non-degenerate asymptotically quadratic Hamiltonian which is non-resonant at infinity. For the time being, rapid decay conditions on the sub-quadratic part play no role.

Let $J$ be an adequate almost-complex structure. It can be shown that the Floer trajectories of $H,J$ with finite energy connect 1-periodic orbits of $H$ \cite[Section 6.5.b]{AudinDamian2014_MorseFloer}. More specifically, if $u\colon\bR\times S^1\to\bR^{2n}$ solves \eqref{eq:fleq_auton} and has finite energy, then there exist unique 1-periodic orbits $\gamma^-,\gamma^+$ of $H$ such that $u(s,\cdot)\to\gamma^\pm$ as $s\to\pm\infty$. On the other hand, let $\gamma^-$ and $\gamma^+$ be 1-periodic orbits of $H$. Assume that a Floer trajectory $u$ connects $\gamma_-$ to $\gamma_+$. Then, since the Floer equation is formally the $L^2$-gradient descent equation of the action functional $\cA_H$, it will have energy given by
\begin{equation}\label{eq:auton_FlCyl_energy}
	E_{H,J}(u)=\cA_H\left(\gamma_-\right)-\cA_H\left(\gamma_+\right).
\end{equation}

One may show that for a generic choice of adequate almost-complex structure $J$, the moduli spaces of Floer trajectories between two fixed 1-periodic orbits of a Hamiltonian $H$ are smooth, with their dimension given by the difference of the Conley-Zehnder index of the orbits at the ends \cite[Chapter 8]{AudinDamian2014_MorseFloer} (or \cite{FloerHoferSalamon1995_Transv}). A pair $(H,J)$ with this property will be called a \emph{regular pair}.

Since we have an uniform estimate on the energy \eqref{eq:auton_FlCyl_energy} over the space of Floer trajectories connecting two given 1-periodic orbits, proposition \ref{prop:unif_Linfty_bounds} holds. This is enough input to apply the standard compactness theory of the Floer equation in the symplectically aspherical case \cite[Section 9.1.b]{AudinDamian2014_MorseFloer}. Thus we can compactify these spaces of trajectories by adding broken configurations of trajectories. Their boundary, at least for dimension 1 and 2, can be described using the gluing theory of Floer trajectories \cite[Section 9.2]{AudinDamian2014_MorseFloer}. This information can be used to carry out the construction of the Floer chain complex.

Let $(H,J)$ be a regular pair. The Floer chain complex $\left(\CF_*(H,J),d_{H,J}\right)$ is defined as follows: $\CF_k(H,J)$ is the free $\bZ/2$-vector space generated by the 1-periodic orbits $\gamma$ of $H$ of index $\CZ(\gamma)=k$. The differential $d_{H,J}:\CF_k(H,J)\to\CF_{k-1}(H,J)$ is defined extending by linearity the following expression on generators:
\begin{equation}\label{eq:def_of_Floer_differential}
	d_{H,J}\gamma=\sum_{\CZ(\xi,H)=k-1}N_{H,J}(\gamma,\xi)\xi
\end{equation}
where $N_{H,J}(\gamma,\xi)$ is found by considering the smooth 1-dimensional manifold of Floer trajectories which start at $\gamma$ and end at $\xi$, quotienting out the natural $\bR$-action on Floer trajectories given by shifting in the $s$-direction, and counting the points in the resulting compact $0$-dimensional space. The parity of this number is $N_{H,J}(\gamma,\xi)$.
Since $H$ is non-degenerate, all its 1-periodic orbits are isolated. Since it is non-resonant at infinity, by lemma \ref{lem:Linfty_estimate_perorb} they are contained in a compact set. In particular $H$ has finitely many 1-periodic orbits. Therefore $\CF_k(H,J)$ is finitely generated.

Using the description of the boundary via gluing one may show that $d^2_{H,J}=0$ \cite[Corollary 9.2.2]{AudinDamian2014_MorseFloer}. The homology of the resulting chain complex is Floer homology $\HF_*(H)$.

\subsection{Continuations}
\label{sub:continuations}

Invariance of Floer homology is studied via the definition of continuation morphisms. These are morphisms
\begin{equation}\label{eq:notation_for_contin_morph}
	\cC(\cH,\cJ)\colon\CF_*(H^0,J^0)\to\CF_*(H^1,J^1)
\end{equation}
induced by the choice of a suitable asymptotically quadratic continuation $\cH$ between $H^0$ and $H^1$ and an adequate almost-complex structure $\cJ$ between $J^0$ and $J^1$. They are defined in terms of counts of the cardinality of zero-dimensional moduli spaces of continuation Floer trajectories solving the continuation Floer equation \eqref{eq:fleq_contin} for the pair $(\cH,\cJ)$ and asymptotic to fixed 1-periodic orbits of the same index. We must guarantee compactness of these moduli spaces. Proposition \ref{prop:unif_Linfty_bounds} gives us the necessary initial input for the compactness theory of this moduli problem. For it to hold, we must show the existence of uniform energy bounds across the moduli space in analysis.

Fix 1-periodic orbits $\gamma^0,\gamma^1$ of $H^0,H^1$. Consider a continuation Floer cylinder $u\in C^\infty(\bR\times S^1,\bR^{2n})$, defined by a continuation Hamiltonian $\cH$ and a generic choice of adequate almost complex structure $\cJ$, with $u(s,\cdot)\to\gamma^0$ as $s\to-\infty$ and $u(s,\cdot)\to\gamma^1$ as $s\to +\infty$. A simple calculation in this case gives that
\begin{equation}\label{eq:energy_calc_contin}
	E_{\cH,\cJ}(u)\leq \cA_{H^0}\left(\gamma^0\right)-\cA_{H^1}\left(\gamma^1\right)-\int_{\bR\times S^1}\partial_s\cH(s,t,u(s,t))dsdt
\end{equation}
We therefore must exhibit continuation Hamiltonians $\cH$ for which it is possible to use the above estimate to reach a uniform bound on the energy. When $H^0$ and $H^1$ are general asymptotically quadratic Hamiltonians which are non-resonant at infinity, and $\mathcal{H}$ is a non-resonant homotopy between them, this is done in \cite{MyPhDThesis}. For the purposes of the present work, this degree of generality is not needed.

In the proof of the Poincaré-Birkhoff theorem for rapidly asymptotically unitary Hamiltonian diffeomorphisms, we only need continuations between rapidly asymptotically quadratic Hamiltonians, which moreover have the \emph{same quadratic Hamiltonian at infinity}. Therefore for the purposes of this paper, it suffices to assume that $H^i=Q+h^i$ where $h^i$ are smooth, \emph{bounded} functions -- we don't actually need rapid decay here. For this kind of Hamiltonians, we can define a very simple kind of continuation Hamiltonian
\begin{equation}\label{eq:same_Q_at_infty_continuation}
	\cH^s_t(z)=Q_t(z)+(1-\chi(s))h^0_t(z)+\chi(s)h^1_t(z)
\end{equation}
where $\chi\colon\bR\to[0,1]$ is a smooth non-decreasing function such that $\chi(s)=0$ for all $s\leq 0$ and $\chi(s)=1$ for all $s\geq 1$.
For this special kind of continuation Hamiltonian, we can further estimate the last term in \eqref{eq:energy_calc_contin} as
\begin{equation}
	\begin{split}
		\int_{\bR\times S^1}\partial_s\cH^s_t(u(s,t))dsdt&=\int_0^1\int_0^1\chi'(s)\left[h^1_t(u(s,t))-h^0_t(u(s,t))\right]dsdt\leq\\
		&\leq\left\|h^1-h^0\right\|_{L^\infty}=\left\|H^1-H^0\right\|_{L^\infty}<\infty.
	\end{split}
\end{equation}
This implies the following bound on the energy of a solution of the Floer equation:
\begin{equation}\label{eq:moduli_continuation_unif_energy_estimate}
	E_{\cH,\cJ}(u)\leq \cA_{H^0}\left(\gamma^0\right)-\cA_{H^1}\left(\gamma^1\right)+\left\|H^1-H^0\right\|_{L^\infty}.
\end{equation}
We have reached an uniform energy bound over all continuation Floer trajectories between $\gamma^0$ and $\gamma^1$.

\begin{rmk}
	Another case where compactness will hold is when one may arrange the Hamiltonians $H^0$, $H^1$ and the continuation Hamiltonian $\cH$ such that (in the conventions used in this paper) $\partial_s\cH>0$. This is usually called \emph{monotone continuation}, and is used for example in the definition of symplectic homology.
\end{rmk}

\subsubsection{Definition of the morphism and functoriality}
\label{ssub:definition_of_the_morphism_and_functoriality}

Given the uniform energy bound, which as remarked above can be obtained in much greater generality, we can proceed as in \cite[Chapter 11]{AudinDamian2014_MorseFloer}. Namely, transversality, compactness and gluing theorems can be shown for the continuation Floer equation, up to picking a generic $\cJ$. We can thus define a continuation morphism
\begin{equation}
	\cC(\cH,\cJ)\colon\CF_*\left(H^0,J^0\right)\to\CF_*\left(H^1,J^1\right)
\end{equation}
extending by linearity the following expression on generators $\gamma^0\in\CF_k\left(H^0,J^0\right)$:
\begin{equation}\label{eq:def_of_contin_morphism_chain}
	\cC\left(\cH,\cJ\right)\gamma^0=\sum_{\CZ\left(\gamma^1,H^1\right)=k}\mathcal N_{\cH,\cJ}\left(\gamma^0,\gamma^1\right)\gamma^1
\end{equation}
where $\mathcal N_{\cH,\cJ}\left(\gamma^0,\gamma^1\right)$ is the parity of the cardinality of the compact $0$-dimensional manifold of $(\cH,\cJ)$-Floer trajectories starting in $\gamma^0$ and ending in $\gamma^1$. That this morphism is indeed a chain morphism follows from a standard gluing argument \cite[Section 11.1, pg. 395]{AudinDamian2014_MorseFloer}.

The morphism induced by the constant homotopy $\cH^s\equiv H$ is easily seen to be the identity on $\CF_*(H,J)$ irregardless of the choice of adequate almost complex structure $\cJ$ \cite[Proposition 11.1.14]{AudinDamian2014_MorseFloer}. This fact implies that Floer homology does not depend on the almost complex structure $J$. Moreover, if we pick continuation data $(\cH,\cJ)$ and $(\cG,\cI)$ between the same Hamiltonians $H^0$ and $H^1$ which give rise to an uniform energy bound on their respective spaces of continuation Floer trajectories, then the morphisms they define are chain homotopic. Hence at the level of homology they are the same morphism $\cC\colon\HF_*(H^0)\to\HF_*(H^1)$ \cite[Proposition 11.2.8]{AudinDamian2014_MorseFloer}.

Continuation morphisms are in some sense functorial with respect to concatenation, as we shall see next. Given a further regular pair of Hamiltonian and almost complex structure $(H^2,J^2)$, consider continuation Hamiltonians $\cH^{01},\cH^{12}$ connecting $H^0$ to $H^1$ and $H^1$ to $H^2$ respectively. Assume that they give rise to an uniform energy bound across their moduli spaces of Floer trajectories, for example we could assume they are of the form \eqref{eq:same_Q_at_infty_continuation}. Denote by $\cC_{ij}\colon\CF_*\left(H^i,J^i\right)\to\CF_*\left(H^j,J^j\right)$ the continuation morphisms they give rise to. Since the ``positive end'' of $\cH^{01}$ and the ``negative end'' of $\cH^{12}$ are both equal to $H^1$, we can define a ``concatenation'' of the continuation Hamiltonians $\cH^{02}$ which is equal to $H^1$ on a long ``neck'', like in the following sketch:
\begin{center}
	\vspace{1.3em}
	\begin{overpic}[width=0.8\linewidth]{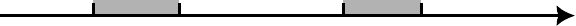}
		\put (-10, 3.5) {$\mathcal{H}^{02}=$}
		\put (5, 3.5) {$H^0$}
		\put (21,3.5) {$\cH^{01}$}
		\put (43,3.5) {$H^1$}
		\put (31.5,1) {$\underbrace{\qquad\qquad\qquad\qquad}_{s_0}$}
		\put (64,3.5) {$\cH^{12}$}
		\put (83,3.5) {$H^2$}
		\put (100,0) {$s\in\bR$}
	\end{overpic}
	\vspace{1.3em}
\end{center}
It is very simple to see that an uniform energy bound of the spaces of its Floer trajectories can be reached. In particular, it defines a continuation morphism $\cC_{02}\colon\CF_*\left(H^0,J^0\right)\to\CF_*\left(H^2,J^2\right)$ which on the level of homology does not depend on the particular continuation Hamiltonian chosen. Studying the breaking of trajectories as $s_0\to\infty$ for this special kind of concatenated continuation, one may show that $\cC_{12}\circ\cC_{01}$ is chain homotopic to $\cC_{02}$, so at the level of homology they are equal \cite[Proposition 11.2.9]{AudinDamian2014_MorseFloer}.

\subsubsection{Continuation isomorphisms and ``total'' Floer homology}
\label{ssub:continuation_isomorphisms_and_total_floer_homology}

We state now an important consequence of the above mentioned properties. Consider a continuation $\cC(\cH,\cJ)\colon\CF_*\left(H^0,J^0\right)\to\CF_*\left(H^1,J^1\right)$. We can ``reverse'' the continuation Hamiltonian, setting $\overline{\cH}(s,t,z)=\cH(1-s,t,z)$. This defines a continuation $\cC(\overline\cH,\cI)\colon\CF_*\left(H^1,J^1\right)\to\CF_*\left(H^0,J^0\right)$ for a generically chosen adequate almost complex structure $\cI$. The above mentioned results imply that these two morphisms are chain homotopy inverses, so on the level of homology they are inverses. In particular, in the class of non-degenerate asymptotically quadratic and non-resonant at infinity Hamiltonians with the \emph{same} quadratic Hamiltonian at infinity, every continuation morphism is an isomorphism. The same holds true for general non-degenerate asymptotically quadratic and non-resonant at infinity Hamiltonians which are non-resonant homotopic \cite{MyPhDThesis}. Since two such Hamiltonians are non-resonant homotopic if and only if they have the same index at infinity, their Floer homology depends only on their index at infinity. This fact may be used to recover the results of Amann, Conley and Zehnder on 1-periodic orbits in asymptotically linear Hamiltonian systems \cite{AmannZehnder1980-FinDimRed,AmannZehnder1980-PeriodicSolutionsAsyLin,ConleyZehnder1984_CZ}. Let us explain this property for the simpler case of bounded non-quadratic part.

Let $H=Q+h$ be non-degenerate, asymptotically quadratic, non-resonant at infinity and with $h$ bounded. Using a continuation which homotopes $h$ to $0$, we can compute $\HF_*(H)$ directly from the quadratic form $Q$:
\begin{equation}\label{eq:HF_calc_global}
	\HF_*\left(H\right)\cong\HF_*\left(Q\right)=
	\begin{cases}
		\bZ/2, &*=\ind_\infty(H)\\
		0, &\text{otherwise}.
	\end{cases}
\end{equation}
Indeed, as a non-degenerate Hamiltonian $Q$ only has $0\in\bR^{2n}$ as a 1-periodic orbit, and the Conley-Zehnder index of $0$ as a 1-periodic orbit of $Q$ is exactly $\ind_\infty H$. The exact same result holds for a non-degenerate, non-resonant at infinity Hamiltonian $H$ with no additional hypothesis on the sub-quadratic part.

Using this calculation, we can give a remarkably simple proof of the following lemma, which we used in section \ref{sub:a_twist_condition} to show that the definition of twist fixed points is independent of the choice of generating Hamiltonian. This argument was inspired by \cite[Section 2.3.1]{Ginzburg2010_Conley}.
\begin{lemma}\label{lem:loops_ALHDs_maslov}
	Let $K=R+k$ be an asymptotically quadratic Hamiltonian. Assume that $\phi^1_K=\id_{\bR^{2n}}$, so $\phi^t_K$ defines a loop in the space of asymptotically linear Hamiltonian diffeomorphisms. It holds that
	\begin{equation}\label{eq:loops_ALHDs_maslov_pts}
		\Mas\left(d\phi^t_K(z_0)\right)=\Mas\left(d\phi^t_K(z_1)\right)\quad\forall z_0,z_1\in\bR^{2n}.
	\end{equation}
	Denote by $\Mas K\in\bZ$ this Maslov index. It holds that
	\begin{equation}\label{eq:loops_ALHDs_maslov_infty}
		\Mas\left(\phi^t_R\right)=\Mas K.
	\end{equation}
\end{lemma}
\begin{proof}
	The first claim follows because the two loops $d\phi^t_K(z_0)$ and $d\phi^t_K(z_1)$, $t\in[0,1]$, are homotopic in $\Sp(2n)$. We show the second claim.
	Choose an arbitrary non-degenerate asymptotically linear Hamiltonian $H=Q+h$ which is non-resonant at infinity and an adequate almost complex structure $J$ which is regular for $H$. Let $H'=K\#H$, $\sigma=2\Mas K$ and $J'_t=\left(\phi^t_K\right)^*J$. It is easy to see that $J'$ is regular for $H'$ (see e.g. \cite[Lemma 4.1]{Seidel1997_pi1Ham_QH}). A loop $\gamma$ is a 1-periodic orbit of $H$ if and only if $\gamma'(t)=\phi^t_K(\gamma(t))$ is a 1-periodic orbit of $H'$. Their Conley-Zehnder indices are related by the usual loop composition formula
	\begin{equation}
		\CZ\left(\gamma'\right)=\CZ\left(\gamma\right)+\sigma.
	\end{equation}
	Moreover, a quick computation shows that a map $u$ is an $(H,J)$-Floer trajectory if and only if the map $v(s,t)=\phi^t_K(u(s,t))$ is a $(H',J')$-Floer trajectory. Composing with the flow of $K$ thus gives rise to a chain isomorphism
	\begin{equation}\label{eq:loops_ALHDs_chcplx_iso}
		\CF_*\left(H',J'\right)\cong\CF_{*-\sigma}\left(H,J\right).
	\end{equation}
	Now, notice that
	\begin{equation}
		\ind_\infty H'=\ind_\infty H+2\Mas\left(\phi^t_R\right).
	\end{equation}
	So at the level of homology, setting $\sigma'=2\Mas\left(\phi^t_R\right)$,
	\begin{equation}
		\HF_{*-\sigma}\left(H'\right)\cong\HF_*\left(H'\right)\cong\HF_*\left(R\#Q\right)=\HF_{*-\sigma'}\left(Q\right)\cong\HF_{*-\sigma'}(H).
	\end{equation}
	Using \eqref{eq:HF_calc_global} this is only possible if $\sigma'=\sigma$.
\end{proof}

\subsection{Filtered Floer homology}
\label{sub:filtered_floer_homology}

Equation \eqref{eq:HF_calc_global} implies the existence of one 1-periodic orbit for non-degenerate, non-resonant at infinity asymptotically quadratic Hamiltonians. To gain further information on the finer structure of the periodic orbits of asymptotically linear Hamiltonian systems, we must equip the Floer chain complex with the natural \emph{action filtration} arising from the action functional.

Here and below all Hamiltonians we consider are asymptotically quadratic, non-resonant at infinity and non-degenerate, and moreover all have the same quadratic Hamiltonian at infinity and bounded sub-quadratic part. We will refer to them simply as ``Hamiltonians''. General asymptotically linear Hamiltonian systems can be treated given suitable energy estimates, found in \cite{MyPhDThesis}.

Let $H$ be a Hamiltonian and $J$ a generically chosen almost complex structure. Since $\CF_k(H,J)$ is generated by 1-periodic orbits of $H$ of index $k$, for any $a\in\bR$ we can consider the subspace $\CF^{(-\infty,a]}_k(H,J)$ generated by the 1-periodic orbits $\gamma$ of action $\cA_H(\gamma)\leq a$ and index $k$. The energy of an $(H,J)$-Floer trajectory $u$ connecting two 1-periodic orbits $\gamma_0$ at $s=-\infty$ and $\gamma_1$ at $s=+\infty$ is given by the difference of the actions of these orbits, as seen in equation \eqref{eq:auton_FlCyl_energy}. This implies that the differential $d_{H,J}$ decreases the action, so $\left(\CF^{(-\infty,a]}_*(H,J),d_{H,J}\right)$ is a sub-complex. Set, for $b>a$,
\begin{equation}
	\CF^{(a,b]}_*\left(H,J\right)=\left.\CF^{(-\infty,b]}_*(H,J)\right/\CF^{(-\infty,a]}_*(H,J).
\end{equation}
As a vector space, it is (freely) generated by the orbits with action in $(a,b]$. We endow $\CF^{(a,b]}_*(H,J)$ with the quotient differential. Its homology is called \emph{filtered Floer homology}, denoted by $\HF^{(a,b]}_*(H)$. The classes in this homology carry an action value, usually referred to as the \emph{spectral number}, defined for $\alpha\in\HF^{(a,b]}(H),\alpha\neq 0$ by
\begin{equation}\label{eq:spectral_nr_class}
	\cA_H\left(\alpha\right)=\inf_{\sigma\in\CF^{(a,b]}(H,J):\left[\sigma\right]=\alpha}\max\left\{\cA_H\left(\gamma\right):\sigma_\gamma\neq0\text{ where }\sigma=\sum_{\gamma\in\operatorname{Per}^1H}\sigma_\gamma\gamma\right\}
\end{equation}
and $\cA_H(0)=-\infty$. This infimum is always attained, so it is always the action value of some 1-periodic orbit of $H$ (see e.g. \cite{Schwarz200_OnTheActionSpectrum,Usher2008_SpecNrs}).

\subsubsection{Inclusion-quotient morphism}
\label{ssub:inclusion_quotient_morphism}

First a trivial algebraic observation.

If we have $a<b<c$, the inclusion of $\CF^{(-\infty,b]}_*(H,J)$ into $\CF^{(-\infty,c]}_*(H,J)$ descends the quotient to an injective map $\CF^{(a,b]}_*(H,J)\to\CF^{(a,c]}_*(H,J)$. The cokernel of this map is $\CF^{(a,c]}_*(H,J)/\CF^{(a,b]}_*(H,J)$ which by the third isomorphism theorem is just $\CF^{(b,c]}_*(H,J)$. This leads to the short exact sequence of chain complexes
\begin{equation}
	0\to\CF^{(a,b]}_*(H,J)\to\CF^{(a,c]}_*(H,J)\to\CF^{(b,c]}_*(H,J)\to 0.
\end{equation}
The resulting long exact sequence in homology
\begin{equation}\label{eq:LES_in_filtered_HF}
	\dots\to\HF^{(a,b]}_*(H)\to\HF^{(a,c]}_*(H)\to\HF^{(b,c]}_*(H)\to\HF^{(a,b]}_{*-1}(H)\to\dots
\end{equation}
is called ``$(a,b,c)$ exact sequence'' in \cite[Section 3.2.1]{Ginzburg2007_Coisotropic}. With a slight abuse of language, we will call the first arrow $i:\HF^{(a,b]}_*(H)\to\HF^{(a,c]}_*(H)$ the \emph{inclusion} morphism, and the middle arrow $q:\HF^{(a,c]}_*(H)\to\HF^{(b,c]}_*(H)$ the \emph{quotient} morphism.

As an important particular case, consider $b>a$ and $C>0$. Consider first the long exact sequence \eqref{eq:LES_in_filtered_HF} with $a<b<b+C$, and then with $a<a+C<b+C$. We obtain
\begin{equation}\label{eq:def_of_iq_morphism-diagram}
	\begin{tikzcd}
		\cdots \ar[r]& \HF^{(a,b]}_*\left(H\right)\ar[r,"i"]&\HF^{(a,b+C]}_*\left(H\right)\ar[r,"q"]\ar[d,equal]&\HF^{(b,b+C]}_*\left(H\right) \ar[r]&\cdots\\
		\cdots \ar[r] & \HF^{(a,a+C]}_*\left(H\right)\ar[r,"i"]&\HF^{(a,b+C]}_*\left(H\right)\ar[r,"q"]&\HF^{(a+C,b+C]}_*\left(H\right) \ar[r]&\cdots
	\end{tikzcd}
\end{equation}
The composition of the upper inclusion with the lower quotient is called \emph{inclusion-quotient morphism}:
\begin{equation}\label{eq:def_of_iq_morphism}
	\Phi^{(a,b]}_{H}(C)\colon\HF^{(a,b]}_*(H)\to\HF^{(a+C,b+C]}_*(H)
\end{equation}
We usually omit the Hamiltonian $H$, the window $(a,b]$ and the shift $C$, denoting this morphism directly by $\Phi$, since this information is contained in the domain and range of the map. Also, we denote by $(a,b]+C$ the interval $(a+C,b+C]$.

\subsubsection{Degenerate Hamiltonians}
\label{ssub:degenerate_hamiltonians}

Here we explain how to drop the assumption of non-degeneracy of the Hamiltonians. This is possible because filtered Floer homology can be shown to be, in some sense, locally constant in the Hamiltonian. Recall the standing assumptions: ``Hamiltonian'' means asymptotically quadratic, non-resonant at infinity Hamiltonian with bounded sub-quadratic part. We can use much weaker hypotheses with further energy estimates \cite{MyPhDThesis}.

Given a Hamiltonian $H$, not necessarily non-degenerate, denote by $\bS(H)$ the set of critical values of its Hamiltonian action functional $\cA_H$, i.e. the set of actions of its 1-periodic orbits. Similarly as in \cite[Chapter 5, Proposition 8]{HoferZehnder2011_SymplecticInvariants} it is easy to show that $\bS(H)$ is a compact and nowhere-dense subset of $\bR$. In particular, when $H$ is non-degenerate, $\bS(H)$ is a finite subset of $\bR$.

Following \cite[\S 4.4]{BiranPolterovichSalamon2003_Propagation}, one may show that for any Hamiltonian $F$, and for all $a,b\notin\bS(F)$, there exists a compact set and a neighborhood $\mathfrak U_F$ of $F$ in the space of Hamiltonians which equal $F$ outside this compact set, with the following property: if $G\in\mathfrak U_F$ is non-degenerate, and $a,b\notin\bS(G)$, then there exists a continuation \emph{isomorphism}
\begin{equation}\label{eq:filteredHF_loc_cst}
	\HF^{(a,b]}_*(F)\cong\HF^{(a,b]}_*(G).
\end{equation}

This fact implies that we can define the filtered Floer homology of a \emph{degenerate} Hamiltonian $H$, essentially by perturbation. Indeed, non-degenerate Hamiltonians are dense in the $C^\infty$-strong topology on the space of Hamiltonians \cite{FloerHofer94_SymplecticHomology-I}. Fix $a,b\notin\bS(H)$ and take $\tilde H$ a non-degenerate Hamiltonian which is close enough to $H$ so that $a,b\notin\bS(\tilde H)$. This is possible because $\bS(H)$ is compact and nowhere dense. Up to choosing $\tilde H$ even closer to $H$, we have that $H\in\mathfrak U_{\tilde H}$. Setting
\begin{equation}\label{eq:def_of_deg_HF_by_pert}
	\HF^{(a,b]}_*(H)=\HF^{(a,b]}_*\left(\tilde H\right)
\end{equation}
and using \eqref{eq:filteredHF_loc_cst} we see that the definition is independent of the $\tilde H$ chosen.

We can talk about inclusion, quotient and inclusion-quotient morphisms on $\HF^{(a,b]}_*(H)$ with the understanding that we actually are defining them on the filtered Floer homology of a non-degenerate Hamiltonian close enough to $H$. Therefore, they are defined as long as \emph{the end-points of the action windows in the filtered homologies are not critical values of the action functionals in question}.

\subsection{Action shift of continuation morphisms}
\label{sub:action_shift_of_continuation_morphisms}

We now explain the effect of continuations on the filtered Floer homology, which is without a doubt one of the most crucial elements in the proof of the Poincaré-Birkhoff theorem. The results in this section follow the treatment of \cite[Section 3.2.2]{Ginzburg2007_Coisotropic}. We continue to assume that all Hamiltonians are asymptotically quadratic, non-resonant at infinity, have the same quadratic Hamiltonian at infinity and bounded sub-quadratic part. As always, we can use much weaker hypotheses with further energy estimates \cite{MyPhDThesis}.

First, assume that $H^0$ and $H^1$ are non-degenerate. Let $\cH$ be the continuation Hamiltonian defined in \eqref{eq:same_Q_at_infty_continuation}. Recall the definition of the continuation morphism \eqref{eq:def_of_contin_morphism_chain}, in particular, the number $\mathcal N_{\cH,\cJ}(\gamma^0,\gamma^1)$ which gives the parity of the zero-dimensional moduli space of continuation Floer trajectories starting at $\gamma^0$ and ending at $\gamma^1$. A continuation Floer trajectory may contribute to this count only if it has positive energy. From the energy estimate \eqref{eq:energy_calc_contin} and the choice of continuation Hamiltonian \eqref{eq:same_Q_at_infty_continuation}, we see
\begin{equation}\label{eq:energy_nonneg_action_shift}
	\begin{split}
		0<E_{\cH,\cJ}(u)\leq\cA_{H^0}\left(\gamma^0\right)-\cA_{H^1}\left(\gamma^1\right)+\left\|H^1-H^0\right\|_{L^\infty}\\
		\implies \cA_{H^1}\left(\gamma^1\right)\leq \cA_{H^0}\left(\gamma^0\right)+\left\|H^1-H^0\right\|_{L^\infty}.
	\end{split}
\end{equation}
Define the \emph{action shift constant}
\begin{equation}\label{eq:def_of_action_shift_constant}
	C=\left\|H^1-H^0\right\|_{L^\infty}.
\end{equation}
The continuation morphism $\cC(\cH,\cJ)\colon\CF_*(H^0,J^0)\to\CF_*(H^1,J^1)$ is thus a filtered chain complex morphism only up to a shift in the filtration:
\begin{equation}
	\cC(\cH,\cJ)\colon\CF_*^{(-\infty,a]}(H^0,J^0)\to\CF_*^{(-\infty,a+C]}(H^1,J^1),
\end{equation}
and it descends the quotient to a morphism
\begin{equation}
	\cC\left(\cH,\cJ\right)\colon\CF^{(a,b]}_*\left(H^0,J^0\right)\to\CF^{(a+C,b+C]}_*\left(H^1,J^1\right).
\end{equation}
We denote the induced morphism on the filtered homologies by
\begin{equation}\label{eq:contin_with_action_shift}
	\cC\colon\HF^{(a,b]}_*\left(H^0\right)\to\HF^{(a+C,b+C]}_*\left(H^1\right).
\end{equation}
Assuming that $a,b\notin\bS(H^0)$ and $a+C,b+C\notin\bS(H^1)$, we can allow ourselves to work with \emph{degenerate} Hamiltonians, as always with the understanding that the Floer homology is really of an arbitrarily close non-degenerate Hamiltonian.

The naturality of the long exact sequence in homology implies that the long exact sequence \eqref{eq:LES_in_filtered_HF} is functorial with respect to continuations, namely the following diagram commutes:
\begin{equation}\label{eq:LES_in_filtered_HF_functorial}
	\begin{tikzcd}[cramped]
		\cdots \ar[r]& \HF^{(a,b]}_*\left(H^0\right)\ar[r,"i"]\ar[d,"\cC"]&\HF^{(a,c]}_*\left(H^0\right)\ar[r,"q"]\ar[d,"\cC"]&\HF^{(b,c]}_*\left(H^0\right) \ar[d,"\cC"]\ar[r]&\cdots\\
		\cdots \ar[r] & \HF^{(a,b]+C}_*\left(H^1\right)\ar[r,"i"]&\HF^{(a,c]+C}_*\left(H^1\right)\ar[r,"q"]&\HF^{(b,c]+C}_*\left(H^1\right) \ar[r]&\cdots
	\end{tikzcd}
\end{equation}
This fact implies the following lemma:
\begin{lemma}\label{lem:iq_commutes_with_continuation}
	Let $H^0,H^1$ be Hamiltonians and $\mathcal{H}$ a continuation between them. Let $C>0$ be the action shift constant of $\cC=\cC(\cH,\cJ)$. If $D>0$ and $a,b\in\bR$ are such that $a,b,a+D,b+D\notin\bS(H^0)$ and $a+C,b+C,a+C+D,b+C+D\notin\bS(H^1)$ then the following diagram
	\begin{equation}\label{eq:iq_commutes_with_continuation}
		\begin{tikzcd}
			\HF^{(a,b]}_*\left(H^0\right) \ar[r,"\cC"]\ar[d,"\Phi"]& \HF^{(a,b]+C}_*\left(H^1\right)\ar[d,"\Phi"]\\
			\HF^{(a,b]+D}_*\left(H^0\right) \ar[r,"\cC"]& \HF^{(a,b]+C+D}_*\left(H^1\right)\\
		\end{tikzcd}
	\end{equation}
	commutes.
\end{lemma}
\begin{proof}
	We can prove the claim for $H^0$ and $H^1$ non-degenerate, since the action shift of a continuation between non-degenerate Hamiltonians arbitrarily close to $H^0$ and $H^1$ will be arbitrarily close to $C$. We combine the definition \eqref{eq:def_of_iq_morphism-diagram} for $H^0$ and $H^1$, and functoriality \eqref{eq:LES_in_filtered_HF_functorial} of the four corresponding long exact sequences to obtain the following commutative diagram:
	\begin{equation}\label{eq:iq_commutes_with_contin}
		\begin{tikzcd}[row sep=small,column sep=tiny, cramped]
			\HF^{(a,b]}_*\left(H^0\right)\ar[dd,"\cC"]\ar[dr,"i"] & & \HF^{(a,b+D]}_*\left(H^0\right)\ar[dd,"\cC"]\ar[dr,"q"] &\\
			& \HF^{(a,b+D]}_*\left(H^0\right)\ar[dd,"\cC"]\ar[ur,equals] & & \HF^{(a+D,b+D]}_*\left(H^0\right)\ar[dd,"\cC"]\\
			\HF^{(a,b]+C}_*\left(H^1\right)\ar[dr,"i"] & & \HF^{(a,b+D]+C}_*\left(H^1\right)\ar[dr,"q"] &\\
			& \HF^{(a,b+D]+C}_*\left(H^1\right)\ar[ur,equals] & & \HF^{(a,b]+C+D}_*\left(H^1\right)\\
		\end{tikzcd}
	\end{equation}
	The claim is the commutation of the outermost square.
\end{proof}
The next lemma is very simple but worth spelling out, because it elucidates the structure of the inclusion-quotient morphism. It can be found in \cite[Example 3.3]{Ginzburg2007_Coisotropic}. The constant homotopy $\mathcal{H}^s\equiv H$ induces a continuation morphism $\HF^{(a,b]}_*\left(H^0\right)\to\HF^{(a+C,b+C]}_*\left(H^1\right)$ for any $C\geq0$ such that $a,b,a+C,b+C\notin\bS(H)$. But as we observed in section \ref{ssub:definition_of_the_morphism_and_functoriality} (see \cite[Proposition 11.1.14]{AudinDamian2014_MorseFloer}), on the Floer chain complex of $H$ this continuation is just the identity morphism. Therefore by \eqref{eq:iq_commutes_with_continuation}
\begin{lemma}\label{lem:iq_maps_classes_to_themselves}
	Assume that $a<b\in\bR$, $C>0$ are such that $a,b,a+C,b+C\notin\bS(H)$. The inclusion-quotient morphism $\Phi=\Phi^{(a,b]}_{H}(C)\colon\HF^{(a,b]}_*(H)\to\HF^{(a,b]+C}_*(H)$ is characterized by the following property. If $\alpha\in\HF^{(a,b]}_*(H)$ is such that $\cA_H(\alpha)\in(a,b]\cap(a+C,b+C]$, then $\Phi(\alpha)=\alpha$, otherwise $\Phi(\alpha)=0$.
\end{lemma}
We used a corollary of this lemma at the crucial diagram \eqref{eq:crucial_triangle}, as a kind of functoriality ``up to a shift'' in the specific case of continuation back and forth between two Hamiltonians. Recall from section \ref{ssub:definition_of_the_morphism_and_functoriality} that if $\cH$ is a continuation Hamiltonian, we can ``reverse'' it by setting $\overline{\cH}^s=\cH^{1-s}$. If $\cH$ defines a continuation $\cC\colon\CF_*\left(H^0,J^0\right)\to\CF_*\left(H^1,J^1\right)$, then $\overline\cH$ defines a homotopy inverse $\overline\cC\colon\CF_*\left(H^1,J^1\right)\to\CF_*\left(H^0,J^0\right)$. On the level of filtered homology, assume that $\cH$ defines a continuation morphism which shifts the action filtration by $C\geq 0$. Then the same is true for $\overline\cH$. Thus the concatenation of $\cH$ and $\overline\cH$ defines a continuation morphism which shifts the action filtration by $2C$, while sending every generator of the Floer homology of $H^0$ to itself. On the filtered Floer homology of $H^0$ this is precisely the inclusion-quotient morphism with shift $2C$, as we have just seen in the previous lemma. We thus have proven:
\begin{lemma}\label{lem:contin_forth_and_back_factors_IQ}
	Let $a<b\in\bR$ be such that $a,b,a+2C,b+2C\notin\bS(H^0)$ and $a+C,b+C\notin\bS(H^1)$. Then $\overline\cC\circ\cC=\Phi_{H,(a,b]}(2C)$ on filtered homology:
	\begin{equation}\label{eq:crucial_triangle_abstract}
		\begin{tikzcd}
			& \HF^{(a,b]+C}_*\left(H^1\right)\ar[dr,bend left,"\overline\cC"]&\\
			\HF^{(a,b]}_*\left(H^0\right)\ar[rr,"\Phi"]\ar[ur,bend left,"\cC"]& &\HF^{(a,b]+2C}_*\left(H^0\right)
		\end{tikzcd}
	\end{equation}
\end{lemma}

\subsection{Local Floer homology}
\label{sub:local_floer_homology}

When we talk about the Floer homology of a \emph{degenerate} Hamiltonian $H$, we can't actually think of the 1-periodic orbits of the original Hamiltonian $H$ as generators, unless they are non-degenerate orbits. Intuitively, we can think of the non-degenerate infinitesimal bifurcations of the degenerate orbits as the generators. It is then natural to investigate what kinds of ``infinitesimal relations'' given by very low energy Floer trajectories appear between these bifurcating orbits. Local Floer homology contains this homological information when the perturbation process is localized at a single isolated 1-periodic orbit. If $H$ has 1-periodic orbits which are not isolated as critical points of the action functional, perturbing $H$ to a non-degenerate Hamiltonian will of course destroy a lot of its 1-periodic orbits. Since we are interested in statements which \emph{imply} an infinitude of periodic orbits, it is harmless to ignore the non-isolated situation.

For concreteness assume again that ``Hamiltonian'' means asymptotically quadratic, non-resonant at infinity Hamiltonian with bounded sub-quadratic part. But in this section the behaviour of the Hamiltonian at infinity is actually irrelevant.

\subsubsection{Sketch of definition of local Floer homology}
\label{ssub:sketch_of_definition_of_local_floer_homology}

Let $H$ be a Hamiltonian. Let $z_0\in\Fix\phi^1_H$ be an isolated fixed point, and $\gamma_0$ the corresponding $1$-periodic orbit of $X_H$. Let $\cU_0\subset S^1\times\bR^{2n}$ be an open neighborhood such that the graph of the orbit $\gamma_0$ is the only graph of a 1-periodic orbit of $H$ contained in the closure of $\cU_0$. We call this an \emph{isolating neighborhood} of the 1-periodic orbit $\gamma_0$.

Consider the set of Hamiltonians $G$ which are equal to $H$ outside $\cU_0$, and inside $\cU_0$ are non-degenerate. This is a residual set in the space of all Hamiltonians equal to $H$ outside $\cU_0$, as can be seen from \cite{FloerHoferSalamon1995_Transv}. Since $\cU_0$ is an isolating neighborhood of $\gamma_0$, if $G$ is $C^\infty$-close enough to $H$ then every 1-periodic orbit of $G$ which intersects $\cU_0$ is completely contained in it \cite[Lemma 2.1, point 1]{CFHW_SHApplsII}. We call these orbits of $G$ in $\cU_0$ the \emph{orbits bifurcating from $\gamma_0$}.

For $J$ generically chosen, one may define a Floer chain complex $\CF^\loc_*\left(G,\cU_0,J\right)$ which is generated by the 1-periodic orbits of $G$ bifurcating from $\gamma_0$, and whose differential counts the Floer cylinders which connect these orbits. If $G$ is a suitably small perturbation of $H$, one shows that such Floer cylinders cannot leave the neighborhood $\cU_0$ \cite[Lemma 2.1, point 2]{CFHW_SHApplsII}. As a consequence, the complex is well defined.

Let $\cU_0$ and $\cU_1$ be isolating neighborhoods for $\gamma_0$ and $G_0,G_1$ two perturbations of $H$ such that $G_i-H$ is supported in $\cU_i$ for $i=0,1$. We can define continuation Hamiltonians between them, which are constantly equal to $H$ outside $\cU_0\cup\cU_1$, and thus local continuation morphisms between the corresponding local chain complexes of $G_0$ and $G_1$. If $G_0$ and $G_1$ are $C^\infty$-close enough, it is easy to show that the induced morphism on homology is an isomorphism. Hence, the homology of $\CF^\loc_*\left(G,\cU_0,J\right)$ is independent of $G$ and $\cU_0$. We call it the \emph{local} Floer homology of $\gamma_0$ and we denote it by $\HFloc_*\left(H, z_0\right)$. Since we carry the generating Hamiltonian in the notation, it is sufficient to refer to the fixed point only.

\subsubsection{Properties of local Floer homology}
\label{ssub:properties_of_local_floer_homology}

First, we study how the local Floer homology depends on the Hamiltonian.
We state an important invariance property, which can be found in \cite[Section 3.2,(LF1)]{Ginzburg2010_Conley}.
\begin{lemma}\label{lem:HFloc_invariant_under_unif_isolating_homotopies}
 Assume there exists a homotopy $F=F^s_t$, $s\in[0,1]$, with $F^0=H^0$, $F^1=H^1$, $z_0\in\Fix\phi^1_{F^s}$ for all $s\in[0,1]$ and there exists a fixed neighborhood $\mathcal{U}_0$ which is isolating for $\gamma_s(t)=\phi^t_{F^s}(z_0)$ for every $s\in[0,1]$. Then
	\begin{equation}\label{eq:HFloc_invariant_under_unif_isolating_homotopies}
		\HFloc_*\left(H^0,z_0\right)\cong\HFloc_*\left(H^1,z_0\right).
	\end{equation}
\end{lemma}
\begin{proof}[Sketch of proof]
	As explained right below the statement of \cite[Section 3.2,(LF1)]{Ginzburg2010_Conley}, this follows from a continuation argument.

	We can choose an arbitrarily small perturbation $\mathcal{G}=\mathcal{G}^s_t$ of $F$ and a family $\mathcal{J}=\mathcal{J}^s_t$ of almost-complex structures with the following list of properties. $\mathcal{J}$ is constant outside $\mathcal{U}_0$. $\mathcal{G}$ is equal to $F$ outside $\mathcal{U}_0$ for all $s\in[0,1]$. The end-points of $\mathcal{G}$ are non-degenerate Hamiltonians $G^0$ and $G^1$, which can be chosen as close as we wish to respectively $H^0$ and $H^1$, and are equal to them outside $\mathcal{U}_0$. The local Floer chain complexes $\CF^\loc_*\left(G^0,\mathcal{U}_0,\mathcal{J}^0\right)$ and $\CF^\loc_*\left(G^1,\mathcal{U}_0,\mathcal{J}^1\right)$ are well defined. Finally, $\left(\mathcal{G},\mathcal{J}\right)$ define continuation Floer trajectories connecting the generators of the aforementioned local Floer chain complexes which are transversally cut out.

	Since $\mathcal{U}_0$ isolates $\gamma_s$ for all $s\in[0,1]$, if $\mathcal{G}$ is a small enough perturbation, the continuation Floer trajectories that it defines can be shown to be contained in $\mathcal{U}_0$. Moreover, the moduli spaces of continuation Floer trajectories between generators of the local Floer chain complexes of $G_0$ and $G_1$ can be shown to be compact, up to breaking on Floer trajectories inside $\mathcal{U}_0$. Therefore, the continuation map defined by $\left(\mathcal{G},\mathcal{J}\right)$ between the local Floer chain complexes of $G_0$ and $G_1$ is well defined. The standard arguments on continuations then show that the induced morphism on homology is an isomorphism. Independence of local Floer homology on the data required to define it gives us the conclusion.
\end{proof}
Next, following \cite[Section 3.2, (LF3) and (LF4)]{Ginzburg2010_Conley}, we study the action of loops in $\Ham$ on local Floer homology. We are interested only in contractible loops and loops of linear symplectomorphisms.
\begin{lemma}\label{lem:HFloc_action_of_linear_loops_shifts_grading}
	Let $P$ be a quadratic Hamiltonian generating a loop of linear symplectomorphisms with Maslov index $\mu\in\bZ$. Set $\sigma=2\mu$.
	\begin{equation}\label{eq:HFloc_action_of_linear_loops_shifts_grading}
		\HFloc_*\left(P\#H,z_0\right)\cong\HFloc_{*-\sigma}\left(H,z_0\right).
	\end{equation}
\end{lemma}
\begin{proof}
	Choose $(G,\mathcal{U}_0,J)$ giving rise to the local Floer chain complex $\CF^\loc_*\left(G,\mathcal{U}_0,J\right)$. Set $G'=P\#G$ and $J'_t=\left(\phi^t_P\right)^*J_t$. This is also a regular pair giving rise to the local Floer chain complex $\CF^\loc_{*}\left(G',\mathcal{U}_0,J'\right)$. If $\gamma'$ is a 1-periodic orbit of $G'$ in $\mathcal{U}_0$, then there is a unique 1-periodic orbit $\gamma$ of $G$ in $\mathcal{U}_0$ such that $\gamma'(t)=\phi^t_P(\gamma(t))$. Moreover $\CZ\left(\gamma'\right)=\CZ\left(\gamma\right)+\sigma$. The map defined by composition with $\phi^t_P$ thus induces an isomorphism up to a between the graded $\bZ/2$-vector spaces underlying $\CF^\loc_*\left(G,\mathcal{U}_0,J\right)$ and $\CF^\loc_{*-\sigma}\left(G',\mathcal{U}_0,J'\right)$. What remains to show is that this is a chain map. Consider a $(G,J)$-Floer trajectory $u$. Set $v(s,t)=\phi^t_P\left(u(s,t)\right)$. A simple calculation shows that $v$ solves the $(G',J')$-Floer equation, and in fact every Floer trajectory for $(G',J')$ is of this form. Therefore the composition with $\phi^t_P$ is a chain map inducing a chain isomorphism $\CF^\loc_*\left(G,\mathcal{U}_0,J\right)\cong\CF^\loc_{*-\sigma}\left(G',\mathcal{U}_0,J'\right)$.
\end{proof}
Notice that assuming that $P$ was quadratic was not essential for the proof, which in fact holds also for a general Hamiltonian $P$ generating a loop in $\Ham$ and $\mu$ the Maslov index of the loop in $\Sp(2n)$ obtained by linearizing its flow at any given point. The following lemma can be established in an analogous manner.
\begin{lemma}\label{lem:HFloc_invariant_contractible_loop_compos}
	Assume $H^0$ and $H^1$ are such that $\phi^1_{H^0}=\phi^1_{H^1}=\upphi$ and their flows are related by a contractible loop in $\Ham$. If $z_0\in\Fix\upphi$ is an isolated fixed point, then
	\begin{equation}\label{eq:HFloc_invariant_contractible_loop_compos}
		\HFloc_*\left(H^0,z_0\right)\cong\HFloc_*\left(H^1,z_0\right).
	\end{equation}
\end{lemma}

Next, we want to single out in which degrees the local Floer homology can be non-trivial. Let $z_0\in\Fix\phi^1_H$ be an isolated fixed point. Define the \emph{degree support} of the local homology
	\begin{equation}\label{eq:def_of_degsupp}
		\deg\supp\HFloc_*(H,z_0)=\left\{k\in\bZ:\HFloc_k(H,z_0)\neq\{0\}\right\}
	\end{equation}
There is an easy bound on the degree support.
\begin{lemma}\label{lem:degsupp_of_HFloc}
	It holds that
	\begin{equation}\label{eq:degsupp_estimate}
		\deg\supp\HFloc_*(H,z_0)\subseteq\left[\meanCZ\left(z_0,H\right)-n,\meanCZ\left(z_0,H\right)+n\right].
	\end{equation}
\end{lemma}
This lemma is an immediate consequence of \cite[Section 1.3.7, equation 1.11]{Abbondandolo2001_MorseHamilt}. Since we are taking the lower-semicontinuous extension of the Conley-Zehnder index when dealing with the possibly degenerate orbit $\gamma_0$, the Conley-Zehnder index of an orbit bifurcating from $\gamma_0$ can be at distance at most $2n$ from $\meanCZ(z_0,H)$.

Now we turn to the relation between local Floer homology and filtered Floer homology of a degenerate Hamiltonian. We will find that local Floer homology classes provide the ``building blocks'' of filtered Floer homology in certain simple situations. Recall that the set of critical values of the Hamiltonian action functional $\cA_H$ is denoted by $\bS(H)$.
\begin{prop}\label{prop:local_to_nbh_HF-general}
	Assume that $z_0$ is an isolated fixed point of $\phi^1_H$ such that $\cA_H(z_0)=a_0$ is an isolated critical value. If $a<b$ are such that $[a,b]\cap\bS(H)=\left\{a_0\right\}$, then there exists an \emph{injective} morphism
	\begin{equation}\label{eq:local_to_global_HF-general}
		\HFloc_*\left(H,z_0\right)\hookrightarrow\HF^{(a,b]}_*\left(H\right).
	\end{equation}
\end{prop}
\begin{proof}
	Recall that when $a,b\notin\bS(H)$, the Floer homology of $H$ in action window $(a,b]$ is by definition the Floer homology in the same action window of a non-degenerate $\tilde H$ which is sufficiently $C^\infty$-close to $H$.
	
	Fix once and for all a sufficiently small isolating neighborhood $\cU_0$ of the 1-periodic orbit $\gamma_0(t)=\phi^t_H(z_0)$.
	Let $G$ be a Hamiltonian which equals $H$ outside $\cU_0$ and is non-degenerate in $\cU_0$, so close to $H$ such that the local Floer chain complex $\CF^\loc_*\left(G,\cU_0,J\right)$ is well defined. Since $G$ is non-degenerate in $\cU_0$, there always exists a non-degenerate Hamiltonian $\tilde H$ which coincides with $G$ in $\cU_0$. Up to taking $G$ even closer to $H$, we can choose $\tilde H$ close enough to $H$ so that its Floer homology defines the Floer homology of $H$. We can pick the almost-complex structure $J$ generically such that the Floer chain complex of $\tilde H$ and the local Floer chain complex defined by $G$ are both well defined.

	Let $I=(a,b]$ and $I'=(a',b']$ be intervals whose end-points are not critical values of $\cA_{\tilde H}$. Assume that
	\begin{equation}
		I\cap\bS(\tilde H)=I'\cap\bS(\tilde H).
	\end{equation}
	Then there exists an isomorphism $\HF_*^I(\tilde H)\cong\HF_*^{I'}(\tilde H)$ given by appropriate inclusion-quotient morphisms.
	Hence it suffices to show that for every $\epsilon>0$ small enough, and for every $G$ and $\tilde H$ chosen as above, the obvious inclusion of chain complexes
	\begin{equation}
		\CF^\loc_*\left(G,\cU_0,J\right)\subset\CF_*^{\left(a_0-\epsilon,a_0+\epsilon\right]}\left(\tilde H, J\right)
	\end{equation}
	remains an injection at the level of homologies.

	We argue by contradiction. Assume there exist a sequence of Hamiltonians $\left(\tilde H^{(k)}\right)_{k\in\bN}$ which are non-degenerate, asymptotically quadratic and non-resonant at infinity, and a sequence $\left(G^{(k)}\right)_{k\in\bN}$ of Hamiltonians equal to $H$ outside $\cU_0$ and non-degenerate in $\cU_0$. Assume that $G^{(k)}|_{\cU_0}=\tilde H^{(k)}|_{\cU_0}$, and that $\tilde H^{(k)}\to H$ and $G^{(k)}\to H$ in $C^\infty$ as $k\to\infty$. Let $\left(J^{(k)}\right)_{k\in\bN}$ be a sequence of adequate almost complex structures, with $J^{(k)}\to J$ in $C^\infty$ and chosen generically in order to define the Floer chain complexes in study. Suppose there exists a sequence $\epsilon_k\to 0$, and sequences of chains
	\begin{equation}
		\left(c_k\in\CF^\loc_*\left(G^{(k)},\cU_0,J^{(k)}\right)\right)_{k\in\bN},\quad \left(\tilde c_k\in\CF_{*+1}^{(a-\epsilon_k,a+\epsilon_k]}\left(\tilde H^{(k)},J^{(k)}\right)\right)_{k\in\bN}
	\end{equation}
	which represent non-zero elements in the corresponding Floer homologies, but for which
	\begin{equation}\label{eq:HFloc_inj_into_small_window-homicide}
		c_k=d_{\tilde H^{(k)},J^{(k)}}\tilde c_k\in\operatorname{Im} d_{\tilde H^{(k)},J^{(k)}}
	\end{equation}
	i.e. each $c_k$ becomes zero in the filtered homology $\HF^{(a-\epsilon_k,a+\epsilon_k]}\left(\tilde H^{(k)},J^{(k)}\right)$.
	Recall the definition of the Floer differential \eqref{eq:def_of_Floer_differential}. By \eqref{eq:HFloc_inj_into_small_window-homicide}, there are sequences of 1-periodic orbits $\left(\xi_k\right)_{k\in\bN}$ and $\left(\chi_k\right)_{k\in\bN}$ of $\tilde H^{(k)}$, with $\xi_k$ appearing as a summand in $c_k$ and $\chi_k$ appearing as a summand in $\tilde c_k$, such that
	\begin{equation}\label{eq:HFloc_inj_into_small_window-action_difference_estimate}
		a-\epsilon_k<\cA_{H^{(k)}}(\chi_k)\leq a+\epsilon_k,\quad a-\epsilon_k<\cA_{H^{(k)}}(\xi_k)\leq a+\epsilon_k\quad\forall k
	\end{equation}
	and
	\begin{equation}\label{eq:HFloc_inj_into_small_window-cylcount_nonzero}
		N_{\tilde H^{(k)},J^{(k)}}\left(\xi_k,\chi_k\right)\neq0\quad\forall k.
	\end{equation}
	For the sake of the argument, we can assume that each $\chi_k$ in the sequence $\left(\chi_k\right)$ does not lie completely in $\cU_0$. Indeed, if for $k$ large enough every orbit $\chi_k$ which appears as a summand of $\tilde c_k$ lies completely in $\cU_0$, then we can see it as an orbit of $G^{(k)}$, hence $\tilde c_k$ is a chain in $\CF^\loc_{*+1}\left(G^{(k)},\cU_0,J^{(k)}\right)$. This implies that for $k$ large enough, $c_k=d_{G^{(k)},J^{(k)}}\tilde c_k$, so $c_k$ is zero in homology. But this is a contradiction, because we assumed that $c_k$ represented a non-zero class in the local homology. This shows there are infinitely many $k$ such that $\chi_k$ does not lie completely in $\cU_0$.

	The isolation of $\gamma_0$ and the convergence of $\tilde H^{(k)}$ to $H$ imply that for $k$ large enough, $\chi_k$ is completely disjoint from $\cU_0$. We thus can assume that every orbit in the sequence $\chi_k$ lies entirely outside $\cU_0$.

	Since \eqref{eq:HFloc_inj_into_small_window-cylcount_nonzero} holds, there is a sequence of $(\tilde H^{(k)},J^{(k)})$-Floer trajectories $v_k$, such that $v_k(s,\cdot)\to\chi_k$ as $s\to-\infty$, $v_k(s,\cdot)\to\xi_k$ as $s\to+\infty$ for all $k$. Since $\chi_k$ is non-degenerate and lies outside $\cU_0$, there exists a sequence $\left(s_k\in\bR\right)_{k\in\bN}$ such that for all $k$, $v_k(s,\cdot)$ lies outside $\cU_0$ for all $s>s_k$. Define $u_k(s,\cdot)=v_k(s+s_k,\cdot)$. By construction, $u_k\left(0,\cdot\right)$ lies in $\overline\cU_0$. From \eqref{eq:HFloc_inj_into_small_window-action_difference_estimate} and the energy calculation \eqref{eq:auton_FlCyl_energy} we see that
	\begin{equation}\label{eq:HFloc_inj_into_small_window-energy_vanishes}
		E_{H^{(k)},J^{(k)}}\left(u_k\right)\to 0
	\end{equation}
	This implies that we can use proposition \ref{prop:unif_Linfty_bounds} and a routine bubbling argument (see e.g. \cite[Proposition 6.6.2]{AudinDamian2014_MorseFloer}) to show that our sequence $\left(u_k\right)_k$ has a $C^0_\loc$ converging subsequence. Elliptic bootstrapping \cite[Section B.4]{McDuffSalamon2012_BigJ}, appropriately modified for the Floer equation, gives us a further subsequence which converges in $C^\infty_\loc$. In particular, there exists an $(H,J)$-Floer trajectory $u$ and a subsequence of the $u_k$ converging to $u$ in $C^\infty_\loc$. Since $u_k(0,\cdot)\subset\overline\cU_0$ for all $k$, it holds that $u(0,\cdot)\subset\overline\cU_0$. From \eqref{eq:HFloc_inj_into_small_window-energy_vanishes} it follows that $E_{H,J}(u)=0$. Hence $u$ is constant in the $s$-direction, and therefore is a 1-periodic orbit $\gamma(t)=u(0,t)$ of $H$ contained in $\overline\cU_0$. Since $u_k(s,\cdot)\to\chi_k$ as $s\to-\infty$, it must be that $\gamma\neq\gamma_0$. But we assumed that $\cU_0$ was an isolating neighborhood for $\gamma_0$, so no 1-periodic orbit of $H$ other than $\gamma_0$ can be found in its closure. This contradiction proves the proposition.
\end{proof}
The proof shows that we can think of the classes in $\HFloc_*\left(H,z_0\right)$ as classes in $\HF^{(a,b]}_*(H)$ when $[a,b]\cap\bS(H)=\left\{\cA_H\left(z_0\right)\right\}$ by choosing perturbations which coincide on a small isolating neighborhood of $z_0$. In particular, using lemma \ref{lem:iq_maps_classes_to_themselves} on the inclusion-quotient morphism, we conclude that the classes corresponding to the local Floer homology of $z_0$ are mapped to themselves by the inclusion-quotient morphism:
\begin{lemma}\label{lem:iq_commutes_with_HFloc_inclusion}
	Let $z_0\in\Fix\phi^1_H$ be an isolated fixed point such that its critical value $\cA_H(z_0)=a_0\in\bR$ is isolated. If $a<b$ and $C>0$ are such that $[a,b]\cap\bS(H)=[a+C,b+C]\cap\bS(H)=\{a_0\}$, then the following diagram commutes:
	\begin{equation}\label{eq:iq_commutes_with_HFloc_inclusion}
		\begin{tikzcd}
			\HFloc_*\left(H,z_0\right) \ar[r,hook]\ar[d,equals] & \HF^{I}_*\left(H\right)\ar[d,"\Phi(C)"]\\
			\HFloc_*\left(H,z_0\right)\ar[r,hook] &\HF^{I+C}_*\left(H\right)
		\end{tikzcd}
	\end{equation}
	where the horizontal arrows are the injections of proposition \ref{prop:local_to_nbh_HF-general} and the vertical right arrow is the inclusion-quotient morphism \eqref{eq:def_of_iq_morphism}.
\end{lemma}
In the proof of the main theorem we use the following corollary of proposition \ref{prop:local_to_nbh_HF-general}, which can be also found in \cite[Section 3.2, (LF2)]{Ginzburg2010_Conley}.
\begin{lemma}\label{lem:local_to_nbh_HF}
	Assume that $\Fix\phi^1_H$ is a finite set. If $a_0\in\bS(H)$ and $a<b$ are such that $[a,b]\cap\bS(H)=\{a_0\}$, then
	\begin{equation}\label{eq:local_to_global_HF-nbh}
		\HF^{(a,b]}_*\left(H\right)\cong\bigoplus\left\{\HF^\loc\left(H,z'\right):z'\in\Fix\phi^1_H,\ \cA_H\left(z'\right)=a_0\right\}.
	\end{equation}
\end{lemma}

Finally, it is important to know how local Floer homology behaves under iteration. We state a corollary of \cite[Theorem 1.1]{GinzburgGuerel2010_LFHAG} by Ginzburg and Gürel, which suffices for the purposes of this paper:
\begin{lemma}\label{lem:HFloc_iterated}
	Assume that $\HFloc_*(H,z_0)\neq\{0\}$. Then $\HFloc_*\left(H^{\times k},z_0\right)\neq\{0\}$ as long as $k$ is not a multiple of the order of any non-trivial root of unity in the spectrum of $d\phi^1_H(z)$.
\end{lemma}

\section*{Declarations}
\subsection*{Conflict of interest} The author declares that there is no conflict of interest.
\subsection*{Acknowledgments} The fundamental ideas and techniques underlying the results in this paper have been developed during my doctorate at the RWTH Aachen. I would like to thank my doctoral thesis advisors, Umberto Hryniewicz and Alberto Abbondandolo, for their the generous and patient supervision. I would like to thank the anonymous referee for the careful reading and helpful critique. I would like to thank Louis Merlin for making me aware of the theorem of Vinogradov on the equidistribution of prime multiples of irrational numbers mod 1. I would like to thank Urs Fuchs for the frequent and very useful conversations. I would like to thank Jungsoo Kang, Dustin Connery-Grigg and Li Meng for their important comments on previous versions of this work. I would like to thank Gabriele Benedetti for his essential contributions to the work environment of the Lehrstuhl MathGA.

\printbibliography

\vfill
\end{document}

\typeout{get arXiv to do 4 passes: Label(s) may have changed. Rerun}